\documentclass{amsart} \usepackage{bbm}
\usepackage[top=2.8cm,left=2.8cm,right=2.8cm,bottom=2.8cm]{geometry}
\usepackage{amsfonts}
\usepackage{amsmath,amsthm,amscd,mathrsfs,amssymb,graphicx,enumerate,
	dsfont}
\usepackage{xypic}

\newtheorem{thm2}{Theorem}
\newtheorem{conj2}{Conjecture}
\newtheorem{thm}{Theorem}[section]
\newtheorem{cor}[thm]{Corollary}
\newtheorem{lem}[thm]{Lemma}
\newtheorem{prop}[thm]{Proposition}

\theoremstyle{definition}
\newtheorem{defn}[thm]{Definition}

\newtheorem{rmk}[thm]{Remark}

 \DeclareMathOperator{\rk}{rk}

\newcommand{\C}{\ensuremath\mathds{C}}

\newcommand{\Q}{\ensuremath\mathds{Q}}
\newcommand{\F}{\ensuremath\mathrm{F}}

\newcommand{\PP}{\ensuremath\mathds{P}}
\newcommand{\calO}{\ensuremath\mathcal{O}}

\newcommand{\HH}{\ensuremath\mathrm{H}}
\newcommand{\CH}{\ensuremath\mathrm{CH}}

\newcommand{\set}[1]{\left\{#1\right\}}

\begin{document}

	\title{The motive of the Hilbert cube}
	\author{Mingmin Shen and Charles Vial}
	
	

	\thanks{The second author is supported by EPSRC Early Career
		Fellowship number EP/K005545/1.}
	
	\address{
		KdV Institute for Mathematics, University of Amsterdam, P.O.Box 94248, 1090 GE
		Amsterdam, Netherlands}
	\email{M.Shen@uva.nl}
	\address{
		DPMMS, University of Cambridge, Wilberforce Road, Cambridge
		CB3 0WB, UK}
	\email{C.Vial@dpmms.cam.ac.uk}
	
	
	\begin{abstract} 
		The Hilbert scheme $X^{[3]}$ of length-$3$ subschemes of a smooth projective
		variety $X$ is known to be smooth and projective. We investigate whether the
		property of having a multiplicative Chow--K\"unneth decomposition is stable
		under taking the Hilbert cube. This is achieved by considering an explicit
		resolution of the map $X^3 \dashrightarrow X^{[3]}$. The case of the Hilbert
		square was taken care of in \cite{sv}. The archetypical examples of varieties
		endowed with a multiplicative Chow--K\"unneth decomposition is given by abelian
		varieties. Recent work seems to suggest that hyperK\"ahler varieties share the
		same property. Roughly, if a smooth projective variety $X$ has a multiplicative
		Chow--K\"unneth decomposition, then the Chow rings of its powers $X^n$ have a
		filtration, which is the expected  Bloch--Beilinson filtration, that is split.
	\end{abstract}
	
	\maketitle

	Let $X$ be a smooth projective variety of dimension $d$, over a field
	$k$, which we assume is endowed with a \emph{Chow--K\"unneth
		decomposition} $\{\pi^i_X \in \CH^d(X \times X) : 0 \leq i \leq
	2d\}$. Informally, this means that the K\"unneth decomposition of the
	$\ell$-adic class ($\ell \neq \mathrm{char} \, k$) of the diagonal
	$[\Delta_X] \in \HH^{2d}(X_{\bar{k}},\Q_\ell)$ is algebraic and lifts
	to a splitting of the rational Chow motive of $X$. Precisely, there
	exist cycles $\pi_X^i \in \CH^d(X\times X)$ with rational coefficients
	such that $\Delta_X = \sum_{i=0}^{2d} \pi_X^i$ and which satisfy, when
	seen as self-correspondences of $X$, $\pi_X^i \circ \pi_X^i =
	\pi_X^i$, $\pi_X^i \circ \pi_X^j = 0$ for $i\neq j$, and
	$(\pi_X^i)_*\HH^*(X_{\bar{k}},\Q_\ell) =
	\HH^i(X_{\bar{k}},\Q_\ell)$. Such a decomposition induces a descending
	filtration on the Chow groups of $X$ by the formula
	$$\F^s\CH^i(X) := \Big(\sum_{l\leq 2i-s} \pi_X^l\Big)_*\CH^i(X).$$
	The conjectures of Murre \cite{murre} predict that
	
	(B) $\CH^i(X) =
	\F^0\CH^i(X)$ and $\F^s\CH^i(X) = 0$ when $s > i$\,; 
	
	(D)
	$\F^1\CH^i(X) = \mathrm{Ker}\, \{cl : \CH^i(X) \rightarrow
	\HH^{2i}(X_{\bar{k}},\Q_\ell)\}$\,; 
	
	(C) any two Chow--K\"unneth
	decompositions for $X$ induce the same filtration $\F^\bullet$ on
	$\CH^i(X)$. 
	
	\noindent By Jannsen (see \cite{jannsen} for precise
	statements), the filtration induced by a Chow--K\"unneth decomposition
	should be the filtration conjectured by Bloch and Beilinson and,
	conversely, the filtration conjectured by Bloch and Beilinson should
	be induced by any Chow--K\"unneth decomposition. In particular, this
	filtration should actually be a filtration on the \emph{Chow ring}
	$\CH^*(X)$ of $X$. Because of Jannsen's result, we will refer to a
	filtration on the Chow groups of $X$ induced by a Chow--K\"unneth
	decomposition as a filtration of expected Bloch--Beilinson type.
	
	Every smooth projective variety is conjectured by Murre \cite{murre}
	to be endowed with a Chow--K\"unneth decomposition. Examples of
	varieties for which the existence of a Chow--K\"unneth decomposition
	has been settled include curves, surfaces \cite{murre2} and abelian
	varieties \cite{beauville}. A natural question is can one show that,
	provided a Chow--K\"unneth decomposition for $X$, the induced
	filtration (which is expected to be the Bloch--Beilinson filtration) is a
	filtration on the Chow ring of
	$X$\,? That is can one show that the induced filtration on the Chow
	groups of $X$ is compatible with intersection product\,? Before moving
	on to a more specific question, let us mention that the answer in the
	cases listed above (curves, surfaces and abelian varieties) is yes.
	
	While it is expected that the Bloch--Beilinson filtration, when it
	exists, is a filtration on the Chow ring of $X$, one cannot expect
	this filtration on the Chow ring of $X$ to be split in general. Indeed, by the
	expected properties of the Bloch--Beilinson filtration, one expects the graded
	pieces $\operatorname{Gr}^0_\mathrm{F} \CH^i (X)$ to inject into
	$\HH^{2i}(X_{\bar k},\Q_\ell)$ under the cycle class map. In particular
	$\operatorname{Gr}^0_\mathrm{F} \CH^d (X)$ is one-dimensional, where $d=\dim X$.
	On the other hand, one also expects that $\F^0 \CH^1(X) = \CH^1(X)$ and $\F^1
	\CH^1(X) = 0$ when $\HH^{1}(X_{\bar k},\Q_\ell)=0$. Therefore, if $X$ is such
	that  $\HH^{1}(X_{\bar k},\Q_\ell)=0$ and if the conjectural Bloch--Beilinson
	filtration is split, then the image of the intersection product map $\CH^1(X)
	\otimes \ldots \otimes \CH^1(X) \to \CH^d(X)$ is one-dimensional.
	There are however
	some examples of simply connected surfaces $X$ 
	for which the rank of the image of the intersection product
	$\CH^1(X) \otimes \CH^1(X) \rightarrow \CH^2(X)$ is $\geq 2$.
	(Consider for instance the blow-up of a simply-connected surface $S$ with $\dim
	\CH_0(S) \geq 2$ at a point that is not rationally
	equivalent to any cycle in $\mathrm{Im} \,\{\CH^1(S) \otimes \CH^1(S)
	\rightarrow \CH^2(S)\}$\,; then the self-intersection of the exceptional divisor
	provides a new dimension which was not in the image of the intersection
	product.)
	There are nonetheless examples of varieties for which the filtration splits. For
	example, Beauville \cite{beauville} proved that the filtration splits for
	abelian varieties. Ten years ago, Beauville and Voisin \cite{bv}
	observed that the filtration also splits for K3 surfaces. Then
	Beauville \cite{beauville1}, after studying intersection of divisors
	on certain hyperK\"ahler varieties, asked whether the conjectural
	Bloch--Beilinson filtration, if
	it exists, would split for all hyperK\"ahler varieties. Here, a
	hyperK\"ahler variety refers to a projective irreducible holomorphic symplectic
	manifold. This question was answered affirmatively in \cite{sv} in the
	case when $X$ is the Hilbert scheme of length-$2$ subschemes on a K3
	surface or the variety of lines on a generic cubic $4$-fold, and in
	\cite{vial} in the case when $X$ is the Hilbert scheme of length-$n$
	subschemes on a K3 surface for any integer $n$.
	
	In fact, when $X$ is the Hilbert scheme $S^{[n]}$ of length-$n$
	subschemes on a K3 surface $S$, we established in \cite{sv, vial} a
	stronger statement. In order to motivate that statement, let us
	make the following observation. Recall that if $X$ and $Y$ are smooth
	projective varieties with a Chow--K\"unneth decomposition, then the
	product $X\times Y$ is naturally endowed with the product
	Chow--K\"unneth decomposition $\pi_{X\times Y}^k := \sum_{i+j=k}
	\pi_X^i \otimes \pi_Y^j$ (see \eqref{eq composition correspondences} for the
	definition of the tensor product of correspondences seen as morphisms of Chow
	motives). However, having a Chow--K\"unneth
	decomposition inducing a filtration on the Chow ring that is split is
	not stable under product. Indeed, any Chow--K\"unneth decomposition on
	a curve induces a filtration that is split (there is nothing to
	intersect on a curve) but in general the filtration induced on the
	product of two curves is not split. A nicer notion, that is stable
	under product, is that of \emph{multiplicative Chow--K\"unneth
		decomposition}. A Chow--K\"unneth decomposition $\{\pi_X^i : 0 \leq
	i \leq 2d\}$ is said to be multiplicative if
	\begin{equation} \label{eq mult} \pi_X^k \circ \Delta_{123} \circ
	(\pi_X^i \otimes \pi_X^j) = 0 \in \CH^{2d}(X\times X \times X) \quad
	\text{whenever} \ k \neq i+j.
	\end{equation} 
	Here, $\Delta_{123} \in \CH^{2d}(X\times X \times X)$ is the class of
	the small diagonal $\{(x,x,x) : x \in X\}$ seen as a correspondence
	from $X\times X$ to $X$. At this point it should be noted that the
	relations \eqref{eq mult} always hold modulo homological equivalence,
	and that if $\alpha$ and $\beta$ are cycles in $\CH^*(X)$, then
	$(\Delta_{123})_*(\alpha \times \beta) = \alpha \cdot \beta$.  If $X$
	admits a multiplicative Chow--K\"unneth decomposition, then the
	induced filtration on the Chow groups of $X$ is a filtration on the
	Chow ring of $X$ that is split. 
	Indeed, given a Chow--K\"unneth decomposition $\{\pi_X^i\}$ for $X$
	and writing (whether $\{\pi_X^i\}$ is multiplicative or
	not) 
	\begin{equation} \label{eq grading}
	\CH^i(X)_s := (\pi_X^{2i-s})_*\CH^i(X),
	\end{equation}
	the condition that $\{\pi_X^i\}$ is
	multiplicative implies that $\CH^*(X) = \bigoplus_{i,s}\CH^i(X)_s$ is
	a bigraded ring, that is, $\CH^i(X)_s \cdot \CH^j(X)_t \subseteq
	\CH^{i+j}(X)_{s+t}$.  As a matter of fact (\emph{cf.} \cite[Section
	8]{sv}), K3 surfaces and abelian varieties not only have a
	filtration of expected Bloch--Beilinson type that is split but also have a
	multiplicative Chow--K\"unneth decomposition by work of
	Beauville--Voisin \cite{bv} and Beauville \cite{beauville},
	respectively. In \cite{sv}, we proved, under the technical, but yet natural,
	assumption
	that the Chern classes $c_p(X)$ belong to the graded-$0$ part
	$\CH^p(X)_0$ of $\CH^p(X)$, that if $X$ is a smooth projective variety
	that admits a multiplicative Chow--K\"unneth decomposition
	(\emph{e.g.} $X$ a K3 surface), then the Hilbert scheme of length-$2$
	subschemes $X^{[2]}$ also admits a multiplicative Chow--K\"unneth
	decomposition.

	Given a smooth projective variety $X$, Cheah \cite{cheah} showed that the
	Hilbert
	scheme $X^{[3]}$ of length-$3$ subschemes of $X$ is smooth and
	projective. (Cheah also showed that if $\dim X \geq 3$, the Hilbert scheme
	$X^{[n]}$ is never smooth when $n\geq 4$.) In this manuscript, we want to push
	further the method of \cite[Section 13]{sv} to show that whenever $X$
	admits a multiplicative Chow--K\"unneth decomposition, then its
	Hilbert cube $X^{[3]}$ also admits a multiplicative Chow--K\"unneth
	decomposition. The idea is basic\,: we resolve the rational map $X^3
	\dashrightarrow X^{[3]}$ by successively blowing up subvarieties that
	are invariant under the action of the symmetric group $\mathfrak{S}_3$. We
	obtain a generically finite morphism denoted $p : X_3 \to X^{[3]}$, 
	and we check that the properties of the Chow--K\"unneth decomposition
	of $X^3$ induced by that of $X$ (self-duality, multiplicativity, Chern
	classes belonging to the graded-$0$ part of the Chow groups) lift to $X_3$. We
	then check, and this requires a careful analysis of
	the geometry of $X^{[3]}$, that the resulting multiplicative Chow--K\"unneth
	decomposition on $X_3$
	descends along the morphism $p$. The main result of
	this paper is the
	following.
	\begin{thm2}\label{thm main}
		Let $X$ be a smooth projective variety that admits a self-dual multiplicative
		Chow--K\"unneth decomposition (see Definitions \ref{def selfdual} \& \ref{def
			mult}). Assume that the Chern
		classes of $X$ satisfy $c_p(X) \in \CH^p(X)_0$. 
		Then the Hilbert
		cube $X^{[3]}$ also admits a self-dual multiplicative
		Chow--K\"unneth decomposition, with the property that the Chern classes
		$c_p(X^{[3]})$ sit in $\CH^p(X^{[3]})_0$ . 
		
		\noindent In particular, still assuming $c_p(X) \in \CH^p(X)_0$, the Chow ring
		$\CH^*(X^{[3]})$ has a filtration,
		which is the candidate Bloch--Beilinson filtration, that is split.
	\end{thm2}
	
	Here, a \emph{self-dual} Chow--K\"unneth decomposition means a
	Chow--K\"unneth decomposition $\{\pi_X^i, 0 \leq i \leq 2d\}$ such
	that $\pi_X^{2d-i}$ is the transpose of $\pi^i_X$ for all $i$. The
	self-duality assumption on $\{\pi_X^i\}$ is important because it
	ensures, together with the multiplicativity assumption, that the classes of the
	several diagonals in $X^3$ belong to
	$\CH^*(X^3)_0$, which is 
	crucial
	for
	checking that 
	the assumptions of
	Proposition \ref{prop multCK blow-up} are met for $X^3$ and its diagonals. 
	We also
	take the trouble of showing that a blow-up admits a \emph{self-dual}
	Chow--K\"unneth decomposition (Proposition \ref{prop self-dualCK}) essentially
	because we will have to blow up $X^3$ several times and at each step
	self-duality will be required. Ultimately we find that
	$X^{[3]}$ admits a
	multiplicative Chow--K\"unneth decomposition that is \emph{self-dual}
	and this makes it possible to iterate the process, \emph{e.g.} to show that
	$(X^{[3]})^{[3]}$ also has a multiplicative Chow--K\"unneth
	decomposition. 
	
	Beyond the case of Hilbert schemes of points, one may consider \emph{nested
		Hilbert schemes}. Given a projective variety $X$ and positive integers
	$n_1<\cdots < n_l$, the nested Hilbert scheme $X^{[n_1,\ldots, n_l]}$ is the
	projective scheme consisting of
	$\{(x_1,\ldots, x_l) : x_1 \subset \cdots \subset x_l \} \subset  X^{[n_1]}
	\times \cdots \times X^{[n_l]}$.
	Cheah \cite{cheah} showed that for a smooth projective variety $X$ of dimension
	$\geq 3$, the
	nested Hilbert scheme  $X^{[n_1,\ldots, n_l]}$ is smooth if and only if it is
	one of $X^{[1,2]}$ or $X^{[2,3]}$. In Theorem \ref{thm mainnested}, we establish
	the analogue of Theorem \ref{thm main} in those cases, by showing that
	$X^{[1,2]}$ or $X^{[2,3]}$ admit a
	self-dual multiplicative Chow--K\"unneth decomposition with Chern classes
	belonging to the graded-$0$ part of the Chow ring,
	whenever $X$ has a
	self-dual multiplicative Chow--K\"unneth decomposition with $c_p(X) \in
	\CH^p(X)_0$ for all $p \geq 0$.
	\medskip
	
	Together with \cite[Theorem 6]{sv}, \cite[Theorem
	1]{vial} and \cite[Theorem 1.7]{ftv}, we then obtain
	a large class of varieties admitting a multiplicative Chow--K\"unneth
	decomposition\,:
	\begin{thm2} \label{thm2 CK} Let $E$ be the smallest subset of smooth
		projective varieties that contains varieties with Chow groups of
		finite rank (as $\Q$-vector spaces), abelian varieties, generalized Kummer varieties, and Hilbert schemes of
		length-$n$ subschemes of hyperelliptic curves, K3 surfaces, or abelian surfaces,
		and that is
		stable under the following operations\,:
		\begin{enumerate}[(i)]
			\item if $X$ and $Y$ belong to $E$, then $X\times Y \in E$\,;
			\item if $X$ belongs to $E$, then $\PP(\mathscr{T}_X) \in E$, where
			$\mathscr{T}_X$ is the tangent bundle of $X$\,;
			\item if $X$ belongs to $E$, then the Hilbert scheme of length-$2$
			subschemes $X^{[2]}$ belongs to $E$\,;
			\item if $X$ belongs to $E$, then the Hilbert scheme of length-$3$
			subschemes $X^{[3]}$ belongs to $E$\,;
			\item  if $X$ belongs to $E$, then the nested Hilbert schemes  $X^{[1,2]}$ and
			$X^{[2,3]}$ belong to $E$.
		\end{enumerate}
		If $X$ be a smooth projective variety that is isomorphic to a
		variety in $E$, then $X$ admits a self-dual multiplicative Chow--K\"unneth
		decomposition.
	\end{thm2}
	
	When the ground field $k$ has characteristic zero, Riess \cite{riess} showed
	that birational hyperK\"ahler varieties have isomorphic Chow rings and
	isomorphic Chow motives as algebras in the category of Chow motives. Therefore,
	one may add to the set $E$ all hyperK\"ahler varieties that are birational to
	the Hilbert scheme of of length-$n$
	subschemes on a K3 surface.
	In fact, by the following conjecture which is directly inspired from Beauville's
	splitting conjecture \cite{beauville1} and which we stated in \cite{sv, vial},
	we expect the  set $E$ of Theorem \ref{thm2 CK} to contain all hyperK\"ahler
	varieties.
	
	\begin{conj2}[\cite{sv, vial}]
		Every hyperK\"ahler variety $X$ admits a self-dual multiplicative
		Chow--K\"unneth
		decomposition such that its Chern classes lie in $\CH^*(X)_0$.
	\end{conj2}

	As a corollary to Theorem \ref{thm2 CK}, one may state decomposition theorems
	for families of
	varieties that belong to the set $E$ described in Theorem \ref{thm2
		CK} as those first stated by Voisin \cite{voisin k3} for families of
	K3 surfaces. One may also consult \cite{vial} for the case of the
	relative Hilbert scheme of length-$n$ subschemes on a family of K3
	surfaces.\medskip

	\textbf{Outline.} We start in section 1 by showing that the
	property of having a \emph{self-dual} Chow--K\"unneth decomposition is stable
	under the following four operations\,: product, projective bundle, blow-up, and
	finite quotient by a group action. In section 2, we show under suitable
	conditions (mostly concerning Chern classes) that the property of having a
	self-dual \emph{multiplicative} Chow--K\"unneth decomposition is stable under
	the same four operations. Propositions \ref{prop projbundleselfdualmult} and
	\ref{prop multCK blow-up} complement \cite[Propositions 13.1 \& 13.2]{sv} where
	the self-duality assumption was overlooked. We then want to use these general
	results to show that a certain desingularization of the map $X^3 \dashrightarrow
	X^{[3]}$ has a self-dual multiplicative Chow--K\"unneth decomposition. Since one
	needs to blow up $X^3$ several times in order to resolve the map $X^3
	\dashrightarrow X^{[3]}$, it is convenient to state a proposition that takes
	care of the centers of the successive blow-ups all at once, so that we don't
	have to check the assumptions of Proposition \ref{prop multCK blow-up} at each
	step. This is the content of section 3 and the main result there is Proposition
	\ref{prop admissible blow-up}. In section 4, we resolve explicitly the rational
	map $X^3
	\dashrightarrow X^{[3]}$ into a morphism $p : X_3 \rightarrow X^{[3]}$ by
	blowing up $X^3$ along subvarieties that are stable under the action of the
	symmetric group $\mathfrak{S}_3$ which permutes the factors, and we show that
	$p$
	factors through the quotient morphism and then contracts two irreducible
	divisors (these contractions are not smooth blow-up morphisms).
	In section 5, we equip $X_3$ with a $\mathfrak{S}_3$-invariant
	self-dual multiplicative Chow--K\"unneth decomposition, show that it descends
	along the resolution $p:X_3 \rightarrow X^{[3]}$, and prove Theorem \ref{thm
		main}. A technical difficulty consists in showing that the correspondence
	${}^t\Gamma_p \circ \Gamma_p$ sits in $\CH^{3d}(X_3\times X_3)_0$, which in turn
	requires the careful analysis, carried out in section 4, of the morphism $X_3
	\rightarrow X^{[3]}$. One should note that for our purpose of constructing
	idempotents in the ring of correspondences, it is not sufficient to show for
	example that $p^*p_*$ preserves the grading on $X_3$ to conclude\,; it is really
	necessary to work at the level of correspondences, which yet again adds a level
	of technicality. Finally, in Section 6, we prove the analogue of Theorem
	\ref{thm main} for the nested Hilbert schemes $X^{[1,2]}$ and $X^{[2,3]}$.

	For the sake of simplicity, we have not considered the question of constructing 
	a Chow--K\"unneth decomposition for $X^{[3]}$ if $X$ is only assumed to be
	endowed with a Chow--K\"unneth decomposition (not necessarily self-dual or
	multiplicative). 
	The main reason is that while Chow--K\"unneth idempotents are central in the
	ring of correspondences
	modulo numerical equivalence, this is far from being the case modulo rational
	equivalence. In fact, a correspondence commutes with the
	Chow--K\"unneth projectors if and only if it sits in the graded-$0$ part of the
	ring of correspondences (see Lemma \ref{lem corr grade 0}).
	Therefore if one assumes the
	Chow--K\"unneth decomposition to be multiplicative, then one can keep track of
	intersections and compositions of correspondences that sit in grade $0$, thus
	making the arguments simpler. Nevertheless, 
	our main theorem adds to the restricted class of varieties that can be endowed
	with a multiplicative Chow--K\"unneth decomposition. That such a property is
	stable by the not-so-simple operation of taking the Hilbert cube is intriguing. 
	\medskip
	
	\textbf{Conventions.}  Chow groups are always understood with rational
	coefficients. The Chow motive of a smooth projective variety $X$ is denoted
	$\mathfrak{h}(X)$. Following Voevodsky, we think of the theory of motives as a
	homology theory. Consequently, the functor $\mathfrak{h}$ is seen as a covariant
	functor, and correspondences act covariantly. Moreover, denoting $\mathds{1} :=
	\mathfrak{h}(\operatorname{Spec}\, k)$, we set the Tate twists to follow the
	convention that $\mathfrak{h}(\PP^1)= \mathds{1} \oplus \mathds{1}(1)$.\medskip
	
	\textbf{Acknowledgments.} We thank Kieran O'Grady and John Ottem for useful
	discussions concerning the map $X^3 \dashrightarrow X^{[3]}$. We are grateful to
	the referees for comments that have helped improve the exposition. Finally, we
	would like to thank one of the referees and
	Manfred Lehn for suggesting treating the cases of the nested Hilbert schemes
	$X^{[1,2]}$ and $X^{[2,3]}$.

	\section{Self-dual Chow--K\"unneth decompositions}
	\label{sec CK}
	
	\begin{defn}\label{def selfdual}
		Let $X$ be a smooth projective variety of dimension $d$ over a field $k$,
		endowed with a Chow--K\"unneth decomposition $\{\pi_X^i : 0 \leq i \leq 2d\}$. 
		The Chow--K\"unneth decomposition $\{\pi_X^i : 0 \leq i \leq 2d\}$ is said to be
		\emph{self-dual} if 
		$${}^t\pi^i = \pi^{2d-i} \ \in \CH^d(X\times X) \quad \mbox{for all } 0\leq i
		\leq 2d.$$
		Here, given a correspondence $\Gamma \in \CH^*(X\times X)$, the \emph{transpose}
		denoted ${}^t\Gamma$ is the correspondence obtained from $\Gamma$ by switching
		the two factors of $X\times X$.
	\end{defn}
	
	It is well-known that operations on varieties such as taking products,
	projective bundles and blow-ups preserve the property of having a
	Chow--K\"unneth decomposition.  In addition  to recalling these results (and to
	setting up notations along the way), we also give a sufficient condition for a
	Chow--K\"unneth decomposition to descend along a generically finite morphism.
	The principal goal of this section is to show that the property that the
	Chow--K\"unneth decompositions are \emph{self-dual} is also preserved under
	these
	operations as well as under the operation of taking a generically finite
	quotient.
	
	\subsection{Product varieties} 
	Recall that, given a smooth projective variety $X$ endowed with a
	Chow--K\"unneth decomposition $\{\pi_X^i\}$, we have defined in \eqref{eq
		grading} the graded pieces
	$$\CH^i(X)_s := (\pi_X^{2i-s})_*\CH^i(X).$$
	Let $X$ and $Y$ be smooth projective varieties
	endowed with Chow--K\"unneth decompositions  $\{\pi_X^i\}$ and $\{\pi_Y^i\}$,
	respectively.
	Thanks to the K\"unneth formula, the idempotents
	\begin{equation} \label{eq productCK}
	\pi_{X\times Y}^i := \sum_{i=i_1+i_2} \pi_{X}^{i_1} \otimes \pi_{Y}^{i_2}
	\end{equation} define a Chow--K\"unneth decomposition for the product variety $X
	\times Y$.  Here and throughout this work, given any two correspondences $\alpha
	\in \CH^*(X_1\times X_2)$ and $\beta \in \CH^*(Y_1\times Y_2)$, we define the
	correspondence 
	\begin{equation}\label{eq composition correspondences}
	\alpha \otimes \beta := p_{13}^*\alpha \cdot p_{24}^*\beta
	\quad \in \CH^*((X_1\times Y_1)\times (X_2\times Y_2)),
	\end{equation}
	where $p_{ij}$ is the
	projector from $X_1 \times Y_1 \times X_2 \times Y_2$ on the product of the
	$i^\mathrm{th}$ and $j^\mathrm{th}$ factors. 
	With the above product
	Chow--K\"unneth decomposition, we have
	\begin{equation}\label{eq product pull back}
	p_1^*\CH^p(X)_s\cdot p_2^*\CH^q(Y)_t \subseteq \CH^{p+q}(X\times Y)_{s+t}\,;
	\end{equation}
	see \cite[Proposition 8.7]{sv}.
	
	\begin{defn}\label{defn corr grade}
		A correspondence $\Gamma\in\CH^p(X\times Y)$ is said to be of \textit{pure
			grade} $s$ if 
		\[
		\Gamma\in \CH^p(X\times Y)_s.
		\]
		In particular, a morphism $f:X\rightarrow Y$ is of pure grade 0 if its graph is
		in $\CH^{d_Y}(X\times Y)_0$, where $d_Y=\dim Y$.
	\end{defn}
	
	\begin{rmk}\label{rmk prod CK}
		The correspondence $\Gamma\in \CH^p(X\times Y)$ is of pure grade $s$  if and
		only if
		\[
		(\pi_X^i\otimes \pi_Y^j)_*\Gamma = 0 \quad \mbox{for all pairs} \ i+j\neq 2p-s.
		\]
		Indeed, $\Gamma\in \CH^p(X\times Y)$ is of pure grade $s$ if and only
		if $(\pi_{X\times Y}^k)_* \Gamma = 0$ for any $k\neq 2p-s$, and then use
		definition \eqref{eq productCK}.
	\end{rmk}
	
	Now we turn to self-dual Chow--K\"unneth decompositions. The following
	proposition is clear\,:
	\begin{prop} \label{prop self-dualproduct}
		Assume that $X$ and $Y$ are endowed with the action of a finite group $G$.
		If $\{\pi_X^i\}$ and $\{\pi_Y^i\}$ are $G$-invariant and self-dual, then
		$\{\pi_{X\times Y}^i\}$ is $G$-invariant and self-dual.\qed
	\end{prop}
	
	The following lemma is a criterion for a correspondence to be of pure grade.
	
	\begin{lem}\label{lem corr grade 0}
		Let $X$ and $Y$ be smooth projective varieties endowed with self-dual
		Chow--K\"unneth
		decompositions $\{\pi^i_X : 0\leq i \leq 2d_X\}$  and $\{\pi_Y^j : 0\leq j\leq
		2d_Y\}$, where $d_X=\dim X$ and $d_Y=\dim Y$. Let
		$\Gamma\in\CH^p(X\times Y)$ be a correspondence. Then $\Gamma$ is of pure
		grade $s$, namely $\Gamma\in\CH^p(X\times Y)_s$, if and only if it satisfies
		\[
		\pi_Y^i \circ \Gamma = \Gamma\circ \pi_X^{2(d_X-p)+s+i}, \qquad\text{for all
		}i.
		\]
		In particular, a self-correspondence $\Gamma\in\CH^{d_X}(X\times X)$ is of pure
		grade 0 if and only if it commutes with the
		Chow--K\"unneth projectors $\pi_X^i$, namely
		\[
		\pi^i_X\circ \Gamma = \Gamma\circ \pi^i_X, \quad \mbox{for all } 0\leq i \leq
		2d_X.
		\]
	\end{lem}
	
	\begin{proof}
		By Remark \ref{rmk prod CK},
		if the correspondence $\Gamma$ lies in $\CH^d(X\times Y)_s$
		then
		\begin{align*}
		\pi_Y^i\circ \Gamma\circ \pi^{2(d_X-p)+s+i}_X &=
		({}^t\pi^{2(d_X-p)+s+i}_X\otimes\pi^i_Y)_*\Gamma \\
		& =(\pi^{2p-s-i}_X\otimes \pi^i_Y)_*\Gamma\\
		& = \left( (\sum_{j=0}^{2d_X}\pi_X^j)\otimes \pi_Y^i\right)_*\Gamma\\
		& = (\Delta_X\otimes \pi^i_Y)_*\Gamma\\
		& = \pi^i_Y\circ \Gamma,
		\end{align*}
		for all $0\leq i \leq 2d_Y$. Note that here we used the fact that $\{\pi^i_X\}$
		is
		self-dual in an essential way. A similar computation yields
		\begin{equation*}
		\pi^i_Y\circ \Gamma\circ \pi_X^{2(d_X-p)+s+i}  = \left(\pi^{2p-s-i}_X\otimes
		(\sum_{j=0}^{2d_Y}\pi_Y^j)\right)_*\Gamma = (\pi^{2p-s-i}_X\otimes
		\Delta_Y)_*\Gamma
		=\Gamma\circ \pi^{2(d_X-p)+s+i}_X.
		\end{equation*}
		Thus $\pi_Y^i\circ \Gamma =\Gamma\circ \pi_X^{2(d_X-p)+s+i}$. Conversely, if the
		above equality holds, then
		\[
		(\pi^{2p-s-j}_X\otimes \pi_Y^i)_*\Gamma = \pi^i_Y\circ \Gamma \circ
		\pi^{2(d_X-p)+s+j}_X =
		\pi^i_Y\circ \pi_Y^j\circ \Gamma =0
		\]
		for all $i\neq j$. By Remark \ref{rmk prod CK}, $\Gamma$ lies in $\CH^p(X\times
		Y)_s$.
	\end{proof}
	
	The action of a correspondence of pure grade shifts the grading.
	\begin{prop}\label{prop action of correspondence}
		Let $X$, $Y$, $X'$ and $Y'$ be smooth projective varieties with self-dual 
		Chow--K\"unneth decomposition. Then the following are true.
		\begin{enumerate}[(i)]
			\item Let $Z$ be a smooth projective variety with a Chow--K\"unneth
			decomposition $\{\pi_Z^i\}$ which is not necessarily multiplicative or
			self-dual. Let $\Gamma\in\CH^p(X\times Y)$ be of pure grade $s$, i.e.
			$\Gamma\in\CH^p(X\times Y)_s$, then
			\[
			( \Gamma\otimes \mathrm{id})_*:\CH^l(X\times Z)_r \rightarrow
			\CH^{l+p-d_X}(Y\times Z)_{r+s}.
			\]
			In particular, the pull-back and push-forward by a morphism of pure grade 0
			preserves the gradings on the Chow groups.
			\item If $\Gamma\in\CH^*(X\times Y)$ is of pure grade $s$ and and
			$\Gamma'\in\CH^*(Y\times Y')$ is of pure grade $s'$, then
			$\Gamma'\circ\Gamma\in\CH^*(X\times Y')$ is of pure grade $s+s'$.
			\item If $\Gamma\in\CH^*(X\times Y)$ and $\Gamma'\in \CH^*(X'\times Y')$ are of
			pure grades $s$ and $s'$ respectively, then $\Gamma\otimes \Gamma'\in\CH(X\times
			X'\times Y\times Y')$ is of pure grade $s+s'$.
		\end{enumerate}
	\end{prop}
	
	\begin{proof}
		To prove (\textit{i}), we first apply Lemma \ref{lem corr grade 0} and get
		$\pi_Y^i\circ \Gamma = \Gamma\circ \pi_X^{2(d_X-p)+s+i}$. After tensoring with
		the identity correspondence of $Z$ we get
		\[
		(\pi_Y^i\otimes \pi_Z^j) \circ (\Gamma\otimes \mathrm{id}_Z) = (\Gamma \otimes
		\mathrm{id}_Z) \circ (\pi_X^{2(d_X-p)+s+i}\otimes \pi_Z^j).
		\]
		If $\alpha \in \CH^l(X\times Z)_r$, then the above identity implies
		\begin{equation*}
		(\pi_Y^i\otimes \pi_Z^j)_*(\Gamma\otimes \mathrm{id}_Z)_* \alpha  =
		(\Gamma\otimes \mathrm{id}_Z)_* (\pi_X^{2(d_X-p)+s+i}\otimes \pi_Z^j)_* \alpha =
		(\Gamma\otimes\mathrm{id}_Z)_*0 = 0,
		\end{equation*}
		for all $2(d_X-p)+s+i +j \neq 2l-r$. This precisely means that
		\[
		(\Gamma\otimes \mathrm{id}_Z)_*\alpha \in \CH^{l+p-d_X}(Y\times Z)_{r+s}.
		\]
		Statements (\textit{ii}) and (\textit{iii}) are proved similarly using Lemma
		\ref{lem corr grade 0}.
	\end{proof}
	
	Note that a Chow--K\"unneth decomposition being self-dual implies that the
	diagonal is of pure grade 0. One consequence of this fact is the following
	proposition.
	\begin{cor}\label{cor projections grade 0}
		Let $X$ and $Y$ be two smooth projective varieties endowed with self-dual
		Chow--K\"unneth decompositions. Then the two natural projections $p_1:X\times
		Y\rightarrow X$ and $p_2: X\times Y\rightarrow Y$ are of pure grade 0.
		Furthermore, 
		\[
		p_{1_*} \CH^p(X\times Y)_s\subseteq \CH^{p-d_Y}(X)_s,\qquad
		p_{2_*}\CH^p(X\times Y)_s\subseteq \CH^{p-d_X}(Y)_s.
		\]
	\end{cor}
	\begin{proof}
		Note that $\Gamma_{p_1}=p_{13}^*\Delta_X\in\CH^{d_X}(X\times Y\times X)$, where
		$p_{13}:X\times Y\times X\rightarrow X\times X$ is the projection onto the
		product of the first and the third factors. By \eqref{eq product pull back} and
		the fact that $\Delta_X\in\CH^{d_X}(X\times X)_0$, we conclude that
		$\Gamma_{p_1}\in\CH^{d_X}(X\times Y\times X)_0$, namely $p_1$ is of pure grade
		0. One similarly shows that $p_2$ is also of pure grade 0. The action on Chow
		groups follows immediately from Proposition \ref{prop action of correspondence}.
	\end{proof}

	\subsection{Projective bundles} \label{sec selfdualprojbundle} 
	Let $X$ be a
	smooth projective variety of dimension $d$ and let $\mathscr{E}$ be a vector
	bundle on $X$ of rank $r+1$. Denote $\pi : \PP(\mathscr{E}) \rightarrow X$ the
	geometric projectivization of $\mathscr{E}$. We define
	\[
	\gamma_0 := c_r(\pi^*\mathscr{E}/\calO(-1)) = \sum_{i=0}^r
	\pi^*c_i(\mathscr{E})\xi^{r-i},
	\]
	where $\xi\in\CH^1(\PP(\mathscr{E}))$ is the first Chern class of $\calO(1)$\,;
	then
	$\gamma_0$ satisfies $\xi\cdot \gamma_0 = -\pi^*c_{r+1}(\mathscr{E})$. Consider
	the correspondence
	$$\gamma := \iota_*\gamma_0\in\CH^{2r+d}(\PP(\mathscr{E})\times
	\PP(\mathscr{E})),$$  where $\iota: \PP(\mathscr{E})
	\hookrightarrow \PP(\mathscr{E})\times \PP(\mathscr{E})$ is the diagonal
	embedding of $\PP(\mathscr{E})$\,; its  action consists in intersecting with
	$\gamma_0$. It has been known since Manin
	\cite{manin} that the Chow motive of $\PP(\mathscr{E})$ is isomorphic to
	$\bigoplus_{l=0}^{r} \mathfrak{h}(X)(l)$. In fact an isomorphism is given by the
	correspondence  (see the proof of Proposition \ref{prop self-dual projbundle})
	$$\Phi :=
	\left(\bigoplus_{l=0}^{r-1} h^{l}
	\circ {}^t\Gamma_{\pi}\right) \oplus \gamma\circ {}^t\Gamma_{\pi}\ : \ 
	\left(\bigoplus_{l=0}^{r-1} \mathfrak{h}(X)(l)\right)\oplus \mathfrak{h}(X)(r)
	\longrightarrow
	\mathfrak{h}(\PP(\mathscr{E})),$$ 
	which induces an isomorphism\,:
	\begin{equation*}
	\Phi_* :   \bigoplus_{l=0}^{r}
	\CH^{p-l}(X) \overset{\cong}{\longrightarrow}
	\CH^p(\PP(\mathscr{E})), 
	\quad (\beta_0,\ldots,\beta_{r})\mapsto   \sum_{l=0}^{r-1} \xi^l\cdot
	\pi^*\beta_l + \gamma_0\cdot\pi^*\beta_r.
	\end{equation*}
	Here, for a morphism of varieties $g : V \rightarrow W$ we denote
	$\Gamma_g \in \CH_{\dim V}(V\times W)$ the class of the graph $\{(x,g(x)) : x
	\in V\}$, and $h^l \in \CH^{d+r+l}(\PP(\mathscr{E}) \times
	\PP(\mathscr{E}))$ the push-forward of $\xi^l\in \CH^{l}(\PP(\mathscr{E}))$
	under the diagonal embedding
	$\iota$ (so that $h^l$ is the $l^\mathrm{th}$ power of $h$ as a correspondence,
	rather than as an algebraic cycle). This isomorphism is not quite the usual
	projective bundle formula isomorphism $\bigoplus_{l=0}^{r}
	\CH^{p-l}(X) \overset{\cong}{\longrightarrow}
	\CH^p(\PP(\mathscr{E})), 
	(\beta_0,\ldots,\beta_{r})\mapsto   \sum_{l=0}^{r} \xi^l\cdot
	\pi^*\beta_l$. The usual projective bundle isomorphism suffices to get all
	results of this section. The reason for modifying the last summand is to make it
	compatible with the corresponding isomorphism for a smooth blow-up; see Section
	\ref{sec selfdualblowup}. This compatibility is needed in
	Remark \ref{rmk blow-up diagram degree 0} to get a commutative diagram. 
	
	Of course, if $X$ has a Chow--K\"unneth decomposition $\{\pi^i_X\}$, then the
	isomorphism $\Phi$ induces a Chow--K\"unneth decomposition for
	$\PP(\mathscr{E})$, namely  \begin{equation} \label{eq projbundleCK}
	p^i_{\PP({\mathscr{E}})} := \Phi
	\circ \left(\bigoplus_{j=0}^{r} \pi_X^{i-2j} \right)
	\circ \Phi^{-1}.
	\end{equation} However, a variety can be endowed with many different
	Chow--K\"unneth
	decompositions in general and Manin's isomorphism may not preserve certain
	properties of Chow--K\"unneth decompositions\,: if  $\{\pi^i_X\}$ is
	self-dual,
	then $\{p^i_{\PP({\mathscr{E}})}\}$ may not be self-dual. 
	The following proposition shows that if $X$ has a \emph{self-dual}
	Chow--K\"unneth decomposition, then  $\PP(\mathscr{E})$ has a \emph{self-dual}
	Chow--K\"unneth decomposition as well.

	\begin{prop} \label{prop self-dual projbundle} Let $X$ be a smooth projective
		variety endowed with the action of a finite group $G$. Assume that $X$ has a 
		$G$-invariant self-dual Chow--K\"unneth
		decomposition  $\{\pi_X^i\}$. If $\mathscr{E} \rightarrow X$ is a
		$G$-equivariant vector bundle on $X$ of rank $r+1$, then the geometric
		projectivization $\PP(\mathscr{E})$ has a  $G$-invariant self-dual
		Chow--K\"unneth
		decomposition $\{\pi_{\PP(\mathscr{E})}^i\}$. 
	\end{prop}
	\begin{proof}
		The correspondence ${}^t\Phi$ defines a morphism of motives 
		$${}^t\Phi : \mathfrak{h}(\PP(\mathscr{E})) \longrightarrow 
		\mathfrak{h}(X)(r)\oplus \mathfrak{h}(X)(r-1) \oplus\cdots\oplus
		\mathfrak{h}(X).$$
		Let
		$$\sigma : \mathfrak{h}(X)(r)\oplus \mathfrak{h}(X)(r-1)  \oplus \cdots\oplus
		\mathfrak{h}(X) \longrightarrow \mathfrak{h}(X)\oplus \mathfrak{h}(X)(1) \oplus
		\cdots \oplus \mathfrak{h}(X)(r)$$ 
		be the morphism that permutes the direct summands. This morphism is self-dual in
		the sense that
		$${}^t \sigma = \sigma.$$
		The correspondence $\Gamma_\pi \circ h^{l}\circ {}^t\Gamma_\pi$ vanishes for
		$l<r$ and is equal to $\Delta_X$ for $l=r$ (see for instance \cite[Section
		1]{vialfib}). The identity $\xi\cdot\gamma_0=-\pi^*c_{r+1}(\mathscr{E})$ implies
		that the correspondence $\Gamma_{\pi}\circ h^l\circ \gamma\circ
		{}^t\Gamma_{\pi}$ vanishes for $1\leq l\leq r$ and equals $\Delta_X$ for $l=0$.
		Hence
		\begin{align*}
		\Gamma_{\pi}\circ \gamma\circ \gamma\circ {}^t\Gamma_{\pi} & = \sum_{i=0}^r
		\Gamma_{\pi} \circ (\cdot \xi^i \pi^*c_{r-i}(\mathscr{E}))\circ \gamma\circ
		{}^t\Gamma_{\pi}\\
		& = \sum_{i=0}^r \Gamma_{\pi}\circ h^i\circ \gamma \circ {}^t\Gamma_{\pi}\circ
		(\cdot c_{r-i}(\mathscr{E}))\\
		& = \cdot c_r(\mathscr{E}).
		\end{align*}
		Here, we used the following notation. If $\alpha$ is a cycle in $\CH^p(Y)$, then
		$\cdot \alpha$ is the correspondence defined by
		$(\iota_Y)_*\alpha\in\CH^{p+d_Y}(Y\times Y)$ where $\iota_Y:Y\hookrightarrow
		Y\times Y$ is the diagonal embedding. Then it is easy to see that we have the
		identity 
		\begin{equation}\label{eq matrix 1}
		\sigma \circ {}^t\Phi \circ \Phi = \begin{pmatrix}
		\Delta_X & 0 &\cdots &0 &0\\
		0 &\Delta_X & \cdots &0 &0\\
		0 &\ast &\ddots &0 &0\\
		0 &\ast &\cdots &\Delta_X &0\\
		\cdot c_r(\mathscr{E}) &0 &\cdots &0 &\Delta_X
		\end{pmatrix}
		\end{equation}
		In other words, we can write
		$$\sigma \circ {}^t\Phi \circ \Phi =\mathrm{id} + \eta$$
		for a correspondence $\eta$ that is nilpotent of index $r$, that is, satisfies 
		$\eta^{\circ r} = 0$. By taking the transpose of the above equation, we get
		${}^t\Phi\circ\Phi\circ \sigma = \mathrm{id} + {}^t\eta$, and hence $\eta \circ
		\sigma = \sigma \circ {}^t\eta$. ($\eta$ should be thought of as a strict lower
		triangular matrix and $\sigma$ as being the operation of rotating the matrix by
		an angle of $\pi/2$.) We define the correspondence $$\Psi := \Phi
		\circ (\mathrm{id}+\eta)^{-\frac{1}{2}}.$$ Here, for a nilpotent element $x$ of
		order $r$ and for any real number
		$a$, we formally set $$(1+x)^a = 1+ax+ \frac{a(a-1)}{2}x^2 + \cdots +
		\frac{a(a-1)\cdots (a-r+2)}{(r-1)!}x^{r-1}.$$
		The correspondence $\Psi : \bigoplus_{l=0}^{r} \mathfrak{h}(X)(l)
		\longrightarrow
		\mathfrak{h}(\PP(\mathscr{E}))$ is an isomorphism\,; its inverse is
		$(\mathrm{id}+\eta)^{\frac{1}{2}} \circ \Phi^{-1}$. In fact we have $$\sigma
		\circ {}^t\Psi =
		\Psi^{-1},$$ because 
		$$\sigma \circ {}^t\Psi \circ \Psi
		= \sigma\circ (\mathrm{id} + {}^t\eta)^{-\frac{1}{2}}
		\circ
		{}^t \Phi \circ \Phi \circ (\mathrm{id} + \eta)^{-\frac{1}{2}} 
		= (\mathrm{id} + \eta)^{-\frac{1}{2}}
		\circ \sigma \circ
		{}^t \Phi \circ \Phi \circ (\mathrm{id} + \eta)^{-\frac{1}{2}} 
		= (\mathrm{id} + \eta)^{-\frac{1}{2}}\circ  (\mathrm{id} +
		\eta) \circ (\mathrm{id} +
		\eta)^{\frac{1}{2}} = \mathrm{id}.$$ It follows that 
		\begin{equation} \label{eq projbundleselfdualCK} \pi^i_{\PP({\mathscr{E}})} :=
		\Psi
		\circ \left(\bigoplus_{j=0}^{r} \pi_X^{i-2j} \right)
		\circ \Psi^{-1}
		\end{equation}
		defines a Chow--K\"unneth decomposition of $\PP(\mathscr{E})$ that is
		self-dual if $\{\pi_X^i\}$ is self-dual.
		Indeed, on the one hand, we have $${}^t\pi^i_{\PP({\mathscr{E}})} = \Psi \circ
		\sigma \circ (\pi_X^{2d-i} \oplus \ldots \oplus \pi_X^{2d+2r-i})\circ
		{}^t\Psi,$$ and on the other hand, we have $$\pi^{2d+2r-i}_{\PP({\mathscr{E}})}
		= \Psi \circ (\pi_X^{2d+2r-i} \oplus \ldots \oplus \pi_X^{2d-i}) \circ \sigma
		\circ {}^t\Psi,$$ and clearly $ \sigma \circ (\pi_X^{2d-i} \oplus \ldots \oplus
		\pi_X^{2d+2r-i}) = (\pi_X^{2d+2r-i} \oplus \ldots \oplus \pi_X^{2d-i}) \circ
		\sigma$.
		
		The Chow--K\"unneth decomposition \eqref{eq projbundleselfdualCK} is
		$G$-invariant if $\{\pi_X^i\}$
		is $G$-invariant because  $\Gamma_\pi$ is $G$-invariant (by assumption) and
		$h$ is $G$-invariant (the automorphism group of projective space preserves
		$\mathcal{O}(1)$). 
	\end{proof}
	
	\begin{rmk}\label{rmk projselfdual commutes with prod}
		The construction of $\pi_{\PP(\mathscr{E})}^i$ commutes with products. To be
		more precise, let $Y$ be any smooth projective variety with self-dual
		Chow--K\"unneth decomposition $\{\pi_Y^j\}$. The self-dual Chow--K\"unneth
		decomposition of $\PP(\mathscr{E})\times Y$, viewed as a projective bundle over
		$X\times Y$, agrees with the product Chow--K\"unneth decomposition.
	\end{rmk}

	\subsection{Blow-ups} \label{sec selfdualblowup}
	Let $X$ be a smooth projective variety and let $Y$ be smooth closed
	subvariety of codimension $r+1$ of $X$. It has been known since Manin
	\cite{manin} that the motive of the blow-up $\tilde{X}$ of $X$ along
	$Y$ can be expressed as $$\mathfrak{h}{(\tilde{X})} \cong
	\mathfrak{h}(X) \oplus \mathfrak{h}(Y)(1) \oplus \cdots \oplus
	\mathfrak{h}(Y)(r).$$ Consequently, if $X$ and $Y$ have a
	Chow--K\"unneth decomposition then so does $\tilde{X}$.
	\medskip
	
	The set-up for the proposition below is the following\,: $X$ is a
	smooth projective variety endowed with the action of a finite group
	$G$ and $i:Y\hookrightarrow X$ is a smooth closed subvariety of
	codimension $r+1$ such that $g\cdot Y=Y$ for all $g \in G$. The
	blow-up of $X$ along $Y$ is denoted $\tilde{X}$. We have the blow-up
	diagram
	\begin{equation}\label{eq blow-up diagram}
	\xymatrix{
		E\ar[r]^j\ar[d]_\pi & \tilde{X}\ar[d]^\rho\\
		Y\ar[r]^i & X
	}
	\end{equation}
	where $E\cong\PP(\mathscr{N}_{Y/X})$ is the exceptional divisor.

	\begin{prop} \label{prop self-dualCK} In the situation above, if $X$
		and $Y$ both have a  $G$-invariant self-dual Chow--K\"unneth
		decomposition, namely $\{\pi_X^i\}$ and $\{\pi_Y^i\}$, then
		$\tilde{X}$ has a $G$-invariant self-dual Chow--K\"unneth
		decomposition $\{\pi_{\tilde{X}}^i\}$.
	\end{prop}
	\begin{proof}
		First note that by functoriality of blow-ups the action of $G$ on
		$X$ lifts to an action on $\tilde{X}$, and all maps involved in the
		diagram \eqref{eq blow-up diagram} are $G$-equivariant. Therefore,
		the corresponding push-forwards and pull-backs on Chow groups are
		also $G$-equivariant.
		
		By Manin \cite{manin}, the correspondence $$\Phi :=
		{}^t\Gamma_\rho \oplus \bigoplus_{l=1}^{r} \Gamma_{j} \circ h^{l-1}
		\circ {}^t\Gamma_{\pi} \ : \ \mathfrak{h}(X) \oplus
		\bigoplus_{l=1}^{r} \mathfrak{h}(Y)(l) \longrightarrow
		\mathfrak{h}(\tilde{X})$$ is an isomorphism of Chow motives. Here,
		$h := - {}^t\Gamma_j \circ \Gamma_j$\,; its action on $\CH_*(E)$
		consists in intersecting with the first Chern class of the relative
		$\mathcal{O}(1)$-bundle $\mathcal{O}_{\tilde{X}}(-E)|_E$. In fact,
		that $\Phi $ is an isomorphism is a consequence of Manin's identity
		principle coupled with the fact that $\Phi$ induces an isomorphism
		of Chow groups\,: the blow-up formula for Chow groups
		\cite[Proposition 6.7(e)]{fulton}
		\begin{equation}\label{eq Chow of blow-up}
		\Phi_* :  \CH^p(X) \oplus\left( \bigoplus_{l=1}^{r}
		\CH^{p-l}(Y)\right) \overset{\cong}{\longrightarrow}
		\CH^p(\tilde{X}), 
		\quad (\alpha, \beta_1,\ldots,\beta_{r})\mapsto \rho^*\alpha 
		+ j_*\left( \sum_{l=1}^{r} \xi^{l-1}\cdot \pi^*\beta_l \right),
		\end{equation}
		where $\xi := c_1(\mathcal{O}_{\tilde{X}}(-E)|_E) \in \CH^1(E)$. A
		Chow--K\"unneth decomposition for $\tilde{X}$ is then given by
		\begin{equation} \label{eq blowupCK} p^i_{\tilde{X}} := \Phi \circ
		\left( \pi_X^i \oplus \bigoplus_{j=1}^{r} \pi_Y^{i-2j} \right) \circ
		\Phi^{-1}.
		\end{equation}
		(The correspondence $\Phi^{-1}$ is described explicitly in
		\cite[Section 5]{vialfib}.) Since the diagram \eqref{eq blow-up
			diagram} is $G$-equivariant, it is apparent that $\Phi$ is
		$G$-invariant. By assumption, $\{\pi_X^i\}$ and $\{\pi_Y^i\}$ are
		$G$-invariant. It follows that $\{p_{\tilde{X}}^i\}$ is
		$G$-invariant.\medskip
		
		Let us now assume that $\{\pi_X^i\}$ and $\{\pi_Y^i\}$ are both
		self-dual. The Chow--K\"unneth decomposition $\{p_{\tilde{X}}^i\}$
		constructed above is not self-dual in general. Here is a way to make
		it self-dual while preserving its $G$-invariance. Let
		$$ \sigma : \mathfrak{h}(X) \oplus \mathfrak{h}(Y)(r) \oplus \cdots \oplus
		\mathfrak{h}(Y)(1) \longrightarrow \mathfrak{h}(X) \oplus \mathfrak{h}(Y)(1)
		\oplus \cdots \oplus \mathfrak{h}(Y)(r)$$ be the morphism that switches the
		summands. By \cite[Lemma 5.2]{vialfib}, the correspondence $\sigma \circ
		{}^t\Phi \circ \Phi$ from
		$\mathfrak{h}(X) \oplus \bigoplus_{i=1}^{r} \mathfrak{h}(Y)(i)$ to
		itself can be written in matrix form as a lower triangular
		matrix\,: 
		\begin{equation} \label{eq tphiphi}
		\sigma \circ {}^t\Phi \circ \Phi = \left( \begin{array}{ccccc}
		\Delta_X & 0 & 0 & \cdots & 0 \\
		0 & -\Delta_Y & 0 & \cdots & 0 \\
		0 & * & \ddots & \ddots & \vdots \\
		\vdots & \vdots & \ddots & \ddots & 0 \\ 0 & * & \cdots & * &
		-\Delta_Y
		\end{array} \right).
		\end{equation}
		Let us write $\sigma \circ {}^t\Phi \circ \Phi = D\circ (\mathrm{id} + \eta)$,
		where $D$ is the diagonal matrix with first diagonal entry $\Delta_X$
		and remaining diagonal entries $-\Delta_Y$. Note that $\sigma = {}^t \sigma$ and
		that $\sigma$ commutes with $D$. The correspondence $\eta$
		is clearly nilpotent of index $r$ and commutes with $D$\,; we also have $\eta
		\circ \sigma = \sigma \circ {}^t\eta$. We
		define $$\Psi := \Phi \circ (\mathrm{id} + \eta)^{-\frac{1}{2}}.$$
		Clearly, $\Psi$ is an
		isomorphism with inverse $(\mathrm{id}+\eta)^{\frac{1}{2}} \circ
		\Phi^{-1}$. We claim that in fact $$D\circ \sigma \circ {}^t\Psi = \Psi^{-1}.$$
		Indeed, we
		have $$D\circ \sigma \circ {}^t\Psi \circ \Psi = D\circ\sigma \circ (\mathrm{id}
		+ {}^t \eta)^{-\frac{1}{2}} \circ
		{}^t \Phi \circ \Phi \circ (\mathrm{id} + \eta)^{-\frac{1}{2}} =
		(\mathrm{id} + \eta)^{-\frac{1}{2}} \circ D\circ\sigma\circ\sigma\circ D \circ
		(\mathrm{id} +
		\eta)^{\frac{1}{2}} = \mathrm{id},$$ where we have used that $\eta \circ D\circ
		\sigma =D\circ \sigma \circ {}^t\eta$ and that ${}^t\Phi \circ \Phi =\sigma\circ
		D\circ (\mathrm{id} + \eta)$. Thus if one sets
		\begin{equation} \label{eq blowupselfdualCK} \pi^i_{\tilde{X}} := \Psi
		\circ \left( \pi_X^i \oplus \bigoplus_{j=1}^{r} \pi_Y^{i-2j} \right)
		\circ \Psi^{-1},
		\end{equation}
		then, as in the proof of Proposition \ref{prop self-dual projbundle}, we see
		that $\{\pi^i_{\tilde{X}}\}$ defines a self-dual Chow--K\"unneth
		decomposition for $\tilde{X}$. Moreover, it is $G$-invariant because
		the  maps involved in the
		diagram \eqref{eq blow-up diagram} are $G$-equivariant.
	\end{proof}
	
	\begin{rmk}\label{rmk nonconnected center}
		The proposition can be easily generalized to the case where $Y$ is a
		$G$-invariant disjoint union of smooth closed subvarieties of $X$.
		The construction of $\pi^i_{\tilde{X}}$ also commutes with taking product.
		Namely, if $X'$ is another smooth projective variety with self-dual
		Chow--K\"unneth decompositon, then $\tilde{X}\times X'$ is the blow-up of
		$X\times X'$ with center $Y\times X'$. Repeating the above construction in this
		setting, we obtain a self-dual Chow--K\"unneth decomposition on $\tilde{X}\times
		X'$ that agrees with the product Chow--K\"unneth decomposition. 
	\end{rmk}

	\subsection{Generically finite quotients}
	Consider a surjective morphism $f:V \rightarrow W$ of smooth
	projective varieties. The morphism $\Gamma_f : \mathfrak{h}(V)
	\rightarrow \mathfrak{h}(W)$ of motives is surjective and admits a
	section, say $s$. Thus $\mathfrak{h}(W)$ is isomorphic to the direct
	summand $\mathrm{Im}(s\circ \Gamma_f)$ of
	$\mathfrak{h}(V)$. Assume that $V$ has a Chow--K\"unneth decomposition
	$\{\pi^i_V\}$. Although it is true that the homology classes of the cycles
	$\Gamma_f\circ
	\pi_V^i \circ s$ give a K\"unneth decomposition for $V$ (because the
	idempotents $\pi^i_V$ are central modulo homological equivalence in
	the ring of self-correspondences of $V$), it is in general not true
	that the cycles $\Gamma_f\circ \pi_V^i \circ s$ define idempotents
	modulo rational equivalence. Thus finding a Chow--K\"unneth
	decomposition for $W$ is usually not a straightforward
	matter. However, Proposition \ref{prop self-dual CK generic finite} below gives
	a sufficient condition for a Chow--K\"unneth decomposition to descend along a
	generically finite morphism.

	\begin{prop}\label{prop self-dual CK generic finite}
		Let $p:X\rightarrow Y$ be a generically finite morphism between smooth
		projective varieties of dimension $d$. Assume that $X$ is endowed with a
		self-dual Chow--K\"unneth decomposition $\{\pi_X^i\}$. If
		${}^t\Gamma_p\circ\Gamma_p$ sits in $ \CH_d(X\times X)_0$, then $Y$ also has a
		self-dual Chow--K\"unneth decomposition given by
		\[
		\pi_Y^i = \frac{1}{N}\Gamma_p\circ \pi_X^i \circ {}^t\Gamma_p,\qquad 0\leq
		i\leq 2d,
		\]
		where $N$ is the degree of $p$. Furthermore the morphism $p$ is of pure grade 0.
	\end{prop}
	
	\begin{proof}
		It is clear that
		\[
		\Gamma_p\circ {}^t\Gamma_p = N\Delta_Y.
		\]
		It follows that the cycles $\pi^i_Y$ given in the proposition lift the K\"unneth
		components in cohomology. For all $i, j$, we have
		\begin{align*}
		N^2\pi_Y^i\circ \pi_Y^j & = \Gamma_p\circ \pi_X^i \circ ({}^t\Gamma_p \circ
		\Gamma_p)\circ \pi^j_X\circ {}^t\Gamma_p\\
		& = \Gamma_p\circ ({}^t\Gamma_p \circ \Gamma_p)\circ \pi^i_X\circ\pi^j_X \circ
		{}^t\Gamma_p\\
		& = N \Gamma_p\circ \pi_X^i\circ\pi_X^j \circ {}^t\Gamma_p\\
		& =\begin{cases} N^2\pi_Y^i, &i=j\,;\\
		0, &i\neq j.
		\end{cases}
		\end{align*}
		Here the second equality follows from Lemma \ref{lem corr grade 0}.
		Hence $\{\pi_Y^i\}$ is a Chow--K\"unneth decomposition and  it is clearly
		self-dual.
		
		To show that $p$ is of grade 0, we note that if $i+j\neq 2d$ then
		\begin{align*}
		(\pi_X^i\otimes\pi_Y^j)_*\Gamma_{p} & = \pi_Y^j\circ \Gamma_p\circ
		{}^t\pi_X^i\\
		& = \frac{1}{N} \Gamma_p\circ\pi_X^j\circ {}^t\Gamma_p\circ \Gamma_p\circ
		\pi^{2d-i}_X\\
		& = \frac{1}{N} \Gamma_p\circ {}^t\Gamma_p\circ \Gamma_p\circ \pi_X^j\circ
		\pi_X^{2d-i}\\
		& =  0.
		\end{align*}
		Hence $\Gamma_p\in \CH^d(X\times Y)_0$.
	\end{proof}
	
	One application of independent interest of the above proposition concerns finite
	quotients.
	
	\begin{cor}\label{prop G-invariantCK1}
		Let $X$ be a smooth projective variety endowed with the action of a
		finite group $G$. Assume that $X$ has a  $G$-invariant self-dual
		Chow--K\"unneth decomposition $\{\pi_X^i\}$. Then the quotient
		variety $X/G$ has a self-dual Chow--K\"unneth decomposition
		$\{\pi_{X/G}^i\}$.
	\end{cor}
	\begin{proof} First, note that we do not need to assume that the quotient $X/G$
		is smooth because we work with
		Chow groups with rational coefficients.
		For any element $g\in G$, let us still denote $g \in \CH_d(X\times X)$ the
		class of the graph of the action of $g$ on $X$. By Lemma \ref{lem corr grade 0},
		the Chow--K\"unneth decomposition $\{\pi^i_X\}$ is $G$-invariant if and only
		if $g\in \CH^d(X\times X)_0$, for all $g\in G$. Let $p:X \rightarrow Y=X/G$ be
		the quotient morphism. Then
		\[
		{}^t\Gamma_p \circ \Gamma_p = \sum_{g\in G} g\in\CH^d(X\times X)_0
		\]
		and hence Proposition \ref{prop self-dual CK generic finite} applies, showing
		that 
		\begin{equation} \label{eq quotientCK}
		\pi^i_{X/G}  := \frac{1}{|G|} \,
		\Gamma_p \circ \pi^i_X \circ {}^t\Gamma_p 
		\end{equation} defines a self-dual Chow--K\"unneth decomposition of $X/G$.
	\end{proof}

	\section{Multiplicative Chow--K\"unneth decompositions}
	
	\begin{defn}\label{def mult}
		Let $X$ be a smooth projective variety of dimension $d$ over a field $k$,
		endowed with a Chow--K\"unneth decomposition $\{\pi_X^i : 0 \leq i \leq 2d\}$. 
		The Chow--K\"unneth decomposition $\{\pi_X^i, 0
		\leq i \leq 2d\}$ is \emph{multiplicative} if 
		$$\pi_X^k \circ \Delta_{123} \circ
		(\pi_X^i \otimes \pi_X^j) = 0 \in \CH^{2d}(X\times X \times X) \quad
		\text{whenever} \ k \neq i+j.$$
	\end{defn}
	In particular, a multiplicative
	Chow--K\"unneth decomposition induces a multiplicative bi-grading on the Chow
	ring of $X$\,: $$\CH^*(X) = \bigoplus_{i,s} \CH^i(X)_s, \quad \mbox{where} \
	\CH^i(X)_s:=(\pi_X^{2i-s})_*\CH^i(X).$$
	Note that $\pi_X^k \circ \Delta_{123} \circ
	(\pi_X^i \otimes \pi_X^j) = (\pi_X^k \otimes {}^t\pi_X^i \otimes
	{}^t\pi_X^j)_*\Delta_{123}$, so that  if $\{\pi_X^i\}$ is self-dual,  then
	$\{\pi_X^i\}$ is multiplicative if and only if
	\begin{equation}\label{eq criterion multselfdual}
	(\pi_X^k \otimes \pi_X^i \otimes \pi_X^j)_*\Delta_{123} = 0 \in
	\CH^{2d}(X\times X \times X) \quad
	\text{whenever} \ i+j+k \neq 4d,
	\end{equation} 
	\emph{i.e.}, if and only if the small diagonal $\Delta_{123}$ sits in
	$\CH^{2d}(X^3)_0$ for the product Chow--K\"unneth decomposition on $X^3$\,; see
	also \cite[Proposition 8.4]{sv}.
	In this section we study the stability of having a (self-dual) multiplicative
	Chow--K\"unneth decomposition under product, projective bundle, blow-up, and
	contraction under a generically finite morphism.
	
	\subsection{Product varieties}
	We  recall the following easy but crucial property of multiplicative
	Chow--K\"unneth decomposition\,:
	\begin{prop}[Theorem 8.6 in \cite{sv}] \label{prop multstable} 
		Let $X$ and $Y$ be smooth projective varieties, each endowed with a
		multiplicative Chow--K\"unneth decomposition $\{\pi_X^i\}$ and $\{\pi_Y^i\}$
		respectively. Then 
		the product Chow--K\"unneth decomposition 
		$\{\pi_{X\times Y}^i\}$ defined in \eqref{eq productCK} is multiplicative.
		Furthermore, if all the Chern classes of $X$ and $Y$ are in the graded-0 part,
		then so are the Chern classes of $X\times Y$.
	\end{prop}
	\begin{proof} We include a proof for the sake of completeness.
		Let $p_X : X \times Y \times
		X \times Y \times X \times Y \rightarrow X \times X \times X$ be the
		projection on the first, third and fifth factors, and let $p_Y$
		denote the projection on the second, fourth and sixth
		factors. Writing $\Delta_{123}^X$ for the small diagonal of $X$ and
		similarly for $Y$ and $X \times Y$, we have the
		identity $$\Delta_{123}^{X\times Y} = p_X^*\Delta_{123}^X \cdot
		p_Y^*\Delta_{123}^Y.$$ We immediately deduce that
		$$\pi_{X\times Y}^a \circ \Delta_{123}^{X\times Y} \circ  (\pi_{X\times Y}^b
		\otimes \pi_{X\times
			Y}^c) = \sum_{\substack{i+i'=a \\ j+j'=b
				\\ k+k'=c}} p_X^*\left[\pi_X^i \circ \Delta_{123}^X \circ (\pi_X^j \otimes
		\pi_X^k)\right] \cdot
		p_Y^*\left[\pi_Y^{i'} \circ \Delta_{123}^Y \circ (\pi_Y^{j'} \otimes
		\pi_Y^{k'}) \right].$$
		By definition of multiplicativity, the cycles $\pi_X^i \circ \Delta_{123}^X
		\circ (\pi_X^j \otimes  \pi_X^k)$ and $\pi_Y^{i'} \circ \Delta_{123}^Y \circ
		(\pi_Y^{j'} \otimes
		\pi_Y^{k'})$ are both non-zero only if
		$i=j+k$ and $i'=j'+k'$. Therefore $$\pi_{X\times Y}^a \circ
		\Delta_{123}^{X\times Y} \circ  (\pi_{X\times Y}^b \otimes \pi_{X\times
			Y}^c) = 0 \quad \mbox{if } a \neq b+c.$$ This exactly means that the
		product Chow--K\"unneth decomposition on $X \times Y$ is
		multiplicative. 
		
		For the statement concerning the Chern classes, we first note that 
		\[
		p_1^*\CH^p(X)_s \subset \CH^p(X\times Y)_s \qquad\text{and}\qquad
		p_2^*\CH^q(Y)_r\subset \CH^q(X\times Y)_r
		\]
		where $p_1$ and $p_2$ are the two projections of $X\times Y$ onto the two
		factors\,; see \cite[Proposition 8.7]{sv}. It follows from the isomorphism 
		\[
		\mathscr{T}_{X\times Y} \cong p_1^*\mathscr{T}_X\oplus p_2^*\mathscr{T}_Y
		\]
		that $c(X\times Y)= p_1^*c(X)\cdot p_2^* c(Y)\in \CH^*(X\times Y)_0$ as long as
		$c(X)\in\CH^*(X)_0$ and $c(Y)\in \CH^*(Y)_0$.
	\end{proof}

	\subsection{Projective bundles} The notations are those of Paragraph \ref{sec
		selfdualprojbundle}\,; $\pi : \PP(\mathscr{E}) \rightarrow X$ denotes the
	geometric projectivization of the vector bundle $\mathscr{E}$ on $X$.

	\begin{prop}\label{prop projbundleselfdualmult}
		Let $X$ be a smooth projective variety of dimension $d$ endowed with a self-dual
		multiplicative Chow--K\"unneth decomposition $\{\pi_X^i\}$ and let $\mathscr{E}$
		be a vector bundle on $X$ of rank $r+1$. Assume that the Chern classes
		$c_p(\mathscr{E})$ sit in $\CH^p(X)_0$ for all $p\geq 0$. Then the self-dual
		Chow--K\"unneth decomposition $\{\pi_{\PP(\mathcal{E})}^i\}$ of
		$\PP(\mathscr{E})$ defined in \eqref{eq projbundleselfdualCK} is multiplicative.
		If all the Chern classes of $X$
		are in the graded-0 part, then so are the Chern classes of $\PP(\mathscr{E})$.
		Moreover, 
		\begin{equation} \label{eq projbundleformula} \CH^p(\PP(\mathscr{E}))_s :=
		(\pi_{\PP(\mathscr{E})}^{2p-s})_*\CH^p(\PP(\mathscr{E})) =
		\bigoplus_{l=0}^{r}
		\xi^l\cdot\pi^*\CH^{p-l}(X)_s,
		\end{equation}
		where $\CH^p(X)_s := (\pi^{2p-s}_X)_*\CH^p(X)$. Under the above Chow-K\"unneth
		decompositions, the natural morphism $\pi:\PP(\mathscr{E})\rightarrow X$ is of
		pure grade 0.
	\end{prop}
	\begin{proof}
		It is clear that the Chow--K\"unneth decomposition $\{p^i_{\PP(\mathscr{E})}\}$
		defined in \eqref{eq
			projbundleCK} induces the decomposition \eqref{eq projbundleformula} of the Chow
		groups of $\PP(\mathscr{E})$. We first show that the self-dual Chow--K\"unneth
		decomposition $\{\pi_{\PP(\mathscr{E})}^i\}$ in Proposition \ref{prop self-dual
			projbundle} induces the same decomposition \eqref{eq projbundleformula}. With
		notations as in the proof of Proposition \ref{prop self-dual projbundle},
		it is enough to check that the correspondence $\eta :=\sigma\circ {}^t\Phi \circ
		\Phi -
		\mathrm{id}$ preserves the grading on $\CH^*(\bigoplus_{l=0}^r
		\mathfrak{h}(X)(l))$ induced by the Chow--K\"unneth decomposition $\{\pi^i_X\}$.
		By definition of $\Phi$, it suffices to check (in the general situation where
		$X$ is any smooth projective variety with a self-dual multiplicative
		Chow--K\"unneth decomposition $\{\pi^i_X\}$) that if $\alpha$ is a cycle in
		$\CH^*(X)_s$, then $\pi_*(\xi^i \cdot \pi^*\alpha)$ sits in $\CH^*(X)_s$. 
		But then, the projection formula gives $\pi_*(\xi^i \cdot \pi^*\alpha) = \alpha
		\cdot \pi_*(\xi^i)$. By multiplicativity of the Chow--K\"unneth decomposition
		$\{\pi_X^i\}$, we will be done if $\pi_*(\xi^i)$ sits in $\CH^*(X)_0$. The Chow
		ring of $\PP(\mathscr{E})$ is given by
		\begin{equation*}\label{eq chow proj bundle}
		\CH^*(\PP(\mathscr{E})) = \CH^*(X)[\xi],\qquad \text{where }\quad
		\xi^{r+1} + \pi^*c_1(\mathscr{E})\xi^r + \cdots
		+\pi^*c_r(\mathscr{E})\xi + \pi^*c_{r+1}(\mathscr{E}) = 0,
		\end{equation*}
		from which it follows that $\pi_*\xi^i$ is a polynomial in the Chern
		classes of $\mathscr{E}$. By assumption the Chern classes of $\mathscr{E}$
		belong to $\CH^*(X)_0$, and we conclude that $\pi_*(\xi^i)$ sits in $\CH^*(X)_0$
		by multiplicativity of $\{\pi_X^i\}$. 
		
		The triple product $\PP(\mathscr{E}) \times \PP(\mathscr{E}) \times
		\PP(\mathscr{E})$ is obtained by taking successively the projectivization of
		three vector bundles on $X\times X\times X$. Therefore, we may repeat the above
		argument, and we see that the decomposition of the Chow groups of
		$\PP(\mathscr{E})\times \PP(\mathscr{E}) \times \PP(\mathscr{E})$ is the same
		with respect to the following two Chow--K\"unneth decompositions\,: (1) the 
		Chow--K\"unneth decomposition $\{P^i\}$ obtained as the $3$-fold product of
		$\{p^i_{\PP(\mathscr{E})}\}$, and (2) the Chow--K\"unneth decomposition
		$\{\Pi^i\}$ obtained as the $3$-fold product of $\{\pi^i_{\PP(\mathscr{E})}\}$.
		
		In \cite[Proposition 13.1]{sv}, we proved under the assumptions of the
		proposition that the Chow--K\"unneth decomposition $\{p^i_{\PP(\mathscr{E})}\}$
		is multiplicative. By \eqref{eq criterion multselfdual}, this means that the
		small diagonal $\Delta_{123}^{\PP(\mathscr{E})}$ belongs to
		$(P^{4(d+r)})_*\CH_{d+r}(\PP(\mathscr{E}) \times \PP(\mathscr{E}) \times
		\PP(\mathscr{E}))$.
		Thus we obtain from the above that
		$\Delta_{123}^{\PP(\mathscr{E})}$ belongs to
		$(\Pi^{4(d+r)})_*\CH_{d+r}(\PP(\mathscr{E}) \times \PP(\mathscr{E}) \times
		\PP(\mathscr{E}))$.
		Therefore, by \eqref{eq criterion multselfdual}, the Chow--K\"unneth
		decomposition $\{\pi^i_{\PP(\mathscr{E})}\}$ is multiplicative.
		
		For the statement concerning the Chern classes, we only need to show that the
		Chern
		classes of the relative tangent bundle $\mathscr{T}_{\PP(\mathscr{E})/X}$ are in
		the graded-0 part. But this follows immediately from the short exact sequence
		\[
		\xymatrix{
			0\ar[r] &\calO_{\PP(\mathscr{E})}\ar[r] &\pi^*\mathscr{E}\otimes
			\calO(1)\ar[r] &\mathscr{T}_{\PP(\mathscr{E})/X}\ar[r] &0
		}
		\]
		by taking Chern classes of the sheaves involved.
		
		As in Remark \ref{rmk projselfdual commutes with prod}, we take $Y=X$ and let
		\[
		\pi\times \mathrm{id}: \PP(\mathscr{E})\times X\rightarrow X\times X
		\]
		be the product morphism. Then the fact that $\Delta_X$ belongs to
		$\CH^{d_X}(X\times X)_0$
		implies
		\[
		\Gamma_\pi = (\pi\times
		\mathrm{id})^*\Delta_X\in\CH^{d_X}(\PP(\mathscr{E})\times X)_0.
		\]
		Hence $\pi$ is of pure grade 0.
	\end{proof}

	\subsection{Blow-ups}
	We take on
	the notations from paragraph \ref{sec selfdualblowup}\,; the embedding
	$i:Y\hookrightarrow X$ is a smooth closed subvariety of codimension
	$r+1$ and $\tilde{X}:=\mathrm{Bl}_Y(X)$ is the blow-up of $X$ along
	$Y$, and we have the blow-up diagram \eqref{eq blow-up diagram}
	\begin{equation*} 
	\xymatrix{
		E\ar[r]^j\ar[d]_\pi & \tilde{X}\ar[d]^\rho\\
		Y\ar[r]^i & X
	}
	\end{equation*}
	where $E\cong\PP(\mathscr{N}_{Y/X})$ is the exceptional
	divisor. Recall from the blow-up formula \eqref{eq Chow of blow-up}
	that the Chow groups of $\tilde{X}$ can be described explicitly as
	\begin{equation*}
	\Phi_* :  \CH^p(X) \oplus\left( \bigoplus_{l=1}^{r}
	\CH^{p-l}(Y)\right)
	\overset{\cong}{\longrightarrow} \CH^p(\tilde{X}), 
	\quad (\alpha, \beta_1,\ldots,\beta_{r})\mapsto \rho^*\alpha 
	+ j_*\left( \sum_{l=1}^{r} \xi^{l-1}\cdot \pi^*\beta_l \right),
	\end{equation*}
	where $\xi\in\CH^1(E)$ is the Chern class of the relative
	$\calO(1)$-bundle. 
	
	\begin{prop}\label{prop multCK blow-up}
		Assume that both $X$ and $Y$ admit self-dual multiplicative
		Chow--K\"unneth decompositions $\{\pi^i_X\}$ and $\{\pi^i_Y\}$, respectively, 
		such that
		\begin{enumerate}[(i)]
			\item the Chern classes of the normal bundle $\mathscr{N}_{Y/X}$ sit
			in $\CH^*(Y)_0$\,;
			\item the morphism $i:Y\rightarrow X$ is of pure grade 0.
		\end{enumerate}
		Then the self-dual Chow--K\"unneth decomposition \eqref{eq blowupselfdualCK} of
		$\tilde{X}$ is multiplicative. 
		Moreover, 
		\begin{equation} \label{eq blowupformula}
		\CH^p(\tilde{X})_s := (\pi_{\tilde{X}}^{2p-s})_*\CH^p(\tilde{X}) =
		\rho^*\CH^p(X)_s \oplus\left( \bigoplus_{l=1}^{r}
		j_*(\xi^{l-1}\cdot\pi^*\CH^{p-l}(Y)_s) \right),
		\end{equation}
		where $\CH^p(X)_s = (\pi_X^{2p-s})_*\CH^p(X)$ and $\CH^p(Y)_s  =
		(\pi_Y^{2p-s})_*\CH^p(Y)$ and $\xi :=c_1\left(\mathcal{O}_{\tilde{X}}(-E)|_E
		\right)$. Furthermore, if the Chern classes of $X$ are in the graded-0 part,
		then so are the Chern classes of $\tilde{X}$. The natural blow-up morphism
		$\rho$ is of pure grade 0.
	\end{prop}
	\begin{proof} 
		The proof is similar to that of \cite[Proposition 13.2]{sv}, and consequently we
		only sketch the main steps. The difference with \cite[Proposition 13.2]{sv} is
		that we have removed the injectivity assumption on $i_*$. For that purpose, an
		explicit computation of the small diagonal is carried out in Lemma \ref{lem
			blowup pullback diagonal}. Note that the assumption \emph{(ii)} implies that
		$i^*$ and
		$i_*$ are compatible with the gradings on the Chow groups, namely
		\[
		i_*\CH^p(Y)_s \subseteq \CH^{p+r+1}(X)_s,\quad i^*\CH^p(X)_s \subseteq
		\CH^p(Y)_s.
		\]
		In addition to that, by (\textit{i}) of Proposition \ref{prop action of
			correspondence}, the push-forward and pull-back via $i\times 
		\mathrm{id}_Z: Y\times Z\rightarrow X\times Z$ also respect the gradings on
		Chow ring, for all smooth projective variety $Z$ endowed with a multiplicative 
		Chow--K\"unneth decomposition. 
		
		\emph{Step 1}: Without the condition of being self-dual, by \cite[Proposition
		13.2]{sv}, we can construct a Chow--K\"unneth decomposition
		$\{p_{\tilde{X}}^i\}$ of $\tilde{X}$ such that the induced grading on the Chow
		ring is the same as the one given in equation \eqref{eq blowupformula}.
		Furthermore, the homomorphisms $j_*$ and $j^*$ are compatible with the gradings,
		where $\CH^*(E)_s$ is given by \eqref{eq projbundleformula} which is induced by
		a self-dual multiplicative Chow--K\"unneth decomposition on $E$. More generally,
		if $Z$ is a smooth projective variety with a multiplicative Chow--K\"unneth
		decomposition, then 
		\[
		(j\times \mathrm{id}_Z)_*: \CH^p(E\times Z) \rightarrow
		\CH^{p+1}(\tilde{X}\times Z)\qquad \text{and} \qquad
		(j\times \mathrm{id}_Z)^*: \CH^p(\tilde{X}\times Z)\rightarrow \CH^p(E\times Z)
		\]
		respect the gradings. Indeed, the action of $(i\times \mathrm{id}_Z)_*$ and
		$(i\times \mathrm{id}_Z)^*$ respects the gradings of $\CH^*(Y\times Z)$ and
		$\CH^*(X\times Z)$ and hence the argument for the compatibility of $j_*$ and
		$j^*$
		with the grading of Chow groups (as in the proof of \cite[Proposition 13.2]{sv})
		applies to show the compatibility of $(j\times \mathrm{id}_Z)_*$ and $(j\times
		\mathrm{id}_Z)^*$ with the gradings. We apply this successively to
		\[
		E\times E\times E \longrightarrow E\times E\times \tilde{X} \longrightarrow
		E\times \tilde{X}\times\tilde{X} \longrightarrow \tilde{X}\times \tilde{X}\times
		\tilde{X}
		\]
		and show that the pull-back and push-forward of $j^{\times 3}: E^3\rightarrow
		\tilde{X}^3$ respect the grading on the Chow groups. This was implicitly used in
		the proof of \cite[Proposition 13.2]{sv} but the assumption there (compatibility
		of $i^*$ and $i_*$ with the grading on Chow groups) is insufficient to deduce
		it.
		
		\emph{Step 2}: The self-product $\tilde{X}\times\tilde{X}\times\tilde{X}$ has
		two natural Chow--K\"unneth decompositions. The first one is the $3$-fold
		product of
		$\{p_{\tilde{X}}^i\}$. The second one is obtained when $\tilde{X}^3$ is viewed
		as the successive blow-up of $X^3$ along $Y\times X \times X$, $\tilde{X}\times
		Y \times X$, and $\tilde{X}\times \tilde{X} \times Y$, where $X^3$ is given the
		product
		Chow--K\"unneth decomposition. Then it turns out that the above two
		Chow--K\"unneth decompositions are the same\,; see the proof of
		\cite[Proposition
		13.2]{sv}. As a consequence we have
		\begin{equation}\label{eq big inclusion}
		(\rho^{\times 3})^*\CH(X^3)_0 + (j^{\times 3})_*\left( \xi_1^{i_1} \xi_2^{i_2}
		\xi_3^{i_3}\cdot (\pi^{\times 3})^*\CH^*(Y^3)_0 \right) \subseteq
		\CH^*(\tilde{X}^3)_0.
		\end{equation}
		Here the notation is as follows. If $f:Z\rightarrow Z'$ is a morphism, then
		$f^{\times 3}: Z^3\rightarrow Z'^3$ is the 3-fold self product of $f$\,; the
		class
		$\xi_i\in\CH^1(E^3)$ is the pull-back of $\xi$ via the projection onto the
		$i^{\text{th}}$ factor, $i=1,2,3$.
		
		\emph{Step 3}: The small diagonal $\Delta^{123}_{\tilde{X}}\in
		\CH^{2d}(\tilde{X}^3)$ is contained in the graded-0 part. This was proved in
		\cite[Proposition 13.2]{sv} under the assumption that $i_*$ is injective.
		Without assuming that $i_*$ is injective, this follows immediately from Lemma
		\ref{lem blowup pullback diagonal} and equation \eqref{eq big inclusion}.
		
		\emph{Step 4}: The above three steps can be carried out if we replace
		$p_{\tilde{X}}^i$ by $\pi_{\tilde{X}}^i$.
		We have seen that the small diagonal  $\Delta_{123}^{\tilde{X}}$ belongs to
		$(p^{4d}_{\tilde{X} \times \tilde{X} \times \tilde{X}})_*\CH_{d}(\tilde{X}
		\times\tilde{X} \times \tilde{X})$,
		where $p^{4d}_{\tilde{X} \times \tilde{X} \times \tilde{X}} := \sum_{i+j+k
			=4d} p^i_{\tilde{X}} \otimes p^j_{\tilde{X}} \otimes p^k_{\tilde{X}}$ 
		and  $p^i_{\tilde{X}}$ is the idempotent defined in \eqref{eq blowupCK}.
		In order to prove the proposition, it suffices as in the proof of Proposition
		\ref{prop projbundleselfdualmult} to show that the Chow--K\"unneth decomposition
		$\{\pi^i_{\tilde{X}}\}$ of \eqref{eq blowupselfdualCK} 
		satisfies 
		$$(\pi_{\tilde{X}}^i)_*\CH^p(\tilde{X}) =
		(p_{\tilde{X}}^i)_*\CH^p(\tilde{X}) = \rho^*\CH^p(X)_{2i-s} \oplus\left(
		\bigoplus_{l=0}^{r-1}
		j_*(\xi^l\cdot\pi^*\CH^{p-l-1}(Y)_{2i-s}) \right).$$ 
		For this it is enough to show that $\eta := \sigma\circ D \circ {}^t\Phi \circ
		\Phi -
		\mathrm{id}$ preserves the grading on $\CH^*( \mathfrak{h}(X) \oplus
		\bigoplus_{l=1}^{r} \mathfrak{h}(Y)(l))$. It is in fact enough to show that $
		{}^t\Phi \circ \Phi$ preserves the grading. By \eqref{eq tphiphi}, we only need
		to show that $\rho_*\rho^*\CH^*(X)_s \subseteq \CH^*(X)_s$ and that $\pi_*\circ
		h^l \circ \pi^*\CH^*(Y)_s \subseteq \CH^*(Y)_s$ for all $l\geq 0$ and all
		integers $s$. The first inclusion is obvious because $\rho_*\rho^*$ is the
		identity on $\CH^*(X)$, and the second inclusion was already established in the
		proof of Proposition \ref{prop projbundleselfdualmult}.
		
		The statement concerning the Chern classes follows from the short exact sequence
		in Lemma \ref{lem blow-up tangent}, which
		again follows from a Chern class computation as before.
		
		Now we take $X'=X$ in Remark \ref{rmk nonconnected center} and note that
		\[
		\Gamma_\rho = (\rho\times \mathrm{id})^*\Delta_X.
		\]
		This implies that $\rho$ is of pure grade 0 since $\Delta_X\in\CH^{d_X}(X\times
		X)_0$. 
	\end{proof}
	
	\begin{lem}\label{lem blowup pullback diagonal}
		Let $\rho^{\times l}: \tilde{X}^l\rightarrow X^l$ be the $l$-fold self-product
		of $\rho$ and let
		\[
		j^{\times l}_{/Y}: E^{\times l}_Y=E\times_Y E\times_Y\cdots \times_Y
		E\rightarrow \tilde{X}^l
		\]
		be the closed immersion induced by $j$. Let $\mathscr{N}_l$ be the sheaf on
		$E^{\times l}_Y$ that is obtained as the pull-back of $\mathscr{N}_{Y/X}$ from
		$Y$ and let $\xi_k\in\CH^1(E^{\times l}_Y)$ be the pull-back of $\xi\in\CH^1(E)$
		from the $k^\text{th}$ factor, $1\leq k\leq l$. Then the following equality
		holds in $\CH_d(\tilde{X}^{2})$
		\begin{equation}
		\label{eq diagonal of blowup}
		(\rho \times \rho)^*\Delta_X = \Delta_{\tilde{X}} + (j\times_Y j)_*
		\left[\frac{c(\mathscr{N}_2)}{(1-\xi_1)(1-\xi_2)} \right]_d,
		\end{equation}
		where $[-]_d$ takes the dimension $d$ component. More generally,
		\[
		(\rho^{\times l})^*\Delta^{(l)}_X = \Delta^{(l)}_{\tilde{X}} + (j^{\times
			l}_{/Y})_* P(\xi_k, c(\mathscr{N}_l)),
		\]
		where $P(-)$ is a polynomial, $\Delta_X^{(l)}=\{(x,x,\ldots,x)\}\subset X^l$ is
		the small diagonal.
	\end{lem}
	
	\begin{proof}
		Let $\Delta_{ij}(E_{/Y})\in\CH^r(E^{\times l}_{/Y})$ be the
		relative (over $Y$) bigger diagonal given by $\{(x_1,\ldots,x_l):x_i=x_j\}$. We
		claim that $\Delta_{ij}(E/Y)$ is a polynomial of $\xi_i$, $\xi_j$ and
		$c(\mathscr{N}_l)$. To prove this claim, it suffices to show that in the case
		$l=2$. In this case we have the natural homomorphisms of sheaves
		\[
		\calO(-\xi_1)\rightarrow \mathscr{N}_2\qquad \text{and} \qquad
		\calO(-\xi_2)\rightarrow \mathscr{N}_2
		\]
		and they give rise to a section $\alpha$ of
		$\Big(\mathscr{N}_2/\calO(-\xi_1)\Big)\otimes \calO(\xi_2)$. The cycle
		$\Delta_{12}(E/Y)$ is the vanishing locus of $\alpha$ and hence its class is the
		top Chern class of the above sheaf. The claim follow immediately. Hence it
		suffices to show that
		\[
		(\rho^{\times l})^*\Delta^{(l)}_X = \Delta^{(l)}_{\tilde{X}} + (j^{\times
			l}_{/Y})_* P(\Delta_{ij}(E_{/Y}), \xi_k, c(\mathscr{N}_l)),
		\]
		for some polynomial $P$.
		
		Note that $\tilde{X}\times\tilde{X}$ can be viewed as a successive blow-up of
		$X\times X$ as follows
		\[
		\xymatrix{
			\tilde{X}\times\tilde{X} \ar[rr]^{\mathrm{id}_{\tilde{X}}\times \rho}
			&&\tilde{X}\times X \ar[rr]^{\rho\times \mathrm{id}_X} &&X\times X.
		}
		\]
		For each blow-up, we use \cite[Theorem 6.7]{fulton} (Blow-up Formula) and easily
		get
		\begin{equation*}
		(\rho\times \mathrm{id}_X)^* \Delta_X = \Gamma_\rho
		\end{equation*}
		and
		\begin{align*}
		(\mathrm{id}_{\tilde{X}} \times \rho)^* \Gamma_\rho & = \Delta_{\tilde{X}}+
		(\mathrm{id}_{\tilde{X}} \times j)_* 
		\left[ (p_{\tilde{X}\times
			E,2})^*c\left(\frac{\pi^*\mathscr{N}_{Y/X}}{\calO_E(-1)}\right) \cap
		(\mathrm{id}_X \times \pi)^* s(E,\tilde{X}) \right]_d \\
		& = \Delta_{\tilde{X}} + (j\times_Y j)_*
		\left[\frac{c(\mathscr{N}_l)}{(1-\xi_1)(1-\xi_2)} \right]_d.
		\end{align*}
		The general case follows by induction on $l$. The induction step is established
		by observing that
		\[
		(\rho^{\times (l+1)})^*\Delta^{(l+1)}_{X} = \Big( (\rho^{\times
			l})^*\Delta^{(l)}_X \times \tilde{X} \Big)\cdot \Big(\tilde{X}^{l-1} \times
		(\rho\times \rho)^*\Delta_X\Big)
		\] 
		and by applying \cite[Theorem 6.3]{fulton} (Excess Intersection Formula) to the
		following square
		\[
		\xymatrix{
			E^{\times l}_{/Y} \ar[rr]^{j^{\times{(l-1)}}_{/Y}\times \mathrm{id}_{E^{\times
						2}_Y}\quad}\ar[d]_{\mathrm{id}\times j} &&\tilde{X}^{l-1}\times E\times_Y\times
			E\ar[d]^{\mathrm{id}\times j^{\times 2}_{/Y}} \\
			E^{\times l}_{/Y}\times \tilde{X} \ar[rr]^{j^{\times l}_{/Y} \times
				\mathrm{id}_{X}} &&\tilde{X}^{l-1}\times \tilde{X}\times \tilde{X}.
		}
		\]
	\end{proof}
	
	\begin{rmk}\label{rmk blow-up diagram degree 0} 
		Under the assumptions of Proposition \ref{prop multCK blow-up}, all the
		varieties involved in the blow-up diagram \eqref{eq blow-up
			diagram}, namely $E$, $Y$, $X$ and
		$\tilde{X}$, are given a self-dual multiplicative Chow--K\"unneth decomposition
		such that all the morphisms involved in that diagram , namely
		$i$, $\pi$, $j$ and $\rho$,
		are of pure grade 0. This fact will be  needed in the proof of Lemma \ref{lem
			S_1 admissible} to show that a certain set of closed subvarieties is admissible.

		By Proposition \ref{prop multCK blow-up}, it only remains to show that the
		closed immersion $j$ is of pure grade 0. In
		what follows, we will use the notations from the proof of
		Proposition \ref{prop self-dualCK}. We will also borrow the notations from the
		proof of Proposition \ref{prop self-dual projbundle} and we add a ``$'$" to each
		notation to avoid possible confusion. Note that $E\cong \PP(\mathscr{N}_{Y/X})$
		is naturally a projective bundle over $Y$. Let
		\[
		\Phi'=\left(\bigoplus_{l=0}^{r-1} h^l \circ {}^t\Gamma_{\pi} \right) \oplus
		\gamma'\circ {}^t\Gamma_{\pi}: \left( \bigoplus_{l=0}^{r-1}\mathfrak{h}(Y)(l)
		\right)\oplus \mathfrak{h}(Y)(r) \longrightarrow \mathfrak{h}(E)
		\]
		be the isomorphism defined in Section \ref{sec selfdualprojbundle} (note that
		this was denoted $\Phi$ there). Then we have a commutative diagram
		\[
		\xymatrix{
			\bigoplus_{l=1}^{r+1}\mathfrak{h}(Y)(l)\ar[d]_{\Phi'}
			\ar[rr]^{\mathrm{id}\oplus \Gamma_i\quad } && \left( \bigoplus_{l=1}^{r} 
			\mathfrak{h}(Y)(l)\right) \oplus \mathfrak{h}(X) \ar[d]^{\Phi}\\
			\mathfrak{h}(E)(1) \ar[rr]^{\Gamma_j} &&\mathfrak{h}(\tilde{X}).
		}
		\]
		The only nontrivial part of the above commutativity is 
		\[
		\Phi|_{\mathfrak{h}(X)} \circ \Gamma_i = \Gamma_j\circ (\gamma'\circ
		{}^t\Gamma_{\pi}).
		\]
		Note that $ \Phi|_{\mathfrak{h}(X)}={}^t\Gamma_{\rho}$, and one easily sees that
		the
		above identity is simply the ``key formula" of \cite[Proposition 6.7(a)]{fulton}
		stated at the level of correspondences. It is clear that the commutativity holds
		only when $\Phi'$ is taken to be the modified projective bundle formula
		isomorphism, instead of the usual one\,; see Section \ref{sec
			selfdualprojbundle}.
		Let $\sigma'$, $\eta'$ and $\Psi'=\Phi'\circ (\mathrm{id}+\eta')^{-\frac{1}{2}}$
		be the corresponding  $\sigma$, $\eta$ and $\Psi$ of the proof of Proposition
		\ref{prop self-dual projbundle} for the projective bundle $\pi:E\rightarrow Y$.
		By comparing \eqref{eq matrix 1} and \eqref{eq tphiphi}, we get
		\[
		\Gamma_j\circ \Phi'|_{\oplus_{l=1}^{r} \mathfrak{h}(Y)(l)} \circ
		\eta'|_{\oplus_{l=1}^{r} \mathfrak{h}(Y)(l)} = \Phi|_{\oplus_{l=1}^{r}
			\mathfrak{h}(Y)(l)} \circ \eta|_{\oplus_{l=1}^{r} \mathfrak{h}(Y)(l)}.
		\]
		As a result, we have
		\[
		\Gamma_j\circ \Psi' = \Psi \circ (\mathrm{id}\oplus \Gamma_i), \qquad \text{on
		}\bigoplus_{l=1}^{r} \mathfrak{h}(Y)(l).
		\]
		Note that $\eta'|_{\mathfrak{h}(Y)(r+1)}=\cdot c_{r}(\mathscr{N}_{Y/X}):
		\mathfrak{h}(Y)(r+1)\rightarrow \mathfrak{h}(Y)(1)$ and
		$\eta'|_{\mathfrak{h}(Y)(1)}=0$, and hence $\eta'^{\circ l}=0$ on
		$\mathfrak{h}(Y)(r+1)$ for all $l\geq 2$. As a consequence
		\[
		\Gamma_j\circ \Psi'|_{\mathfrak{h}(Y)(r+1)} =\Gamma_j\circ \Phi'\circ (1
		-\frac{1}{2}\eta')|_{\mathfrak{h}(Y)(r+1)}.
		\]
		Meanwhile, since $\eta|_{\mathfrak{h}(X)}=0$, one easily shows that
		\[
		\Psi\circ (\mathrm{id}\oplus \Gamma_i)|_{\mathfrak{h}(Y)(r+1)} =\Phi\circ
		(1+\eta)^{-\frac{1}{2}}|_{\mathfrak{h}(X)}\circ \Gamma_i = \Phi\circ \Gamma_i.
		\]
		It follows that
		\begin{align*}
		\Gamma_j\circ \Psi' &= \Psi\circ (\mathrm{id}\oplus \Gamma_i) -\frac{1}{2}
		(\Gamma_j\circ \Phi')|_{\mathfrak{h}(Y)(1)}\circ (\cdot
		c_r(\mathscr{N}_{Y/X}))\\
		& =  \Psi\circ (\mathrm{id}\oplus \Gamma_i) -\frac{1}{2} \Phi
		|_{\mathfrak{h}(Y)(1)}\circ (\cdot c_r(\mathscr{N}_{Y/X}))\\
		& =  \Psi\circ (\mathrm{id}\oplus \Gamma_i) -\frac{1}{2} \Psi
		|_{\mathfrak{h}(Y)(1)}\circ (\cdot c_r(\mathscr{N}_{Y/X})).
		\end{align*}
		Here the last equality uses $\Psi |_{\mathfrak{h}(Y)(1)} = \Phi
		|_{\mathfrak{h}(Y)(1)}$, which follows from the fact that
		$\eta|_{\mathfrak{h}(Y)(1)}=0$. Hence we get
		\[
		\Gamma_j = \Psi\circ (\mathrm{id}\oplus \Gamma_i)\circ (\Psi')^{-1}
		-\frac{1}{2} \Psi |_{\mathfrak{h}(Y)(1)}\circ (\cdot
		c_r(\mathscr{N}_{Y/X}))\circ (\Psi')^{-1}.
		\]
		This equality implies that $\Gamma_j$ is of pure grade 0 since all the morphisms
		appearing on the right-hand side are of pure grade 0.
	\end{rmk}

	\subsection{Generically finite quotients} The following proposition shows that
	the Chow--K\"unneth decomposition constructed
	in Proposition \ref{prop self-dual CK generic finite} is multiplicative.
	
	\begin{prop}\label{prop multCK generic finite}
		Let $p:X\rightarrow Y$ be a generically finite morphism between smooth
		projective varieties. Assume that $X$ is endowed with a self-dual multiplicative
		Chow--K\"unneth decomposition $\{\pi^i_X\}$ and that
		\[
		{}^t\Gamma_p\circ\Gamma_p\in \CH_d(X\times X)_0, \quad d=\dim X.
		\] 
		Then the Chow--K\"unneth decomposition of $Y$, as given in Proposition \ref{prop
			self-dual CK generic finite}, is also multiplicative.
	\end{prop}
	
	\begin{proof}
		Let $N$ be the degree of the morphism $p$. We first note that
		\[
		\Delta^Y_{123} =\frac{1}{N} (p\times p\times p)_*\Delta^X_{123} =
		\frac{1}{N}\Gamma_p\circ\Delta^X_{123} \circ ({}^t\Gamma_p\otimes {}^t\Gamma_p)
		\]
		and that
		\begin{align*}
		({}^t\Gamma_p\otimes{}^t\Gamma_p)\circ (\pi_Y^i\otimes\pi_Y^j) 
		& = \frac{1}{N^2}\Big( {}^t\Gamma_p\circ\Gamma_p\circ\pi^i_X\circ
		{}^t\Gamma_p\Big) \otimes
		\Big( {}^t\Gamma_p\circ\Gamma_p\circ\pi^j_X\circ {}^t\Gamma_p\Big)\\
		& = \frac{1}{N^2} \Big( \pi^i_X\circ {}^t\Gamma_p\circ\Gamma_p\circ
		{}^t\Gamma_p\Big) \otimes
		\Big( \pi^i_X\circ {}^t\Gamma_p\circ\Gamma_p\circ {}^t\Gamma_p\Big)\\
		& =  \Big( \pi^i_X\circ {}^t\Gamma_p\Big) \otimes  \Big( \pi^j_X\circ
		{}^t\Gamma_p\Big)\\
		& =(\pi^i_X\otimes \pi_X^j)\circ ({}^t\Gamma_p\otimes {}^t\Gamma_p).
		\end{align*}
		Hence we have
		\begin{align*}
		\pi^k_Y\circ \Delta^Y_{123} \circ (\pi_Y^i\otimes \pi_Y^j) 
		& = \frac{1}{N^2} \Gamma_p\circ \pi_X^k\circ {}^t\Gamma_p \circ \Gamma_p
		\circ \Delta^X_{123} \circ ({}^t\Gamma_p\otimes {}^t\Gamma_p)\circ
		(\pi_Y^i\otimes \pi_Y^j)\\
		& = \frac{1}{N^2} \Gamma_p\circ  {}^t\Gamma_p \circ \Gamma_p \circ \pi^k_X
		\circ
		\Delta^X_{123} \circ (\pi^i_X\otimes \pi^j_X)\circ ({}^t\Gamma_p\otimes
		{}^t\Gamma_p)\\
		& = \frac{1}{N}\Gamma_p \circ \pi^k_X \circ \Delta^X_{123} \circ (\pi^i_X\otimes
		\pi^j_X)\circ ({}^t\Gamma_p\otimes {}^t\Gamma_p).
		\end{align*}
		Note that in the above computation, we have used the commutativity of
		${}^t\Gamma_p\circ\Gamma_p$ and $\pi^i_X$ as correspondences\,; see Lemma
		\ref{lem
			corr grade 0}.
		By assumption $\pi^i_X$ is multiplicative and hence
		\[
		\pi^k_X \circ \Delta^X_{123} \circ (\pi^i_X\otimes \pi^j_X)=0,
		\]
		for all $k\neq i+j$. It follows that
		\[
		\pi^k_Y\circ \Delta^Y_{123} \circ (\pi_Y^i\otimes \pi_Y^j) =0,
		\]
		for all $k\neq i+j$ and this establishes the multiplicativity of $\{\pi_Y^i\}$.
	\end{proof}
	
	\begin{cor}\label{prop G-invariantCK}
		Let $X$ be a smooth projective variety endowed with the action of a
		finite group $G$. Assume that $X$ has a  $G$-invariant self-dual
		multiplicative
		Chow--K\"unneth decomposition $\{\pi_X^i\}$. Then the  Chow--K\"unneth
		decomposition
		$\{\pi_{X/G}^i\}$ defined in \eqref{eq quotientCK} of the quotient variety
		$X/G$ is self-dual and multiplicative.\qed
	\end{cor}

	\begin{rmk}
		Proposition \ref{prop multCK blow-up} and Corollary \ref{prop G-invariantCK}
		fix a couple gaps in the proof of \cite[Theorem 6]{sv}. First the proof that
		the Hilbert square $X^{[2]}$ can be endowed with a multiplicative
		Chow--K\"unneth decomposition that is \emph{self-dual} was omitted. Second
		Proposition \ref{prop multCK blow-up} corrects and improves \cite[Proposition
		13.2]{sv}\,: the assumption \emph{(iii)} of  \cite[Proposition 13.2]{sv} is
		superfluous, while the assumption \emph{(ii)} requiring $i_*$ and $i^*$ to be
		compatible with the gradings should be strengthened to requiring the inclusion
		morphism $i : Y \rightarrow X$ to be of pure grade $0$ (which is a stronger
		condition by Proposition \ref{prop action of correspondence}).
	\end{rmk}

	\section{Successive blow-ups}
	Let $X$ be a smooth projective variety with a multiplicative
	Chow--K\"unneth decomposition. In this section, we wish to understand
	when a variety obtained from $X$ by successive smooth blow-ups admits
	a multiplicative Chow--K\"unneth decomposition. For that matter we
	provide sufficient conditions on the centers of these successive
	blow-ups for the resulting variety to admit a multiplicative
	Chow--K\"unneth decomposition.
	We consider a fairly general situation, which could prove useful in future work.
	We are led to formulate the
	following technical definition. 
	\begin{defn} \label{def complete admissible}  Let
		$\mathcal{S}=\{Y_1,Y_2,\ldots,Y_n\}$ be a finite set of closed
		subvarieties of $X$. The set $\mathcal{S}$ is
		said to be \textit{complete} if the following two conditions hold\,:
		\begin{enumerate}[(i)]
			\item $X\in \mathcal{S}$\,;
			\item If $Y_{j_1},Y_{j_2}\in\mathcal{S}$ and the scheme-theoretic intersection
			$Y'=Y_{j_1}\cap
			Y_{j_2}$ is non-empty, then $Y'\in\mathcal{S}$.
		\end{enumerate}
		The set $\mathcal{S}$ is \textit{admissible} if the following
		conditions hold\,:
		\begin{enumerate}[(i)]
			\item each $Y_j$ is smooth and has a self-dual multiplicative Chow--K\"unneth
			decomposition\,;
			\item if $Y_{j_1} \subset Y_{j_2}$ are closed subvarieties of $X$ that
			belong to $\mathcal{S}$ and if $i:Y_{j_1}\hookrightarrow Y_{j_2}$
			denotes the embedding morphism, then
			\begin{enumerate}[(a)]
				\item The Chern classes of the normal bundle
				$\mathscr{N}_{Y_{j_1}/Y_{j_2}}$ sit in $\CH^*(Y_{j_1})_0$\,;
				\item The morphism $i: Y_{j_1}\hookrightarrow Y_{j_2}$ is of pure grade 0.
			\end{enumerate}
		\end{enumerate}
	\end{defn}
	
	\begin{rmk}\label{rmk admissible composition}
		Being admissible is transitive in the following sense. If $\{Y_1\subset Y_2\}$
		and $\{Y_2\subset Y_3\}$ are two admissible subsets, then $\{Y_1\subset Y_3\}$
		is admissible.
	\end{rmk}
	
	The first reason for introducing admissible sets of subvarieties of $X$ lies
	in the following.
	\begin{prop}\label{prop adm mult}
		Given any two elements $Y_{j_1}$ and $Y_{j_2}$ of an admissible set
		such that $Y_{j_1} \subset Y_{j_2}$, the blow-up of $Y_{j_2}$ along
		$Y_{j_1}$ has a multiplicative Chow--K\"unneth decomposition.
	\end{prop}
	\begin{proof}
		The proposition follows at once from  Proposition \ref{prop multCK blow-up}.
	\end{proof}
	
	Pick $Y\in \mathcal{S}$ and let $\tilde{X}$ be the blow-up of $X$ with
	center $Y$. Define $\mathrm{Bl}_Y(\mathcal{S})$ to be the subset of
	the set of smooth closed subvarieties of $\tilde{X}$ that consists of
	the strict transforms of the $Y_j$ for all $Y_j\in \mathcal{S}$ such
	that $Y_j$ is not contained in $Y$. Here is the main reason for introducing
	admissible sets. The following key proposition shows
	that admissible sets behave well after blowing-up along one of their
	elements, making it thus possible to 
	avoid checking that the assumptions of Proposition \ref{prop multCK blow-up} are
	met after each blow-up.
	
	\begin{prop}\label{prop admissible blow-up}
		Let $\mathcal{S}$ be a complete admissible subset of closed
		subvarieties of $X$. Let $Y\in\mathcal{S}$ and let $\tilde{X}$ be
		the blow-up of $X$ along $Y$. Then $\mathrm{Bl}_Y( \mathcal{S})$ is
		complete and admissible.
	\end{prop}
	
	Before proving Proposition \ref{prop admissible blow-up}, we state and prove
	three
	auxiliary lemmas. The notations are those of diagram \eqref{eq blow-up
		diagram}.\medskip

	First, the tangent bundles of $X$ and its blow-up $\tilde{X}$ are linked as
	follows.
	
	\begin{lem}\label{lem blow-up tangent}
		With notations as in diagram \eqref{eq blow-up diagram}, there is a
		short exact sequence
		\[
		\xymatrix{ 0\ar[r] &\mathscr{T}_{\tilde{X}}\ar[rr]
			&&\rho^*\mathscr{T}_X\ar[rr] &&j_*\mathscr{E}\ar[r] &0, }
		\]
		where $\mathscr{E}=\pi^*\mathscr{N}_{Y/X}/\calO_E(-1)$.
	\end{lem}
	
	\begin{proof}
		Let $\mathscr{Q}$ denote the quotient sheaf of the natural
		homomorphism $\mathscr{T}_{\tilde{X}} \rightarrow
		\rho^*\mathscr{T}_X$, so that we have a short exact sequence
		\[
		\xymatrix{ 0\ar[r] &\mathscr{T}_{\tilde{X}} \ar[rr]
			&&\rho^*\mathscr{T}_X\ar[rr] &&\mathscr{Q}\ar[r] &0.  }
		\]
		Since $\rho$ is an isomorphism away from $E$, we see that
		$\mathscr{Q}$ is supported on $E$. In particular,
		$\mathscr{Q}\otimes_{\calO_{\tilde{X}}} \calO_E
		=\mathscr{Q}$. Tensoring the above short exact sequence with
		$\calO_E$ gives the exact sequence
		\[
		\xymatrix{ \mathscr{T}_{\tilde{X}}|_E\ar[rr]
			&&\rho^*\mathscr{T}_{X}|_E\ar[rr] &&\mathscr{Q}\ar[r] &0.  }
		\]
		This sequence fits into the following commutative diagram
		\[
		\xymatrix{
			& &0 & 0 & & \\
			&0\ar[r] &\mathscr{N}_{E/\tilde{X}}\ar[r]\ar[u] &\pi^*\mathscr{N}_{Y/X}
			\ar[r]\ar[u] &\mathscr{E}\ar[r] &0\\
			0\ar[r] &\mathscr{T}_{E/Y}\ar[r] &\mathscr{T}_{\tilde{X}}|_E\ar[r]\ar[u]
			&\rho^*\mathscr{T}_X|E \ar[r]\ar[u] &\mathscr{Q}\ar[r]\ar[u]_\alpha &0\\
			0\ar[r] &\mathscr{T}_{E/Y}\ar@{=}[u]\ar[r] &\mathscr{T}_E\ar[r]\ar[u]
			&\pi^*\mathscr{T}_Y\ar[r]\ar[u] &0 & \\
			& &0\ar[u] &0\ar[u] && }
		\]
		The snake lemma implies that $\alpha:\mathscr{Q}\rightarrow
		\mathscr{E}$ is an isomorphism and we thus have an isomorphism $\mathscr{Q}\cong
		j_*\mathscr{E}$ of torsion sheaves on
		$\tilde{X}$. This proves the lemma.
	\end{proof}
	
	Secondly, the following lemma on the behavior of normal bundles under a blow-up
	will be useful.
	
	\begin{lem}\label{lem normal bundle blow up}
		Let $X$ be a smooth projective variety and let $Y,Z\subset X$ be two smooth
		closed subvarieties. Assume that $Z$ is not contained in $Y$ and that the
		scheme-theoretic intersection $Y':=Y\cap Z$ is smooth. Let
		$\rho:\tilde{X}\rightarrow X$ be the blow-up of $X$ along $Y$ and let
		$\tilde{Z}\subset \tilde{X}$ be the strict transform of $Z$. Let
		$\mathscr{N}_{Y,Z/X}$ be the locally free sheaf on $Y'$ that is the quotient of
		$\mathscr{T}_X|_{Y'}$ by the subbundle generated by $\mathscr{T}_{Y}|_{Y'}$ and
		$\mathscr{T}_{Z}|_{Y'}$. Namely,
		\[
		\mathscr{N}_{Y,Z/X} := \frac{\mathscr{T}_X|_{Y'}}{\langle
			\mathscr{T}_{Y}|_{Y'}, \mathscr{T}_{Z}|_{Y'} \rangle}.
		\]
		Then there is a short exact sequence
		\[
		\xymatrix{
			0\ar[r] & \mathscr{N}_{\tilde{Z}/\tilde{X}} \ar[r]
			&\rho'^*\mathscr{N}_{Z/X}\ar[r] &j'_*\pi'^* \mathscr{N}_{Y,Z/X}\ar[r] &0.
		}
		\]
		Here $\rho':\tilde{Z}\rightarrow Z$ is the blow-up morphism\,; $E'\subset
		\tilde{Z}$ is the exceptional divisor with $j': E'\hookrightarrow \tilde{Z}$
		being the closed immersion\,; $\pi': E'\rightarrow Y'$ is the natural
		projection.
		
		In particular, if $Y\subset Z$ then $\mathscr{N}_{\tilde{Z}/\tilde{X}} \cong
		\rho'^*\mathscr{N}_{Z/X}\otimes \calO_{\tilde{Z}}(-E')$\,; if $Z$ and $Y$
		intersect transversally in the sense that $\mathscr{N}_{Y,Z/X}=0$, then
		$\mathscr{N}_{\tilde{Z}/\tilde{X}} \cong \rho'^*\mathscr{N}_{Z/X}$.
	\end{lem}
	
	\begin{proof}
		Applying
		Lemma \ref{lem blow-up tangent} twice gives the following commutative
		diagram
		\[
		\xymatrix{
			&0\ar[d] &0\ar[d] &0\ar[d] &\\
			0\ar[r] &\mathscr{T}_{\tilde{Z}}\ar[r]\ar[d]
			&\mathscr{T}_{\tilde{X}}|_{\tilde{Z}}\ar[r]\ar[d]
			&\mathscr{N}_{\tilde{Z}/\tilde{X}}\ar[r]\ar[d] &0\\
			0\ar[r] &(\rho')^*\mathscr{T}_Z\ar[r]\ar[d]
			&(\rho')^*\mathscr{T}_{X}|_Z\ar[r]\ar[d] &(\rho')^*\mathscr{N}_{Z/X}\ar[r]\ar[d]
			&0\\
			0\ar[r] &j'_*\mathscr{E}'\ar[r]\ar[d] &j'_*(\mathscr{E}|_{E'})\ar[r]\ar[d]
			&\mathscr{F}\ar[r]\ar[d] &0\\
			&0 &0 &0 & }
		\]
		where
		$\mathscr{E}':=(\pi')^*\mathscr{N}_{Y'/Z}/\mathcal{O}_{E'}(-1)$. Note that 
		\[
		\mathscr{E}|_{E'} \cong \frac{\pi'^*(\mathscr{N}_{Y/X}|_{Y'})}{\calO_{E'}(-1)}.
		\]
		Then one easily deduce from the last row of the above commutative diagram that
		\[
		\mathscr{F}\cong j'_*\pi'^*\mathscr{N}_{Y,Z/X}
		\]
		and hence the lemma follows.
	\end{proof}
	
	Thirdly, the following lemma explains how the geometric projectivization of the
	inclusion of a sub-bundle into a bundle behaves with respect to the induced
	gradings on the Chow rings.
	
	\begin{lem}\label{lem proj subbundle}
		Let $Y$ be a smooth projective variety endowed with a self-dual multiplicative
		Chow--K\"unneth decomposition. Let $\mathscr{F}$ be a vector bundle
		on $Y$ such that $\mathrm{ch}(\mathscr{F})\in\CH^*(Y)_0$. Let
		$\mathscr{F}'$ be a sub-bundle of $\mathscr{F}$ such that
		$\mathrm{ch}(\mathscr{F}')\in\CH^*(Y)_0$. Let $r+1 =\rk \mathscr{F}$
		and $r'+1=\rk \mathscr{F}'$. Then the natural embedding morphism
		$\varphi: \PP(\mathscr{F}')\hookrightarrow \PP(\mathscr{F})$ is of pure grade
		0,
		where both $\PP(\mathscr{F})$ and $\PP(\mathscr{F}')$ are given the
		self-dual multiplicative Chow--K\"unneth decomposition of a projective bundle as
		in Proposition \ref{prop projbundleselfdualmult}. 
	\end{lem}
	\begin{proof}
		We first prove the weaker conclusion that
		\begin{align*}
		\varphi^*\CH^p(\PP(\mathscr{F}'))_s & \subseteq \CH^p(\PP(\mathscr{F}))_s,\\
		\varphi_*\CH^p(\PP(\mathscr{F}'))_s & \subseteq
		\CH^{p+\delta}(\PP(\mathscr{F}))_s,
		\end{align*}
		where $\delta=r-r'$.
		Let $\xi$ (resp. $\xi'$) be the first Chern class of the relative
		$\calO(1)$-bundle on $\PP(\mathscr{F})$
		(resp. $\PP(\mathscr{F}')$). Then we have $\xi'=\varphi^*\xi$.  Let
		$\pi:\PP(\mathscr{F})\rightarrow Y$ and
		$\pi':\PP(\mathscr{F}')\rightarrow Y'$ be the two morphisms. With
		these notations, we have
		\[
		\varphi^*(\xi^l \cdot \pi^* \CH^{p-l}(Y)_s) = \xi'^l\cdot
		\pi'^*\CH^{p-l}(Y)_s.
		\]
		Together with the explicit description \eqref{eq projbundleformula} of the
		graded components of the
		Chow ring of a projective bundle, this implies that $\varphi^*$ is
		compatible with the gradings on the Chow rings. For
		$\alpha\in\CH^{p-l}(Y)_s$, a direct computation yields
		\[
		\varphi_*(\xi'^l\cdot \pi'^*\alpha) = \varphi_* \varphi^*(\xi^l\cdot
		\pi^*\alpha) = [\PP(\mathscr{F}')]\cdot \xi^l\cdot \pi^*\alpha.
		\]
		Thus to show the compatibility of $\varphi_*$ with the gradings, one
		only needs to verify that $[\PP(\mathscr{F}')]$ belongs to
		$\CH^r(\PP(\mathscr{F}))_0$ as a cycle on $\PP(\mathscr{F})$. But then
		this is clear since $\PP(\mathscr{F}')$ can be defined as the
		vanishing locus of a global section of the bundle
		$\pi^*(\mathscr{F}/\mathscr{F'})\otimes \calO(1)$ on
		$\PP(\mathscr{F})$. Hence the cycle class of $\PP(\mathscr{F}')$ is
		equal to the top Chern class of
		$\pi^*(\mathscr{F}/\mathscr{F}')\otimes \calO(1)$, which sits in
		$\CH^r(\PP(\mathscr{F}))_0$ since all the Chern classes of
		$\mathscr{F}/\mathscr{F}'$ are in the graded-0 part. Indeed, denoting
		$c_i = c_i(\mathscr{F}/\mathscr{F}')$, the cycle class of
		$\PP(\mathscr{F}')$ can be expressed as
		\[
		[\PP(\mathscr{F}')] = \xi^\delta + \pi^*c_1\cdot \xi^{\delta-1} +\cdots
		+ \pi^*c_{\delta-1}\cdot \xi +\pi^*c_\delta,\quad \text{in
		}\CH^*(\PP(\mathscr{F}))_0.
		\]
		
		The product $\PP(\mathscr{F}')\times \PP(\mathscr{F})$ can be viewed as the
		projectivization of $p_2^*\mathscr{F}$ on $\PP(\mathscr{F}')\times X$ and
		$\PP(\mathscr{F}')\times \PP(\mathscr{F}')$ is the projectivization of the
		subbundle $p_2^*\mathscr{F}'\subseteq p_2^*\mathscr{F}$. The same argument shows
		that the action of the morphism
		\[
		\mathrm{id}\times \varphi: \PP(\mathscr{F}')\times \PP(\mathscr{F}')
		\longrightarrow \PP(\mathscr{F}')\times \PP(\mathscr{F})
		\]
		on the Chow groups is compatible with the gradings. In particular
		\[
		\Gamma_{\varphi} = (\mathrm{id}\times \varphi)_*\Delta_{\PP(\mathscr{F}')}
		\]
		is of pure grade 0.
	\end{proof}

	We are now in a position to prove Proposition \ref{prop admissible blow-up}.
	
	\begin{proof}[Proof of Proposition \ref{prop admissible blow-up}] That
		$\mathrm{Bl}_Y( \mathcal{S})$ is complete is obvious.  Indeed, take $Y_1,Y_2\in
		\mathcal{S}$ which are not contained in $Y$ and let
		$\tilde{Y}_1,\tilde{Y}_2\in\mathrm{Bl}_Y(\mathcal{S})$ be their strict
		transforms. If $Y'=Y_1\cap Y_2$ is not contained in $Y$, then the strict
		transform $\tilde{Y}'$ of $Y'$ is the intersection of $\tilde{Y}_1$ and
		$\tilde{Y}_2$ and $\tilde{Y}'\in\mathrm{Bl}_Y(\mathcal{S})$\,; if $Y'\subset Y$,
		then $\tilde{Y}_1$ and $\tilde{Y}_2$ do not meet each other since otherwise
		$\mathscr{T}_{Y_1}$ and $\mathscr{T}_{Y_2}$ do not intersect in constant rank
		and hence $Y'$, as a scheme-theoretic intersection, is not reduced. Note that
		here we use the condition \emph{(ii)} of completeness of $\mathcal{S}$ in an
		essential
		way\,; see Definition \ref{def complete admissible}.
		
		Let $Y_j\in\mathcal{S}$ be such that $Y_j$ is not
		contained in $Y$. By completeness of $\mathcal{S}$, we see that
		$Y_j\cap Y\in \mathcal{S}$. The strict transform $\tilde{Y}_j$ of
		$Y_j$ is the blow-up of $Y_j$ along $Y_j\cap Y$. Since
		$\mathcal{S}$ is admissible, we conclude from Proposition \ref{prop
			adm mult} that $\tilde{Y}_j$ is smooth and naturally endowed with
		a self-dual multiplicative Chow--K\"unneth decomposition. \medskip
		
		We first prove that $\mathrm{Bl}_Y( \mathcal{S})$ is admissible for
		the special case
		\[
		\mathcal{S} = \set{X, Y, Z, Y'=Y\cap Z},
		\]
		where $Z$ is not contained in $Y$. In this case, we still use
		$\tilde{X}$ to denote the blow-up of $X$ with center $Y$. Let
		$\tilde{Z}\subset \tilde{X}$ be the strict transform of $Z$. Then
		$\tilde{Z}$ is simply the blow-up $\mathrm{Bl}_{Y'}(Z)$ of $Z$ along
		$Y'$ and we have
		\[
		\mathrm{Bl}_Y(\mathcal{S}) = \set{\tilde{Z},\tilde{X}}.
		\]
		We have seen that both $\tilde{Z}$ and $\tilde{X}$ have a
		self-dual multiplicative Chow--K\"unneth decomposition. To prove the
		proposition, we still need to verify the following conditions: (a)
		the Chern classes of $\mathscr{N}_{\tilde{Z}/\tilde{X}}$ sit in
		$\CH^*(\tilde{Z})_0$\,; (b) the inclusion morphism $\tilde{Z}\hookrightarrow
		\tilde{X}$ is of pure grade 0. 
		To do that, we will first prove (a) and a weaker conclusion (b') the
		push-forward via the
		embedding $j_{\tilde{Z}}:\tilde{Z}\hookrightarrow \tilde{X}$ are
		compatible with the gradings of the Chow rings. 
		Consider the following diagram
		\begin{equation}
		\xymatrix{
			E\ar[r]^j\ar[d]_\pi &\tilde{X}\ar[d]^\rho & \tilde{Z} 
			\ar[l]^{j_{\tilde{Z}}}\ar[d]^{\rho'}  &
			E'\ar[l]^{j'}\ar[d]^{\pi'} 
			\ar@/_1pc/[ll]_{j_{E'}}\\
			Y\ar[r]^i &X &Z\ar[l]_{i_Z} &Y'\ar[l]_{i'}\ar@/^1pc/[ll]^{i_{Y'}}
		}
		\end{equation}
		where the two extremal squares are the blow-up squares. 
		
		It follows from Lemma \ref{lem normal bundle blow up} that
		\begin{equation*}
		\mathrm{ch}(\mathscr{N}_{\tilde{Z}/\tilde{X}}) 
		=(\rho')^*\mathrm{ch}(\mathscr{N}_{Z/X}) -
		\mathrm{ch}(j'_*\pi'^*\mathscr{N}_{Y,Z/X}).
		\end{equation*}
		Applying the Grothendieck--Riemann--Roch theorem to the morphism
		$j'$ yields
		\[
		\mathrm{ch}(j'_*\pi'^*\mathscr{N}_{Y,Z/X}) = j'_*\left(
		\frac{\pi'^*\mathrm{ch}(\mathscr{N}_{Y,Z/X})}{\mathrm{td}(\mathscr{N}_{E'/\tilde{Z}})}\right).
		\]
		We also note that there is a short exact sequence
		\[
		\xymatrix{
			0\ar[r] &\mathscr{N}_{Y'/Y}\oplus \mathscr{N}_{Y'/Z}\ar[r]
			&\mathscr{N}_{Y'/X}\ar[r] &\mathscr{N}_{Y,Z/X}\ar[r] &0.
		}
		\]
		By assumption the set $\mathcal{S}$ is admissible so that the Chern
		classes of $\mathscr{N}_{Y'/Z}$, $\mathscr{N}_{Y'/Y}$ and $\mathscr{N}_{Y'/X}$
		are in $\CH^*(Y')_0$. Then we conclude that 
		\[
		\mathrm{ch}(\mathscr{N}_{Y,Z/X})\in\CH^*(Y')_0.
		\]
		Using the
		explicit description \eqref{eq blowupformula} of the graded pieces of the Chow
		ring of a smooth blow-up  and the compatibility of $j'_*$ with the grading of
		the Chow groups,
		we see that $\mathrm{ch}(j'_*\pi'^*\mathscr{N}_{Y,Z/X})$ sits in
		$\CH^*(\tilde{Z})_0$. Hence it follows that
		$\mathrm{ch}(\mathscr{N}_{\tilde{Z}/\tilde{X}})$ sits in
		$\CH^*(\tilde{Z})_0$. This establishes (a).
		
		The decomposition of the Chow groups of $\tilde{Z}$ is given, as in
		Proposition \ref{prop multCK blow-up}, by
		\begin{equation}\label{eq grade s Z}
		\CH^p(\tilde{Z})_s = (\rho')^*\CH^p(Z)_s
		\oplus\left(\bigoplus_{l=0}^{r'-2} j'_*(\xi'^l \cdot
		\pi'^*\CH^{p-l-1}(Y')_s)\right),\quad r'=\dim Z- \dim Y'.
		\end{equation}
		Similarly the decomposition of the Chow groups of $\tilde{X}$ is given
		by
		\begin{equation}\label{eq grade s Z tilde}
		\CH^p(\tilde{X})_s = \rho^*\CH^p(X)_s
		\oplus\left(\bigoplus_{l=0}^{r-2} j_*(\xi^l \cdot
		\pi^*\CH^{p-l-1}(Y)_s)\right),\quad r=\dim X- \dim Y.
		\end{equation}
		Let $f:Y'\hookrightarrow Y$ be the inclusion morphism. For any
		$\alpha\in\CH_k(Z)$, we claim that
		\begin{equation}\label{eq gen fulton}
		\rho^*(j_Z)_*\alpha =( j_{\tilde{Z}})_*\rho'^*\alpha + j_*\left(
		c(\mathscr{E})\cap \pi^*f_*(s(Y', Z)\cdot i'^*\alpha) \right)_{k},\qquad
		\text{with }\mathscr{E}=\frac{\pi^*\mathscr{N}_{Y/X}}{\calO_E(-1)},
		\end{equation}
		where $(-)_k$ means taking the $k$-dimensional component and $s(Y',Z)$ is the
		Segre class of $Y'$ in $Z$. (We refer to \cite[\S 4.2]{fulton} for the
		definition and properties of Segre classes.) This claim can be proved as
		follows. We may assume that $\alpha$ is represented by a closed subvariety
		$W\subseteq Z$ which intersects $Y'$ properly. Let $\tilde{W}\subseteq
		\tilde{Z}$ be the strict transform of $W$ and set $W':=W\cap Y'$. Since $W$
		intersects $Y'$ properly, we see that
		\[
		\rho'^*\alpha = \tilde{W},\qquad \text{in }\CH_k(\tilde{Z}).
		\]
		We apply \cite[Theorem 6.7]{fulton} to $W$ with respect to the blow-up morphism $\rho$
		and get
		\[
		\rho^*(j_Z)_*\alpha = (j_{\tilde{Z}})_*\tilde{W} + j_* \big(c(\mathscr{E})\cap
		\pi^*f_*(j_{W'})_*s(W',W)\big)_k,
		\]
		where $j_{W'}:W'\rightarrow Y'$. Since $W$ intersects $Y'$ properly, we have
		$s(W',W)= (j_{W'})^*s(Y',Z)$. Thus we obtain
		\begin{align*}
		\rho^*(j_Z)_*\alpha& = (j_{\tilde{Z}})_*\tilde{W} + j_*\big( c(\mathscr{E})\cap
		\pi^*f_*(j_{W'})_*(j_{W'})^*s(Y',Z) \big)_k\\
		& = (j_{\tilde{Z}})_*\tilde{W} +  j_*\big( c(\mathscr{E})\cap
		\pi^*f_*([W']\cdot s(Y',Z)) \big)_k\\
		& = (j_{\tilde{Z}})_*\rho'^*\alpha + j_*\big( c(\mathscr{E})\cap
		\pi^*f_*(i'^*\alpha\cdot s(Y',Z)) \big)_k,
		\end{align*}
		which settles the claim. We also note that $s(Y',Z)$ is the inverse of
		$c(\mathscr{N}_{Y'/Z})$ and hence lies in the graded-0 part. It is also clear
		that $c(\mathscr{E})$ lies in the graded-0 part. If $\alpha\in\CH^p(Z)_s$, then
		equation \eqref{eq gen fulton} implies that
		\[
		(j_{\tilde{Z}})_* \rho'^*\alpha \in
		\CH^{p+\dim X -\dim
			Z}(\tilde{X})_s.
		\]
		If $\alpha\in \CH^{p-l-1}(Y')_s$, then
		\begin{align*}
		(j_{\tilde{Z}})_*j'_*(\xi'^l\cdot \pi'^*\alpha)
		& = (j_{\tilde{Z}})_* j'_*(j'^*(-E')^l\cdot \pi'^*\alpha)\\
		& = (j_{\tilde{Z}})_*((-E')^l\cdot j'_*\pi'^*\alpha)\\
		& = (j_{\tilde{Z}})_*\Big((j_{\tilde{Z}})^* (-E)^l\cdot
		j'_*\pi'^*\alpha\Big)\\
		& = (-E)^l\cdot (j_{\tilde{Z}})_*j'_*\pi'^*\alpha.
		\end{align*}
		Let $j_{E'/E}:E'\hookrightarrow E$ be the embedding morphism. According to
		Lemma \ref{lem proj subbundle} the push-forward map $(j_{E'/E})_*$ respects
		the gradings of the Chow groups. Thus we obtain
		\[
		(j_{\tilde{Z}})_*j'_*(\xi'^l\cdot \pi'^*\alpha) = (-E)^l\cdot
		(j_{\tilde{Z}})_*j'_*\pi'^*\alpha = (-E)^l\cdot j_*(j_{E'/E})_*
		\pi'^*\alpha\in\CH^*(\tilde{X})_s.
		\]
		This proves (b'). 
		
		We now deal with the stronger condition (b). Note that the set
		\[
		\mathcal{S}' = \{ \tilde{Z}\times X, \tilde{Z}\times Y,\tilde{Z}\times Z,
		\tilde{Z}\times Y'\}
		\]
		is admissible and hence the same argument above shows that the push-forward by
		the inclusion $1\times j_{\tilde{Z}}: \tilde{Z}\times \tilde{Z}\hookrightarrow
		\tilde{Z}\times \tilde{X}$ respects the grading on Chow groups. Hence the graph
		of $j_{\tilde{Z}}$, which equals the push-forward of $\Delta_{\tilde{Z}}$, is of
		pure grade 0. This establishes (b). \medskip
		
		We now move on to the general case. Let $Z_1\subset Z_2$ where
		$Z_1,Z_2\in \mathcal{S}$ are such that none of them is contained in
		$Y$.  Then the conditions for the admissibility for the inclusion
		$\tilde{Z}_1\subset \tilde{Z}_2$ can be checked by applying the above
		special case to
		\[
		\mathcal{S}'' =\set{W'= Y\cap Z_1, W=Y\cap Z_2,Z_1,Z_2}.
		\]
		More precisely, we replace $Y$ by $W$, $Y'$ by $W'$, $Z$ by $Z_1$ and
		$X$ by $Z_2$. This settles the proposition.
	\end{proof}

	\section{Resolving the rational map $X^3 \dashrightarrow X^{[3]}$}
	Let $X$ be a smooth projective variety of dimension $d$. Our strategy for
	proving Theorem
	\ref{thm main} consists in desingularizing the rational map $X^3
	\dashrightarrow X^{[3]}$ by blowing up $X^3$ several times in such a
	way that the variety obtained after each blow-up has a multiplicative
	Chow--K\"unneth decomposition, and in such a way that after the final
	blow-up necessary to resolve the map the Chow--K\"unneth decomposition
	descends to a Chow--K\"unneth decomposition of $X^{[3]}$.
	
	\subsection{A stratification of $X^{[3]}$}\label{sec stratification}
	In this subsection, we explicitly describe some canonical subvarieties of
	$X^{[3]}$. These subvarieties will play important roles in the study of the
	resulting morphism after resolving the rational map $X^3\dashrightarrow
	X^{[3]}$. Such subvarieties also appear in the cellular decomposition considered
	in \cite{cheah} and earlier works.
	
	A general point of $X^{[3]}$ corresponds to a reduced length-3 subscheme of $X$,
	or equivalently to an unordered set of three distinct points of $X$. When two of
	those
	points coincide, one gets a non-reduced subscheme of length 3. For any point
	$\xi\in X^{[3]}$, we denote $Z_\xi\subset X$  the corresponding closed subscheme
	of $X$. Let 
	\begin{equation}\label{eq B1}
	B_1 := \{\xi\in X^{[3]}: Z_\xi\text{ is a non-reduced subscheme of }X\}.
	\end{equation}
	Then $B_1$ is a divisor on $X^{[3]}$. Let
	\begin{equation}\label{eq B2}
	B_2 := \{\xi\in X^{[3]}: Z_\xi \text{ is supported on a single point of }X\}.
	\end{equation}
	Assume that $Z=Z_\xi\subset X$ is a length-3 subscheme supported at a single
	point $x\in X$. Let $\mathscr{I}\subset \calO_X$ be the ideal sheaf of $Z$ and
	let $m_x\subset \calO_X$ be the maximal ideal corresponding to the point $x$.
	There are two cases. 
	
	The first case is $m_x^2\subset\mathscr{I}$, or equivalently the image of
	$\mathscr{I}$ in $m_x/m_x^2$ has dimension $d-2$. In this case, $\mathscr{I}$
	determines a $(d-2)$-dimensional subspace of $\Omega_{X,x}$ which, by
	dualization, gives a 2-dimensional subspace $\mathscr{V}_x$ of
	$\mathscr{T}_{X,x}$. The point $\xi$ is actually determined by $\mathscr{V}_x$
	as follows. Take some local coordinates $\{t_1,t_2,t_3,\ldots, t_d\}$
	of $X$ at the point $x$ such that $\mathscr{V}_x$ is the kernel of
	$\{\mathrm{d}t_3,\ldots,\mathrm{d}t_{d}\}$ at the point $x$. Then we have
	\[
	\mathscr{I}=(t_1^2, t_1t_2, t_2^2, t_3,\ldots,t_d)
	\]
	on an open neighborhood of $x$. If we define
	\begin{equation}\label{eq B3}
	B_3 = \{\xi\in X^{[3]}: Z_\xi\text{ is supported in a single point }x\text{
		such that }m_x^2\subseteq \mathscr{I}_{Z_\xi}\},
	\end{equation}
	then we have the natural isomorphism
	\begin{equation}\label{eq B3 iso}
	B_3\cong \mathrm{Gr}(2,\mathscr{T}_X).
	\end{equation}
	
	The second case is when the image of $\mathscr{I}$ in $m_x/m_x^2$ has dimension
	$d-1$. Dually, this determines a 1-dimensional subspace of $\mathscr{T}_{X,x}$.
	Pick  local coordinates $\{t_1,\ldots,t_d\}$ of $X$ at $x$ such that
	$\mathrm{d}t_1,\ldots,\mathrm{d}t_{d-1}$ generate the image of $\mathscr{I}$ in
	$m_x/m_x^2$. Then we have
	\[
	\mathscr{I} = (t_1+a_1 t_d^2, t_2+a_2 t_d^2,\ldots, t_{d-1}+ a_{d-1}t_d^2,
	t_d^3)
	\]
	for some $a_1,a_2,\ldots, a_{d-1}$ in the base field. This implies that
	$B_2\backslash B_3$ is an $\mathds{A}^{d-1}$-bundle over $\PP(\mathscr{T}_X)$.
	Globalizing the above picture yields
	\[
	B_2\backslash B_3 \cong \PP(\mathscr{E})\backslash
	\PP(\mathscr{T}_{\PP(\mathscr{T}_X)/X}),
	\]
	where $\mathscr{E}$ is a locally free sheaf of rank $d$ on $\PP(\mathscr{T}_X)$
	and $\mathscr{T}_{\PP(\mathscr{T}_X)/X}\subset \mathscr{E}$ is naturally a
	subbundle\,; see \eqref{eq E ses}.

	\subsection{Desingularizing the rational map $X^3 \dashrightarrow
		X^{[3]}$} \label{sec desing} 
	
	In this subsection we first describe how to resolve the map $X^3 \dashrightarrow
	X^{[3]}$ for any smooth projective variety $X$. Then we move on to a careful
	study of the resulting morphism $p:X_3\rightarrow X^{[3]}$ and this will be
	essential to descend the Chow--K\"unneth decomposition from $X_3$ to
	$X^{[3]}$.\medskip
	
	Consider the following smooth closed subvarieties of $X^3$\,:
	\begin{align*}
	\Delta_{12} &: =\{(x,x,y) : x,y \in X\} \subset X^3,\\
	\Delta_{23} &: = \{(y,x,x) : x,y \in X\}\subset X^3,\\
	\Delta_{13} &: =\{(x,y,x) : x,y \in X\}\subset X^3,\\
	\Delta_{123} &: = \{(x,x,x) : x \in X\}\subset X^3.
	\end{align*}
	In practice, we also think of $\Delta_{ij}$ as a morphism
	\begin{equation}\label{eq phi ij}
	\phi_{ij}: X\times X \rightarrow X^3,
	\end{equation}
	such that $\phi_{ij}(x,y)$ is the point whose $i^\text{th}$ and $j^\text{th}$
	coordinates are $x$ and the remaining coordinate is $y$.
	We denote $X_1$ the blow-up of $X_0 := X^3$ along the small diagonal
	$\Delta_{123}$ and $\tilde{\Delta}_{ij}$ the strict transforms of the
	big diagonals $\Delta_{ij}$. Note that $\tilde{\Delta}_{12}$,
	$\tilde{\Delta}_{23}$ and $\tilde{\Delta}_{13}$ are pairwise disjoint
	as subvarieties of $X_1$. We then denote $X_2$ the blow-up of $X_1$
	along the disjoint union of the $\tilde{\Delta}_{ij}$. If $X$ is a surface, then
	the rational map $X_2\dashrightarrow X^{[3]}$ is already a morphism. If $\dim
	X\geq 3$ then, as explained in
	the proof of \cite[Lemma 3.12]{voisin}, the rational map $X_2
	\dashrightarrow X^{[3]}$ is not yet a morphism, but it does become a
	morphism after one more blow-up that we describe now.
	
	Let $E_1\subset X_1$ be the exceptional divisor of the blow-up
	$\rho_1:X_1\rightarrow X_0$, where $X_0:=X^3$, along the small diagonal
	$Y_0:=\Delta_{123}$. The tangent bundle $\mathscr{T}_{X_0}$ restricted
	to $Y_0=\Delta_{123} \cong X$ is naturally isomorphic to
	$\mathscr{T}_X\oplus \mathscr{T}_X \oplus \mathscr{T}_X$. The tangent
	bundle $\mathscr{T}_{Y_0}\cong \mathscr{T}_X$ maps diagonally into
	$\mathscr{T}_{X_0}|_{Y_0}$. Let $\mathscr{N}_{Y_0/X_0}$ be the normal
	bundle of $Y_0$ in $X_0$.  There is a commutative diagram
	\begin{equation} \label{eq comm} \xymatrix{ 0\ar[r]
		&\mathscr{T}_{Y_0}\ar[rr]\ar[d]^{\cong}
		&&\mathscr{T}_{X_0}\ar[rr]\ar[d]^{\cong}
		&&\mathscr{N}_{Y_0/X_0}\ar[r]\ar[d]^{\cong} &0\\
		0\ar[r] &\mathscr{T}_X \ar[rr]^{\delta\qquad} &&\mathscr{T}_X
		\oplus \mathscr{T}_X \oplus \mathscr{T}_X\ar[rr]^{\quad\theta}
		&&\mathscr{T}_X\oplus \mathscr{T}_X\ar[r] &0, }
	\end{equation}
	where $\delta(a)= (a,a,a)$ and $\theta(a,b,c)=(a-b,b-c)$. 
	Hence we get an identification 
	\begin{equation}\label{eq isom E_1}
	E_1\cong \PP(\mathscr{T}_X\oplus \mathscr{T}_X).
	\end{equation}
	Let $W'\subset \PP(\mathscr{T}_X\oplus \mathscr{T}_X)$ be the subvariety
	of all vectors $(a,b)$ such that $a$ and $b$ are colinear in
	$\mathscr{T}_X$. First we note that there is a natural isomorphism
	\[
	\PP^1\times \PP(\mathscr{T}_X)\rightarrow W',\qquad ([s:t],[v])\mapsto
	[(sv,tv)],
	\]
	where $[s:t]$ are the homogeneous coordinates of $\PP^1$ and
	$[v]\in\PP(\mathscr{T}_X)$ is the class of a non-zero vector $v\in
	\mathscr{T}_X$. Under the isomorphism $\mathscr{N}_{Y_0/X_0}\cong
	\mathscr{T}_X\oplus \mathscr{T}_X$, the variety ${W}'$ corresponds to
	a smooth closed subvariety
	\begin{equation} \label{eq W}
	{W}\subset E_1=\PP(\mathscr{N}_{Y_0/X_0}) \subset X_1,
	\end{equation}
	and there is an isomorphism
	\begin{equation}\label{eq isom W}
	W\cong \PP^1\times \PP(\mathscr{T}_X).
	\end{equation}
	Let $\tilde{W}$ be the strict transform of $W$ under the blow-up $X_2
	\rightarrow X_1$. Writing $X_3$ for the smooth blow-up of $X_2$ along
	$\tilde{W}$, we then have (see also Proposition \ref{prop summary})
	
	\begin{prop} \label{prop desing} The rational map $p: X_3\dashrightarrow
		X^{[3]}$ is a generically finite morphism. Moreover, the natural
		action of the symmetric group $\mathfrak{S}_3$ on $X^3$ lifts to $X_3$ and
		the morphism $p: X_3\rightarrow X^{[3]}$ factors through the quotient.
	\end{prop}
	\begin{proof}
		We refer to the proof of \cite[Lemma 3.12]{voisin}. This can also be seen from
		the description of the blow-ups below.
	\end{proof}
	
	For our purpose, we need to understand the morphism $X_3\rightarrow X^{[3]}$ in
	more depth.
	
	The natural subvarieties of $X_1$ have explicit geometric descriptions. For
	example, we first have 
	$$
	\tilde{\Delta}_{ij} \cong  \mathrm{Bl}_{\Delta_X}(X\times X).
	$$
	Note that $\tilde{\Delta}_{ij}\cap E_1\cong \PP(\mathscr{T}_X)$ for all $i,j$
	and under the identification \eqref{eq isom E_1} these can be described as
	follows\,:
	\begin{align}
	W_{12}:=\tilde{\Delta}_{12}\cap E_1 \subset E_1&:\qquad
	\PP(\mathscr{T}_X)\hookrightarrow \PP(\mathscr{T}_X\oplus \mathscr{T}_X),\quad
	[v]\mapsto [0, v]\,; \nonumber \\
	W_{23}:=\tilde{\Delta}_{23}\cap E_1 \subset E_1&:\qquad
	\PP(\mathscr{T}_X)\hookrightarrow \PP(\mathscr{T}_X\oplus \mathscr{T}_X),\quad
	[v]\mapsto [v,0]\,; \label{eq Wij} \\
	W_{13}:=\tilde{\Delta}_{13}\cap E_1 \subset E_1&:\qquad
	\PP(\mathscr{T}_X)\hookrightarrow \PP(\mathscr{T}_X\oplus \mathscr{T}_X),\quad
	[v]\mapsto [v, -v]. \nonumber
	\end{align}
	From the definition of $W$, one immediately sees that $W_{ij}\subset W$. Under
	the identification \eqref{eq isom W} we have
	\[
	W_{12} =[0:1]\times \PP(\mathscr{T}_X),\quad W_{23} =[1:0]\times
	\PP(\mathscr{T}_X),\quad  W_{13} =[1:-1]\times \PP(\mathscr{T}_X).
	\]
	
	Let $E_{ij}\subset X_2$ be the exceptional divisor sitting above
	$\tilde{\Delta}_{ij}$. Thinking of $\Delta_{ij}$ as being isomorphic to $X
	\times X$ as in  \eqref{eq phi ij}, we have
	\[
	\mathscr{N}_{\Delta_{ij}/X_0} = p_1^* \mathscr{T}_X,
	\]
	where $p_1:X\times X\rightarrow X$ is the projection onto the first factor. It
	follows from Lemma \ref{lem normal bundle blow up} that 
	\[
	\mathscr{N}_{\tilde{\Delta}_{ij}/X_1} \cong \tilde{p}_1^*\mathscr{T}_X \otimes
	\calO_{\tilde{\Delta}_{ij}} (-E_{\Delta}),
	\]
	where we identify $\tilde{\Delta}_{ij}$ with $\mathrm{Bl}_{\Delta_X}(X\times X)$
	and $\tilde{p}_1$ is the composition of $p_1:X\times X\rightarrow X$ with the
	blow-up morphism $\mathrm{Bl}_{\Delta_X}(X\times X)\rightarrow X\times X$. It
	follows that $E_{ij}\cong \PP(\tilde{p}_1^*\mathscr{T}_X)$, or equivalently
	$E_{ij}$ is the blow-up of $\PP(\mathscr{T}_X)\times X$ along the graph of
	$\pi:\PP(\mathscr{T}_X)\rightarrow X$. Let $E'_1\subset X_2$ be the strict
	transform of $E_1$. Then $E'_1$ is the blow-up of $E_1$ along the disjoint union
	of the subvarieties $W_{ij}$. One easily sees that
	\[
	E'_{1,ij}:=E'_1 \cap E_{ij}\cong \PP(\mathscr{T}_X)\times_X\PP(\mathscr{T}_X).
	\]
	Furthermore, $E'_{1,ij}$ is the exceptional divisor of the blow-up
	$E'_1\rightarrow E_1$ and also the exceptional divisor of the blow-up
	$E_{ij}\cong\mathrm{Bl}_{\Gamma_\pi} (\PP(\mathscr{T}_X)\times X)\rightarrow
	\PP(\mathscr{T}_X)\times X$. Each blow-up contracts one of the two
	$\PP(\mathscr{T}_X)$-factors respectively. We also see that the strict transform
	$\tilde{W}\subset E'_1$ is isomorphic to $W$.

	Under the identifications  \eqref{eq isom E_1} and \eqref{eq isom W}, the
	inclusion $W\subset E_1$ is induced by
	\[
	p_1^* \calO_{\PP^1}(-1)\otimes
	p_2^*\calO_{\PP(\mathscr{T}_X)}(-1)\hookrightarrow \C^2\otimes \mathscr{T}_X
	\cong \mathscr{T}_X^{\oplus 2},
	\]
	where $p_i$ are the two projections from $W=\PP^1\times \PP(\mathscr{T}_X)$ to
	its factors.
	In particular, 
	\begin{equation}\label{eq O1 on W}
	\calO_{E_1}(1)|_W = p_1^* \calO_{\PP^1}(1)\otimes
	p_2^*\calO_{\PP(\mathscr{T}_X)}(1). 
	\end{equation}
	We have the following natural isomorphisms
	\begin{align*}
	\mathscr{T}_{W/X} & \cong \mathscr{T}_{\PP^1}\oplus
	\mathscr{T}_{\PP(\mathscr{T}_X)/X} = \frac{\C^2\otimes
		\calO_{\PP^1}(1)}{\calO_{\PP^1}} \oplus \frac{\mathscr{T}_X\otimes
		\calO_{\PP(\mathscr{T}_X)}(1)}{\calO_{\PP(\mathscr{T}_X)}},\\
	\mathscr{T}_{E_1/X}|_{W}& \cong \frac{\mathscr{T}_X^{\oplus 2}\otimes
		\calO_{E_1}(1)|_W}{\calO_W} = \frac{\C^2\otimes \calO_{\PP^1}(1) \otimes
		\mathscr{T}_X\otimes \calO_{\PP(\mathscr{T}_X)}(1) } {\calO_W}.
	\end{align*}
	Here all the sheaves are viewed as sheaves on $W$ via pulling back the
	corresponding ones. From this, we see that the normal bundle of $W$ in $E_1$ is
	identified as
	\begin{equation}\label{eq normal W E1}
	\mathscr{N}_{W/E_1} \cong p_1^* \mathscr{T}_{\PP^1} \otimes
	p_2^*\mathscr{T}_{\PP(\mathscr{T}_X)/X}.
	\end{equation}
	Since $E'_1$ is the blow-up of $E_1$ along $W_{ij}$, where $W_{ij}$ are divisors
	on $W$, we conclude by Lemma \ref{lem normal bundle blow up} that
	\begin{equation}\label{eq1}
	\mathscr{N}_{\tilde{W}/E'_1} = \mathscr{N}_{W/E_1}\otimes
	\calO_W(-W_{12}-W_{23}-W_{13}) \cong p_1^*\calO_{\PP^1}(-1)\otimes
	p_2^*\mathscr{T}_{\PP(\mathscr{T}_X)/X}.
	\end{equation}
	Since $E'_1$ meets $\tilde{\Delta}_{ij}$ transversally, again by Lemma \ref{lem
		normal bundle blow up}, we see that $\mathscr{N}_{E'_1/X_2}$ is isomorphic to
	the pull-back of $\mathscr{N}_{E_1/X_1}$. Hence we get
	\begin{equation} \label{eq2}
	\mathscr{N}_{E'_1/X_2}|_{\tilde{W}} \cong  \mathscr{N}_{E_1/X_1}|_{W}=
	\calO_{E_1}(-1)|_W \cong  p_1^*\calO_{\PP^1}(-1)\otimes
	p_2^*\calO_{\PP(\mathscr{T}_X)}(-1). 
	\end{equation}
	The normal bundle of $\tilde{W}$ in $X_2$ fits into the following short exact
	sequence
	\[
	\xymatrix{
		0\ar[r] & \mathscr{N}_{\tilde{W}/E'_1}\ar[r] &\mathscr{N}_{\tilde{W}/X_2}\ar[r]
		&\mathscr{N}_{E'_1/X_2}|_{\tilde{W}}\ar[r] &0.
	}
	\]
	By the isomorphisms \eqref{eq1} and \eqref{eq2}, the short exact sequence can be
	rewritten as
	\begin{equation}\label{eq normal ses}
	\xymatrix{
		0\ar[r] &p_1^*\calO_{\PP^1}(-1)\otimes p_2^*\mathscr{T}_{\PP(\mathscr{T}_X)/X}
		\ar[r] 
		& \mathscr{N}_{\tilde{W}/X_2} \ar[r] &p_1^*\calO_{\PP^1}(-1)\otimes
		p_2^*\calO_{\PP(\mathscr{T}_X)}(-1) \ar[r] & 0.
	}
	\end{equation}
	This defines an element in
	\begin{align*}
	& \mathrm{Ext}^1_{\calO_W} (p_1^*\calO_{\PP^1}(-1)\otimes
	p_2^*\calO_{\PP(\mathscr{T}_X)}(-1), p_1^*\calO_{\PP^1}(-1)\otimes
	p_2^*\mathscr{T}_{\PP(\mathscr{T}_X)/X}) \\
	& \quad =
	\mathrm{Ext}_{\calO_{\PP(\mathscr{T}_X)}}(\calO_{\PP(\mathscr{T}_X)}(-1),
	\mathscr{T}_{\PP(\mathscr{T}_X)/X}).
	\end{align*}
	Hence there is a locally free sheave $\mathscr{E}$ on $\PP(\mathscr{T}_X)$ such
	that
	\[
	\mathscr{N}_{\tilde{W}/X_2} \cong p_1^*\calO_{\PP^1}(-1)\otimes
	p_2^*\mathscr{E}.
	\]
	Furthermore, the short exact sequence \eqref{eq normal ses} shows that
	$\mathscr{E}$ naturally fits into the following short exact sequence 
	\begin{equation}\label{eq E ses}
	\xymatrix{
		0\ar[r] &\mathscr{T}_{\PP(\mathscr{T}_X)/X}\ar[r] &\mathscr{E} \ar[r]
		&\calO_{\PP(\mathscr{T}_X)}(-1)\ar[r] &0.
	}
	\end{equation}
	
	Let $E_W\subset X_3$ be the exceptional divisor of the blow-up $X_3\rightarrow
	X_2$. Then we have
	\begin{equation}\label{eq isom E_W}
	E_W = \PP(\mathscr{N}_{\tilde{W}/X_2}) =\PP^1 \times \PP(\mathscr{E}).
	\end{equation}
	Let $E''_1\subset X_3$ be the strict transform of $E'_1\subset X_2$. Note that
	$E''_1$ is the blow-up of $E'_1$ along $\tilde{W}$ with exceptional divisor
	$E''_W=\PP(\mathscr{N}_{\tilde{W}/E'_1})=\PP^1\times
	\PP(\mathscr{T}_{\PP(\mathscr{T}_X)/X})$. We also have $E''_W=E''_1\cap E_W$ and
	the inclusion $E''_W\subset E_W$ is induced by the subsheaf
	$\mathscr{T}_{\PP(\mathscr{T}_X)/X}\subset \mathscr{E}$ in \eqref{eq E ses}. The
	morphism $p: X_3\rightarrow X^{[3]}$ contracts the two divisors $E_W$ and
	$E''_1$ in the following way. The morphism $p:X_3\rightarrow X^{[3]}$ contracts
	$E''_1$ onto $B_3\cong \mathrm{G}(2,\mathscr{T}_X)$ and it contracts
	$E_W\backslash E''_W$ onto $B_2\backslash B_3$. Note that 
	\[
	E_W\backslash E''_W \cong \PP^1\times \PP(\mathscr{E})\backslash
	\PP(\mathscr{T}_{\PP(\mathscr{T}_X)/X})
	\]
	and 
	\[
	B_2\backslash B_3 \cong \PP(\mathscr{E})\backslash
	\PP(\mathscr{T}_{\PP(\mathscr{T}_X)/X}).
	\]
	The morphism $p$ simply contracts the $\PP^1$-factor of $E_W\backslash E''_W$.
	\medskip
	
	We summarize the above discussion in the following proposition.
	
	\begin{prop}\label{prop summary}
		The morphism $p: X_3\rightarrow X^{[3]}$ contracts two divisors $E''_1$ and
		$E_W$ respectively to $B_3$ and $B_2$. The ramification divisors outside
		$E''_1\cup E_W$ are $\tilde{E}_{ij}$ (strict transform of $E_{ij}$), $1\leq
		i<j\leq 3$. Furthermore, we have
		\begin{align*}
		E_W\cap \tilde{E}_{ij} & = \{t_{ij}\}\times \PP(\mathscr{E}) \hookrightarrow
		E_W=\PP^1\times \PP(\mathscr{E}),\quad t_{12}=[0:1],\ t_{23}=[1:0],\
		t_{13}=[1:-1]\,;\\
		E''_1\cap E_W & = \PP^1\times \PP(\mathscr{T}_{\PP(\mathscr{T}_X)/X})\,;\\
		E''_1\cap \tilde{E}_{ij} & = \mathrm{Bl}_{\Delta}(\PP(\mathscr{T}_X)\times_X
		\PP(\mathscr{T}_X)),
		\end{align*}
		where $\Delta=\Delta_{\PP(\mathscr{T}_X)/X}$ is the diagonal of
		$\PP(\mathscr{T}_X)$ relative to $X$. 
	\end{prop}
	
	\begin{proof}
		We only need to deal with the intersection with $\tilde{E}_{ij}$. By
		construction, we have
		\[
		\tilde{W}_{ij}:= \tilde{W}\cap E_{ij} = \{t_{ij}\}\times
		\PP(\mathscr{T}_X)\hookrightarrow \tilde{W}=\PP^1\times \PP(\mathscr{T}_X).
		\]
		The isomorphism $E_{ij}\cong \mathrm{Bl}_{\Gamma_\pi}(\PP(\mathscr{T}_X)\times
		X)$ gives an exceptional divisor $E'_{1,ij}\cong\PP(\mathscr{T}_X)\times_X
		\PP(\mathscr{T}_X)$. Then $\tilde{W}_{ij}$ is simply the relative diagonal
		$\Delta$ in the exceptional divisor $E'_{1,ij}$. The successive closed
		immersions $\tilde{W}\subset E'_{1,ij}\subset E_{ij}$ gives a short exact
		sequence
		\[
		\xymatrix{
			0\ar[r] & \mathscr{N}_{\tilde{W}_{ij}/E'_{1,ij}}\ar[rr]
			&&\mathscr{N}_{\tilde{W}_{ij}/E_{ij}}\ar[rr]
			&&\mathscr{N}_{E'_{1,ij}/E_{ij}}|_{\tilde{W}_{ij}}\ar[r] &0
		}
		\]
		which is equivalent to \eqref{eq E ses}. Hence we see that
		\[
		E_W\cap \tilde{E}_{ij} = \{t_{ij}\}\times \PP(\mathscr{E}).
		\]
		The last identity of the proposition follows from 
		\[
		E''_1\cap \tilde{E}_{ij} = \mathrm{Bl}_{\tilde{W}_{ij}}(E'_{1,ij}) =
		\mathrm{Bl}_{\Delta}(\PP(\mathscr{T}_X)\times_X \PP(\mathscr{T}_X)).
		\]
	\end{proof}

	\section{Multiplicative Chow--K\"unneth decomposition for $X^{[3]}$}
	
	\subsection{Self-dual $\mathfrak{S}_3$-invariant multiplicative
		Chow--K\"unneth decomposition on $X_3$}
	Let us consider a smooth projective variety $X$ that is equipped with
	a multiplicative Chow--K\"unneth decomposition
	\[
	\Delta_X = \pi^0_X +\pi^1_X +\cdots + \pi^{2d}_X \quad \text{in} \
	\CH_d(X\times X)
	\] that is self-dual, meaning that $\pi^{2d-i}_X$ is the transpose of
	$\pi_X^i$ for all $i$. Here $d=\dim X$.  Then the product Chow--K\"unneth
	decomposition, $$\pi_{X^3}^i := \sum_{i_1+i_2+i_3 = i} \pi_X^{i_1}
	\otimes \pi_X^{i_2} \otimes \pi_X^{i_3}, \quad 0\leq i\leq 6d$$ on
	$X^3$ is clearly self-dual (\emph{cf.} Proposition \ref{prop
		self-dualproduct})\,; it is also multiplicative by
	Proposition \ref{prop multstable}. The symmetric group $\mathfrak{S}_3$ acts on
	$X^3$ and the product Chow--K\"unneth decomposition $\{\pi^i_{X^3}\}$
	is clearly $\mathfrak{S}_3$-invariant. 
	
	We now take up the notations of Paragraph \ref{sec desing} and we
	assume that $X$ is a smooth projective variety of dimension $d$
	equipped with a self-dual multiplicative Chow--K\"unneth
	decomposition $\{\pi_X^i\}$, with the additional property that the Chern classes
	$c_i(X)$ sit in $\CH^i(X)_0$ for all $i$. The goal of this paragraph is to show
	that the variety $X_3$ obtained   in Paragraph \ref{sec desing} by  resolving
	the map $X^3\dashrightarrow X^{[3]}$  is naturally equipped with a self-dual
	multiplicative Chow--K\"unneth decomposition that is
	$\mathfrak{S}_3$-invariant.\medskip
	
	We define 
	\begin{equation*} 
	\mathcal{S}_0 := \{\Delta_{123}, \Delta_{12},
	\Delta_{13}, \Delta_{23}, X_0 \}
	\end{equation*}
	and we equip $\Delta_{123} \cong X$ with the Chow--K\"unneth decomposition
	$\{\pi^i_X\}$, $\Delta_{ij} \cong X^2$ with the product Chow--K\"unneth
	decomposition, and  $X_0 = X^3$ with the product Chow--K\"unneth decomposition.
	These Chow--K\"unneth decompositions are all self-dual and multiplicative by
	Proposition \ref{prop multstable}. Moreover, the Chow--K\"unneth decompositions
	for $\Delta_{123}$ and $X_0$ are $\mathfrak{S}_3$-invariant, while the
	Chow--K\"unneth decompositions $\{\pi^l_{\Delta_{ij}}\}$ for the big diagonals
	$\Delta_{ij}$ satisfy $g\cdot \pi^l_{\Delta_{ij}} = \pi^l_{\Delta_{g(i)g(j)}}$
	for all $g \in \mathfrak{S}_3$ (here, $g(i)$ is the action of $\mathfrak{S}_3$
	on the set $\{1,2,3\}$).

	\begin{lem}\label{lem S_0 admissible}
		The set $\mathcal{S}_0$ is  complete and admissible in
		the sense of Definition \ref{def complete admissible}.
	\end{lem}
	
	\begin{proof}
		That the set $\mathcal{S}_0$ is complete is obvious.
		It is also admissible. Indeed, 
		(a) the
		Chern classes of $X$ belong to $\CH^*(X)_0$ by assumption and it is then
		apparent
		that the Chern classes of the various normal bundles among the pairs
		of elements $(Y_{j_1},Y_{j_2})$ of $\mathcal{S}_0$ such that $ Y_{j_1}
		\subset Y_{j_2}$ have their Chern classes sitting in
		$\CH^*(Y_{j_1})_0$. Also the multiplicative Chow--K\"unneth
		decomposition $\{\pi_X^i\}$ is assumed to be self-dual, it follows  that all
		possible diagonal inclusions are of pure grade 0 since they are obtained from
		the graph of the identity morphism. Hence we have (b) the
		various inclusions $ Y_{j_1} \subset Y_{j_2}$ are of pure grade 0.
	\end{proof}

	Taking $Y_0:=\Delta_{123}$ and denoting $X_1=\mathrm{Bl}_{Y_0}(X_0)$
	the blow-up of $X_0$ along $Y_0$, one sees that
	\[
	\mathrm{Bl}_{Y_0}(\mathcal{S}_0) = \set{\tilde{\Delta}_{12},
		\tilde{\Delta}_{13}, \tilde{\Delta}_{23}, X_1}.
	\]
	Note that the strict transforms of the $\Delta_{ij}$\,:
	$\tilde{\Delta}_{12}$, $\tilde{\Delta}_{23}$ and $\tilde{\Delta}_{13}$
	are pairwise disjoint as subvarieties of $X_1$. These are equipped with a
	self-dual multiplicative Chow--K\"unneth decomposition induced by that of
	$\Delta_{ij}$ and $\Delta_{123}$ via Proposition \ref{prop multCK blow-up}. By
	Proposition
	\ref{prop admissible blow-up}, the set
	$\mathrm{Bl}_{Y_0}(\mathcal{S}_0)$ is complete and admissible. The self-dual
	multiplicative Chow--K\"unneth decompositions of the $\tilde{\Delta}_{ij}$ fit
	together to give a self-dual multiplicative Chow--K\"unneth decomposition of
	$\tilde{\Delta}_{12} \sqcup\tilde{\Delta}_{13}
	\sqcup \tilde{\Delta}_{23}$ that is $\mathfrak{S}_3$-invariant. 

	Let $E_1\subset X_1$ be the exceptional divisor of the blow-up
	$\rho_1:X_1\rightarrow X_0$. In the discussion following \eqref{eq
		comm}, we defined a smooth closed subvariety $W\subset E_1 \subset X_1$. Let
	$W_{ij}\subset
	\PP(\mathscr{N}_{Y_0/X_0}) =E_1$ be  the closed subvarieties
	obtained in \eqref{eq Wij} as the geometric projectivizations of the following
	sub-bundles $\mathscr{W}_{ij}$ of $\mathscr{N}_{Y_0/X_0}$
	\begin{align*}
	\mathscr{W}_{12} & \cong \mathscr{T}_X \hookrightarrow \mathscr{T}_X
	\oplus
	\mathscr{T}_X, \qquad v\mapsto (0,v)\,;\\
	\mathscr{W}_{23} & \cong \mathscr{T}_X\hookrightarrow
	\mathscr{T}_X\oplus
	\mathscr{T}_X, \qquad v\mapsto (v,0)\,; \\
	\mathscr{W}_{13} & \cong \mathscr{T}_X\hookrightarrow
	\mathscr{T}_X\oplus \mathscr{T}_X, \qquad v\mapsto (v,-v).
	\end{align*}
	The isomorphism $\mathscr{W}_{ij}\cong \mathscr{T}_X$ gives that the Chern
	classes
	$c_p(\mathscr{W}_{ij}) = c_p(X)$ sit in $\CH^p(X)_0$. Hence the subvarieties
	$W_{ij}$ are
	naturally endowed with a self-dual multiplicative Chow--K\"unneth decomposition
	thanks to Proposition \ref{prop projbundleselfdualmult}. Recall that under the
	isomorphism \eqref{eq isom
		W}, we have
	\begin{equation}\label{eq isom Wij}
	W_{12} \cong [0:1]\times \PP(\mathscr{T}_X),\quad W_{23} \cong[1:0]\times
	\PP(\mathscr{T}_X),\quad  W_{13} \cong[1:-1]\times \PP(\mathscr{T}_X),
	\end{equation}
	and note that scheme-theoretically
	\[
	W_{ij} = W\cap \tilde{\Delta}_{ij}.
	\]
	We also observe that the variety $W$ is stable under the action of
	$\mathfrak{S}_3$. Let us describe explicitly the action of $\mathfrak{S}_3$ on
	$W$. Note that the isomorphism \eqref{eq isom E_1} factors as
	\[
	E_1=\PP\left(\frac{\mathscr{T}_X \oplus \mathscr{T}_X \oplus
		\mathscr{T}_X}{\mathscr{T}_X} \right) \cong \PP(\mathscr{T}_X\oplus
	\mathscr{T}_X),
	\]
	where the second isomorphism is induced by the homomorphism $\theta$ in
	\eqref{eq comm}. The action of $\mathfrak{S}_3$ on $E_1$ is given by permuting
	the three $\mathscr{T}_X$-factors of the numerator in
	$\PP({\mathscr{T}_X^{\oplus 3}} / {\mathscr{T}_X})$. Under the identification
	\eqref{eq isom W}, the natural inclusion $W\subset E_1$ is given by
	\[
	\PP^1\times \PP(\mathscr{T}_X)\hookrightarrow \PP(\mathscr{T}_X\oplus
	\mathscr{T}_X), \qquad ([s:t],[v]) \mapsto [sv :tv]. 
	\]
	Then the $\mathfrak{S}_3$-action is easy to understand. For example let
	$g=(12)\in\mathfrak{S}_3$ be the transposition that permutes 1 and 2. Then we
	have
	\[
	\xymatrix{
		(a,b,c)\ar[r]^\theta\ar[d]^g &(a-b, b-c)\ar[d]^g\\
		(b,a,c)\ar[r]^\theta &(b-a,a-c)
	}
	\] 
	Hence the action of $g$ on $\PP(\mathscr{T}_X\oplus \mathscr{T}_X)$ is given by
	$[v_1 : v_2]\mapsto [-v_1 : v_1+v_2]$. The following diagram is then commutative
	\[
	\xymatrix{
		\PP^1\times \PP(\mathscr{T}_X) \ar@{^(->}[r]\ar[d]^g &\PP(\mathscr{T}_X\oplus
		\mathscr{T}_X) \ar[d]^g
		&([s:t],[v])\ar@{|->}[r]\ar@{|->}[d]^g & [sv:tv]\ar@{|->}[d]^g\\
		\PP^1\times \PP(\mathscr{T}_X) \ar@{^(->}[r] &\PP(\mathscr{T}_X\oplus
		\mathscr{T}_X) &([-s:s+t],[v]) \ar@{|->}[r] &[-sv : (s+t)v].
	}
	\] 
	Hence the action of the element $g = (12)$ on $\PP^1\times \PP(\mathscr{T}_X)$
	is induced by the action on $\PP^1$ given by $[s:t]\mapsto [-s:s+t]$. Repeating
	the above argument for each element of $\mathfrak{S}_3$, we conclude that under
	the isomorphism \eqref{eq isom W}, $W$ is isomorphic to $\PP^1 \times
	\PP(\mathscr{T}_X)$ and this isomorphism is made  $\mathfrak{S}_3$-equivariant
	by letting $\mathfrak{S}_3$ act trivially on 
	$\PP(\mathscr{T}_X)$ and act on $\PP^1$ as the subgroup of $\mathrm{Aut}(\PP^1)$
	that permutes $\{[1:0],[1:-1],[0:1]\}$. 
	Thus the subvariety $W_{12}\sqcup W_{13}\sqcup W_{23}$ is stable under the
	action of $\mathfrak{S}_3$. Let us endow $\PP^1$ with the
	$\mathfrak{S}_3$-invariant (for the action described above) self-dual 
	multiplicative Chow--K\"unneth decomposition $\{\pi_{\PP^1}^0:= \{[0:1]\} \times
	X, \pi_{\PP^1}^2:=X \times \{[0:1]\}\}$ and $\PP(\mathscr{T}_X)$ with the
	$\mathfrak{S}_3$-invariant (for the trivial action) multiplicative self-dual
	Chow--K\"unneth decomposition given by Proposition \ref{prop
		projbundleselfdualmult}. Then clearly, the obvious Chow--K\"unneth decomposition
	on  $W$ and $W_{12}\sqcup W_{13}\sqcup W_{23}$ are $\mathfrak{S}_3$-invariant,
	self-dual, and multiplicative.
	
	We define the set $\mathcal{S}_1$ of closed subvarieties of $X_1$ to be
	\begin{equation*} 
	\mathcal{S}_1:=\set{X_1,E_1, \tilde{\Delta}_{12} \sqcup\tilde{\Delta}_{13}
		\sqcup \tilde{\Delta}_{23}, W, W_{12}\sqcup W_{13} \sqcup W_{23}}.
	\end{equation*}
	
	\begin{lem}\label{lem S_1 admissible}
		The set $\mathcal{S}_1$ is  complete in
		the sense of Definition \ref{def complete admissible}. Moreover, all elements of
		$\mathcal{S}_1$ are stable under the action of $\mathfrak{S}_3$ and are endowed
		with a $\mathfrak{S}_3$-invariant self-dual Chow--K\"unneth decomposition that
		make $\mathcal{S}_1$ admissible.
	\end{lem}
	\begin{proof}
		The set $\mathcal{S}_1$ is obviously complete and we already saw that all the
		subvarieties in $\mathcal{S}_1$ are
		smooth and endowed with a $\mathfrak{S}_3$-invariant self-dual
		multiplicative Chow--K\"unneth
		decomposition. Hence it remains to verify condition \emph{(ii)} in the
		definition of an admissible set (Definition \ref{def complete
			admissible}). This condition involves the normal bundles and all possible
		inclusions of varieties in $\mathcal{S}_1$. The inclusions
		$W_{ij}\hookrightarrow W$ are taken care of by the isomorphisms
		\eqref{eq isom W} and \eqref{eq isom Wij}. For the inclusion of
		$W_{ij}\hookrightarrow \tilde{\Delta}_{ij}$, we only need to note
		that, under the
		isomorphism $\tilde{\Delta}_{ij}\cong \mathrm{Bl}_{\Delta_X}(X\times X)$,
		the subvariety 
		\[
		W_{ij}\cong\PP(\mathscr{T}_X)\hookrightarrow \mathrm{Bl}_{\Delta_X}(X\times X)
		\]
		is simply the exceptional divisor. The inclusion $E_1\subset X_1$ is also the
		exceptional divisor. Both of the above two cases are taken care of by
		Proposition
		\ref{prop multCK
			blow-up} and Remark \ref{rmk blow-up diagram degree 0}. It remains to deal with
		the inclusion $W\subset E_1$. Note that the
		normal bundle of $W\subset E_1$ is given by \eqref{eq normal W E1}. Let $\iota:
		W\hookrightarrow E_1$ be the inclusion.  Let $\xi\in\CH^1(E_1)$ be the Chern
		class of $\calO_{E_1}(1)$ and let $\xi_1 =
		p_1^*c_1(\calO_{\PP^1}(1))\in\CH^1(W)$ and
		$\xi_2=p_2^*c_1(\calO_{\PP(\mathscr{T}_X)}(1))\in\CH^1(W)$ be the pull-backs of
		the corresponding Chern classes. Equation \eqref{eq O1 on W} implies
		\[
		\iota^* \xi = \xi_1+\xi_2 
		\]
		and hence it follows that $\iota^*$ is compatible with the gradings on the Chow
		groups. On $E_1\cong\PP(\mathscr{T}_X\oplus \mathscr{T}_{X})$, the natural
		homomorphism $\calO_{E_1}(-1) \subset \pi_2^*(\mathscr{T}_X\oplus
		\mathscr{T}_X)$ gives rise to two homomorphism,
		\[
		\varphi_1 : \calO_{E_1}(-1)\hookrightarrow \pi_2^*\mathscr{T}_X\qquad
		\text{and}\qquad
		\varphi_2 : \calO_{E_1}(-1)\hookrightarrow \pi_2^*\mathscr{T}_X,
		\]
		by composing the two projections. Here
		$\pi_2:\PP(\mathscr{T}_X\oplus\mathscr{T}_X)\rightarrow X$ is the structure
		morphism. Let $\mathscr{Q}_1$ be the quotient of $\varphi_1$. Then $\varphi_2$
		induces a homomorphism $\calO_{E_1}(-1)\rightarrow \mathscr{Q}_1$. This can be
		viewed as a section $s$ of the sheaf $\mathscr{Q}_1\otimes \calO_{E_1}(1)$ whose
		vanishing locus is exactly $W$. It follows that
		\[
		\iota_*[W] = c_{d-1}(\mathscr{Q}_1\otimes \calO_{E_1}(1))\in \CH^{d-1}(E_1)_0.
		\]
		It follows that
		\[
		\iota_*(\xi_2^l+l\xi_2^{l-1}\xi_1) =\iota_*(\xi_1+\xi_2)^l=\iota_*\iota^*\xi^l =
		\iota_*[W]\cdot \xi^l\in \CH^*(E_1)_0.
		\]
		We also note that Lemma \ref{lem proj subbundle} implies that
		$\iota_*(\xi_1\xi_2^{l-1})\in\CH^*(E_1)_0$. It follows that
		\[
		\iota_*(\xi_1^a\xi_2^b)\in\CH^*(E_1)_0,\quad \forall a,b\geq 0.
		\]
		Thus one easily concludes that $\iota_*$ respects the gradings on the Chow
		groups. The above argument also works for the inclusion
		\[
		1\times \iota: Z\times W\longrightarrow Z\times E_1,
		\]
		where $Z$ is an arbitrary smooth projective variety with a self-dual
		multiplicative Chow--K\"unneth decomposition. We take $Z$ to be $W$. Then the
		graph of $\iota$ is the push-forward of $\Delta_W$ via the morphism $1\times
		\iota$. Hence $\iota$ is of pure grade 0. By Remark \ref{rmk admissible
			composition}, this proves that $\mathcal{S}_1$ is
		admissible.
	\end{proof}
	
	Let now (as in paragraph \ref{sec desing})
	$X_2:=\mathrm{Bl}_{Y_1}(X_1)$, where $Y_1 :=  \tilde{\Delta}_{12}
	\sqcup\tilde{\Delta}_{13}
	\sqcup \tilde{\Delta}_{23}$. By Proposition \ref{prop admissible blow-up},
	we see that the set
	\[
	\mathcal{S}_2=\set{\tilde{W},E'_1, X_2},
	\]
	where $\tilde{W}$ is the strict transform of $W$, is admissible. Moreover,
	$\tilde{W}$ is equipped with a $\mathfrak{S}_3$-invariant self-dual
	multiplicative Chow--K\"unneth decomposition. \medskip
	
	Let finally $X_3$ be the blow-up of $X_2$ along $Y_2:=\tilde{W}$. We then have 
	
	\begin{prop} \label{prop X3}
		The variety $X_3$ admits a $\mathfrak{S}_3$-invariant self-dual
		multiplicative Chow--K\"unneth decomposition.
	\end{prop}
	\begin{proof}
		The proposition follows immediately from the discussion above and from
		Proposition \ref{prop multCK blow-up} applied to $X_2$ blown up along $Y_2$.
	\end{proof}

	\subsection{Proof of Theorem \ref{thm main}}
	In section 4, we resolved the map $X^3 \dashrightarrow X^{[3]}$ into a morphism
	$p : X_3 \rightarrow X^{[3]}$. The smooth projective variety $X_3$ is naturally
	equipped with a $\mathfrak{S}_3$-action, and the morphism $p$ is generically the
	quotient morphism, meaning that there exists a $\mathfrak{S}_3$-invariant open
	subset $U\subset X_3$ such $p|_U$ is the quotient morphism. Proposition
	\ref{prop X3} shows that $X_3$ is naturally
	endowed with a $\mathfrak{S}_3$-invariant self-dual multiplicative
	Chow--K\"unneth decomposition, whenever $X$ is endowed with a self-dual
	multiplicative Chow--K\"unneth decomposition with the additional property that
	$c_p(X)$ sit in $\CH^p(X)_0$ for all $p$. 
	The following technical lemma  gives sufficient geometric conditions on the
	locus that is contracted by a generic quotient (for the action of a finite group
	$G$) morphism $p$  for a $G$-invariant self-dual multiplicative
	Chow--K\"unneth decomposition to descend along $p$.
	\begin{lem}\label{lem technical}
		Let $G$ be a finite group and let $X$ be a smooth projective variety with a
		$G$-action. Assume that $X$ is endowed with a $G$-invariant self-dual
		multiplicative Chow--K\"unneth decomposition $\{\pi_X^i\}_{i=0}^{2d}$, where
		$d=\dim X$. Let $p:X\rightarrow Y$ be a generically finite morphism such that
		the following conditions hold.
		\begin{enumerate}[(a)]
			\item The morphism $p$ is generically the quotient morphism of the $G$-action.
			Namely, there exists a $G$-invariant open subset $U$ of $X$ such that $p|_U$ is
			the quotient morphism $U\rightarrow U/G$.
			
			\item The morphism $p$ contracts two smooth divisors $D_i\subset X$, $i=1,2$.
			Let $Y_i\subset Y$ be the image of $D_i$. Each $D_i$ is equipped with a
			multiplicative Chow--K\"unneth decomposition which is compatible with that of
			$X$. (This means that  the inclusion morphisms are of pure grade 0.)
			
			\item All components of $D_2\times_{Y_2} D_2$ are of dimension $d$ and viewed as
			cycles on $X\times X$, they are all in $\CH^d(X\times X)_0$. 
			
			\item $Y_1$ is smooth and all the Chern classes of
			$(p|_{D_1})^*\mathscr{T}_{Y_1}$ are in $\CH^*(D_1)_0$\,; the Chern classes of
			$p^*\mathscr{T}_Y$ sit in $\CH^*(X)_0$.
			
			\item $Y_1$ is disjoint from $p(D_2\backslash D_1)$.
			
			\item The morphism $p|_{D_1}:D_1\rightarrow Y_1$ is smooth and $D_1\times_{Y_1}
			D_1$, viewed as a cycle on $D_1\times D_1$, is in $\CH^*(D_1\times D_1)_0$.
		\end{enumerate}
		Then ${}^t\Gamma_p\circ\Gamma_p$ sits in $\CH^d(X\times X)_0$. Furthermore, $Y$
		admits a multiplicative Chow--K\"unneth decomposition such that the Chern
		classes of $Y$ sit in $\CH^*(Y)_0$.
	\end{lem}
	
	\begin{proof}
		Note that we have the following fiber square
		\[
		\xymatrix{
			Z\ar[r] \ar[d] & X\times X\ar[d]^g\\
			X\times X \ar[r]^f & X\times Y\times X
		}
		\]
		where $f(x_1,x_2) = (x_1,p(x_1),x_2)$ and $g(x_1,x_2)=(x_1,p(x_2),x_2)$. By
		definition ${}^t\Gamma_p\circ\Gamma_p$ is supported on $Z$. A component of $Z$
		is called \textit{dominating} if it dominates $X$ via any of the two
		projections. A general point on a dominating component is of the form
		$(x_1,x_2)$ such that
		$p(x_1)=p(x_2)$ is a general point on $Y$. The assumption (a) implies that $x_1$
		and $x_2$ are in the same orbit. Hence any dominating component of $Z$ must
		be the graph of an element of $G$. Since the multiplicative Chow--K\"unneth
		decomposition on $X$ is $G$-invariant, we see that a dominating component of
		$Z$ is in $\CH^d(X\times X)_0$. The components of $Z$ that are contained in
		$D_2\times_{Y_2} D_2$ are taken care of by assumption (c). By (e), all the
		remaining components of $Z$ are contained in $D_1\times_{Y_1} D_1$. We use the
		excess intersection formula \cite[\S 6.3, Theorem 6.3]{fulton} to deal with the
		remaining components. To do that let us fix some notations. Let $\mathscr{N}_p$
		be the normal bundle of the graph of $p$ inside $X\times Y$. We then have the
		following short exact sequence of sheaves on $X$.
		\[
		\xymatrix{
			0\ar[r] & \mathscr{T}_X\ar[r]^{(\mathrm{id},\mathrm{d}p)\qquad} &
			\mathscr{T}_X\oplus p^*\mathscr{T}_Y\ar[r] &\mathscr{N}_p\ar[r] &0.
		}
		\]
		It follows that $c(\mathscr{N}_p)=c(p^*\mathscr{T}_Y)$. Hence, by (d),
		$$
		c(\mathscr{N}_p)\in\CH^*(X)_0.
		$$ 
		We need to understand the normal bundle $\mathscr{N}_1$ of $D_1\times_{Y_1} D_1$
		inside $X\times X$. Consider the successive inclusion
		\[
		\xymatrix{
			D_1\times_{Y_1} D_1 \ar[r]^{j_1} &D_1\times D_1\ar[r]^{j_2} &X\times X.
		}
		\]
		The assumption (f) implies that the normal bundle associated to the first
		inclusion is $f^*\mathscr{T}_{Y_1}$, where $f: D_1\times_{Y_1} D_1 \rightarrow
		Y_1$ is the natural morphism\,; the normal bundle associated to the second
		inclusion is $p_1^*\mathscr{N}_{D_1/X} \oplus p_2^*\mathscr{N}_{D_1/X}$, where
		$p_1$ and $p_2$ are the two projections of $D_1\times D_1$. Hence
		$\mathscr{N}_1$ fits into the following short exact sequence
		\begin{equation}\label{eq normal ses2}
		\xymatrix{
			0\ar[r] &f^*\mathscr{T}_{Y_1}\ar[r] &\mathscr{N}_1\ar[r] &
			p_1^*\mathscr{N}_{D_1/X} \oplus p_2^*\mathscr{N}_{D_1/X}\ar[r] &0,
		}
		\end{equation}
		where the last term is viewed as its restriction to $D_1\times_{Y_1} D_1$. By
		\cite[\S 6.3, Theorem 6.3]{fulton}, the contribution to
		${}^t\Gamma_p\circ\Gamma_p$ which comes from the component $D_1\times_{Y_1} D_1$
		of $Z$ is given by
		\[
		j_*c_{\mathrm{top}}(j^*q_1^*\mathscr{N}_p/\mathscr{N}_1).
		\]
		Here, $j=j_2\circ j_1: D_1\times_{Y_1} D_1 \rightarrow X\times X$ is the
		inclusion and $q_1:X\times X\rightarrow X$ is the projection to the first
		factor. To show that $j_*c_{\mathrm{top}}(j^*q_1^*\mathscr{N}_p/\mathscr{N}_1)$
		lies in $\CH^*(X\times X)_0$, it suffices to check that 
		\[
		j_* c(\mathscr{N}_1)\in\CH^*(X\times X)_0.
		\]
		Thanks to the short exact sequence \eqref{eq normal ses2}, this is further
		reduced to showing that
		\[
		j_* c(f^*\mathscr{T}_{Y_1})\in\CH^*(X\times X)_0.
		\]
		Note that $f^*\mathscr{T}_{Y_1}=j_1^* p_1^*(p|_{D_1})^*\mathscr{T}_{Y_1}$ and
		hence
		\[
		j_* c(f^*\mathscr{T}_{Y_1}) = j_{2,*} j_{1,*} j_1^*
		p_1^*c((p|_{D_1})^*\mathscr{T}_{Y_1}) = j_{2,*}\Big( [D_1\times_{Y_1} D_1]\cdot 
		p_1^*c((p|_{D_1})^*\mathscr{T}_{Y_1}) \Big)
		\]
		It follows from the first part of (d), the second part of (f) and the
		compatibility of $j_{2,*}$ with the grading of the Chow rings that $j_*
		c(f^*\mathscr{T}_{Y_1})\in\CH^*(X\times X)_0$.
		
		The conclusion concerning the Chern classes of $Y$ being in the graded-0 part
		follows immediately from the second part of (d) and the explicit Chow--K\"unneth
		decomposition on $Y$ as is given by Proposition \ref{prop self-dual CK generic
			finite}.
	\end{proof}

	We are finally in a position to prove Theorem \ref{thm main}, which we
	restate here for the convenience of the reader.
	
	\begin{thm} [Theorem \ref{thm main}]
		Let $X$ be a smooth projective variety that admits a self-dual multiplicative
		Chow--K\"unneth decomposition $\{\pi^i_X\}$. Let us denote
		as always $\CH^i(X)_s := (\pi^{2i-s}_X)_*\CH^i(X)$. Assume that the
		Chern classes $c_p(X)$ of $X$ belong to $\CH^p(X)_0$. Then the
		Hilbert cube $X^{[3]}$ admits a  self-dual multiplicative
		Chow--K\"unneth decomposition, with the property that the Chern classes
		$c_p(X^{[3]})$ sit in the degree-zero graded pieces $\CH^p(X^{[3]})_0$.
	\end{thm}
	\begin{proof}
		In Proposition \ref{prop desing}, we obtain a generically finite quotient
		morphism $p: X_3\rightarrow X^{[3]}$ by the symmetric group $\mathfrak{S}_3$. By
		Proposition \ref{prop X3}, $X_3$ has a
		$\mathfrak{S}_3$-invariant self-dual multiplicative Chow--K\"unneth
		decomposition. It then follows from Proposition \ref{prop multCK generic finite}
		that $X^{[3]}$ is endowed with a self-dual multiplicative Chow--K\"unneth
		decomposition as long as ${}^t\Gamma_p\circ \Gamma_p$ belongs to
		$\CH^d(X_3\times X_3)_0$.
		To achieve this, we recall from Proposition \ref{prop summary} that the morphism
		$p$ contracts a divisor $E''_1\subset X_3$ to $B_3\subset X^{[3]}$ and another
		divisor $E_W$ to $B_2\subset X^{[3]}$. Here $B_2$ and $B_3$ form the
		stratification of Section \ref{sec stratification}. $E_W$ is  the exceptional
		divisor of the blow-up $X_3\rightarrow
		X_2$, and $E''_1$ is the strict transform of $E'_1 \subset X_2$; see Section
		\ref{sec desing} for the notation. We would like to apply Lemma \ref{lem
			technical} to $p:X_3\rightarrow
		X^{[3]}$ with $D_1=E''_1$, $D_2=E_W$, $Y_1=B_3$ and $Y_2=B_2$. 
		We need to verify
		all the assumptions of Lemma \ref{lem technical}. For simplicity, we write
		$Y=X^{[3]}$.
		
		\emph{Assumption (a)}. This is immediate from the construction of $p$ as the
		desingularization of $X^3\dashrightarrow X^{(3)}$.
		
		\emph{Assumption (b)}. We need to see that $D_1=E''_1$ and $D_2=E_W$ have
		multiplicative Chow--K\"unneth decompositions that are compatible with that of
		$X_3$. The case of $D_2$ follows from the fact that it is the exceptional
		divisor of the last blow-up $X_3\rightarrow X_2$\,; see Remark \ref{rmk blow-up
			diagram degree 0}.  The case of $D_1$ can be done
		as follows. By the discussion following Lemma \ref{lem S_1 admissible}, we have
		a complete admissible set $\mathcal{S}_2$ of subvarieties of $X_2$.  Hence we
		get
		that
		\[
		\mathcal{S}_3=\mathrm{Bl}_{Y_2}(\mathcal{S}_2)=\{E''_1,X_3\}
		\] 
		is admissible (see Proposition \ref{prop admissible blow-up}) and, in
		particular, the inclusion $D_1\hookrightarrow X_2$ is of pure grade 0.
		
		\emph{Assumption (c)}. We have the isomorphism $D_2\cong \PP^1\times
		\PP(\mathscr{E})$\,;
		see \eqref{eq isom E_W}. Under this isomorphism, the grading of the Chow ring of
		$D_2$ and $D_2\times D_2$ are given by 
		\begin{align}
		\CH^*(D_2)_s & = \sum_{a,b} \xi^a\zeta^{b}\cdot \CH^*(X)_s \, ; \nonumber \\
		\CH^*(D_2\times D_2)_s & = \sum_{a_1,a_2,b_1,b_2}
		\xi_1^{a_1}\xi_2^{a_2}{\zeta}_1^{b_1}{\zeta}_2^{b_2}\cdot \CH^*(X\times X)_s,
		\label{eq grading D2 D2}
		\end{align}
		where $\xi$ is the Chern class of the $\calO(1)$-bundle on $\PP(\mathscr{E})$
		and $\zeta$ is the Chern class of the $\calO(1)$-bundle on $\PP(\mathscr{T}_X)$
		pulled back to $\PP(\mathscr{E})$.
		The morphism $D_2\rightarrow Y_2$ can be decomposed as the projection
		$\PP^1\times \PP(\mathscr{E})\longrightarrow \PP(\mathscr{E})$ followed by a
		morphism $\PP(\mathscr{E})\rightarrow B_2$ contracting
		$\PP(\mathscr{T}_{\PP(\mathscr{T}_X)/X})$ to $B_3$. It follows that
		$D_2\times_{Y_2} D_2$ consists of two components.
		Then one component of $D_2\times _{Y_2} D_2$ is $\PP^1\times \PP^1\times
		\PP(\mathscr{E})$ which lies in $\CH_{3d}(D_2\times D_2)_0$. The second
		component  of $D_2\times _{Y_2} D_2$  sits above
		$Y_1=B_3\cong\mathrm{G}(2,\mathscr{T}_X)$, which is easily seen to be
		$\PP^1\times \PP^1 \times \PP(\mathscr{T}_{\PP(\mathscr{T}_X)/X})\times_{B_3}
		\PP(\mathscr{T}_{\PP(\mathscr{T}_X)/X})$. We note that there is a canonical
		isomorphism
		\[
		\PP(\mathscr{T}_{\PP(\mathscr{T}_X)/X})\cong \mathrm{G}(1,2,\mathscr{T}_X)
		\]
		where $\mathrm{G}(1,2,\mathscr{T}_X)$ is the flag variety bundle over $X$. Under
		this isomorphism, the morphism
		$\PP(\mathscr{T}_{\PP(\mathscr{T}_X)/X})\rightarrow B_3$ becomes the natural
		projection $\mathrm{G}(1,2,\mathscr{T}_X)\rightarrow
		\mathrm{G}(2,\mathscr{T}_X)$. Then the second component becomes
		\[
		\PP^1 \times \PP^1 \times \mathrm{G}(1,2,\mathscr{T}_X)
		\times_{\mathrm{G}(2,\mathscr{T}_X)} \mathrm{G}(1,2,\mathscr{T}_X)
		\]
		which can again been directly checked to be in $\CH_{3d}(D_2\times D_2)_0$,
		thanks to the equation \eqref{eq grading D2 D2}. Since the inclusion
		$D_2\hookrightarrow X_3$ respects the gradings, we see that all components of
		$D_2\times_{Y_2} D_2$ are contained in $\CH_{3d}(X_3\times X_3)_0$.
		
		\emph{Assumption (e)}. We have $p(Y_2\backslash Y_1) = B_2\backslash B_3$.
		
		\emph{Assumptions (d) and (f)}. We already saw that $Y_1\cong
		\mathrm{G}(2,\mathscr{T}_X)$ and hence $Y_1$ is smooth. Now we give an explicit
		description of the morphism $D_1=E''_1\rightarrow
		Y_1=\mathrm{G}(2,\mathscr{T}_X)$. Consider the following diagram
		\begin{equation}\label{eq CD E''_1}
		\xymatrix{
			&E''_1\ar[ld]_{\rho_2}\ar[rd]^{\rho'_1} & \\
			E'_1\ar[d]_{\rho_1} & &\PP(\mathscr{E}_2\oplus
			\mathscr{E}_2)\ar[d]^{\pi''}\ar[lld]_{\rho'_2}\\
			E_1=\PP(\mathscr{T}_X\oplus \mathscr{T}_X)\ar[rd]_{\pi_2} & &
			\mathrm{G}(2,\mathscr{T}_X)=B_3\ar[ld]^{\pi'_2}\\
			&X &
		}
		\end{equation}
		Here $\mathscr{E}_2$ is the natural rank-2 subbundle of
		$\mathscr{T}_X|_{\mathrm{G}(2,\mathscr{T}_X)}$ on $\mathrm{G}(2,\mathscr{T}_X)$.
		The morphism $\rho'_2$ is the blow-up of $E_1$ along $W$ and $\rho'_1$ is the
		blow-up of $\PP(\mathscr{E}_2\oplus\mathscr{E}_2)$ along the union of
		$\tilde{W}_{ij}$, where
		\begin{align*}
		\tilde{W}_{12} & =\PP(\mathscr{E}_2)\hookrightarrow \PP(\mathscr{E}_2\oplus
		\mathscr{E}_2),\quad [v]\mapsto [0,v]\,;\\
		\tilde{W}_{23} & =\PP(\mathscr{E}_2)\hookrightarrow \PP(\mathscr{E}_2\oplus
		\mathscr{E}_2),\quad [v]\mapsto [v,0]\,;\\
		\tilde{W}_{13} & =\PP(\mathscr{E}_2)\hookrightarrow \PP(\mathscr{E}_2\oplus
		\mathscr{E}_2),\quad [v]\mapsto [v,-v].
		\end{align*}
		Furthermore, we have $p|_{E''_1}=\pi''\circ \rho'_1$. One uses this description
		and Lemma \ref{lem blow-up tangent} to study the relative tangent bundle
		$\mathscr{T}_{E''_1/B_3}$ (and hence also $(p|_{E''_1})^*\mathscr{T}_{B_3}$) and
		shows that their Chern classes are in $\CH^*(E''_1)_0$. The smoothness of
		$p|_{E''_1}$ can be checked by showing that the fibers are all smooth. Now we
		want to show that $E''_1\times_{B_3} E''_1$ is in $\CH^*(E''_1\times E''_1)_0$.
		This can be done by a Chern class computation. The multiplicative
		Chow--K\"unneth decomposition on $E''_1$ is constructed from that of $X$ along
		the morphisms $\pi_2$, $\rho_1$ and $\rho_2$. The induced grading on the Chow
		ring of $E''_1$ can be explicitly described using the canonical cycles appearing
		in each step. By a ``change of basis" computation, those cycles can be replaced
		by those appearing in the construction of $\pi'_2$, $\pi''$ and $\rho'_1$. More
		precisely, we have
		\begin{equation}\label{eq grading E1 E1}
		\CH^*(E''_1\times E''_1)_s = \CH^*(X\times X)_s[c_{\cdot}(\mathscr{E}_2)_1,
		c_{\cdot}(\mathscr{E}_2)_2, \xi'_1, \xi'_2, \tau_{ij,1}, \tau_{ij,2}]
		\end{equation}
		where the right-hand side means polynomials in the given variables with
		coefficients
		in $\CH^*(X\times S)_s$. Here $\xi'$ is the $\calO(1)$-class on
		$\PP(\mathscr{E}_2\oplus \mathscr{E}_2)$ and $\tau_{ij}$ are the classes of the
		exceptional divisors of the blow-up $\rho'_1$\,; the subscripts ``1" and ``2"
		indicate which factor they come from. Let $f:E''_1\rightarrow X$ be the natural
		morphism, then 
		\[
		E''_1\times_{B_3} E''_1\subset (f\times f)^{-1}(\Delta_X)
		\]
		is the vanishing locus of the morphism
		\[
		p_1^*\mathscr{E}_2 \rightarrow \frac{\mathscr{T}_X}{p_2^*\mathscr{E}_2}.
		\]
		Given the formula \eqref{eq grading E1 E1}, a Chern class computation shows that
		$E''_1\times_{B_3} E''_1\in \CH^*(E''_1\times E''_1)_0$.
		
		Now it only remains to show that the Chern classes of $p^*\mathscr{T}_Y$ are
		contained in $\CH^*(X)_0$. We consider the short exact sequence
		\[
		\xymatrix{
			0\ar[r] & p^*\Omega_{Y}^1\ar[r] &\Omega_{X_3}^1\ar[r] &\Omega_{X_3/Y}^1\ar[r]
			&0.
		}
		\]
		Since $c(\mathscr{T}_X)\in\CH^*(X_3)_0$, we only need to show that the Chern
		classes of $\Omega^1_{X_3/Y}$ are in $\CH^*(X_3)_0$. Note that
		$\Omega^1_{X_3/Y}$ restricted to $\tilde{E}_{ij}$ is simply the conormal bundle
		of $\tilde{E}_{ij}$ in $X_3$. From the short exact sequence
		\[
		\xymatrix{
			0\ar[r] & \Omega^1_{X_3/Y}|_{E''_1\cup E_W}\otimes\calO(-\sum \tilde{E}_{ij})
			\ar[rr] && \Omega^1_{X_3/Y}\ar[rr] &&\Omega^1_{X_3/Y}|_{\cup
				\tilde{E}_{ij}}\ar[r] &0
		}
		\]
		we see that it suffices to show that the Chern classes of $
		\Omega^1_{X_3/Y}|_{E''_1\cup E_W}$ pushes forward into $\CH^*(X_3)_0$. Then we
		can again use a similar short exact sequence
		\[
		\xymatrix{
			0\ar[r] &\Omega^1_{X_3/Y}|_{E_W}\otimes \calO_{E_W}(-E''_1\cap E_W)\ar[rr] &&
			\Omega^1_{X_3/Y} |_{E''_1\cup E_W}\ar[rr] &&\Omega^1_{X_3/Y}|_{E''_1}\ar[r] &0
		}
		\]
		to reduce to showing that the Chern classes of $\Omega^1_{X_3/Y}|_{E''_1}$ and
		$\Omega^1_{X_3/Y}|_{E_W}$ are in the graded-0 part. But then this becomes
		immediate
		since these two sheaves are simply $\Omega^1_{E''_1/B_3}$ and
		$\Omega^1_{E_W/B_2}$ respectively.
		
		Thus we verified all the assumptions of Lemma \ref{lem technical} and the main
		theorem follows.
	\end{proof}
	
	\section{Multiplicative Chow--K\"unneth decomposition for $X^{[1,2]}$ and
		$X^{[2,3]}$}
	
	Let $X$ be a projective variety. Recall that,
	for positive integers $n<m$, the nested Hilbert scheme $X^{[n,m]}$ is the
	projective variety consisting of
	$\{(x,y) : x \subset y\} \subset  X^{[n]} \times X^{[m]}$.
	If $X$ is smooth and has dimension $\geq 3$, Cheah \cite{cheah} showed that 
	the
	nested Hilbert scheme $X^{[n,m]}$ is smooth if and only if it is one of
	$X^{[1,2]}$ or $X^{[2,3]}$.
	
	\begin{thm}\label{thm mainnested}
		Let $X$ be a smooth projective variety that admits a self-dual multiplicative
		Chow--K\"unneth decomposition. Assume that the Chern
		classes of $X$ satisfy $c_p(X) \in \CH^p(X)_0$. 
		Then the nested Hilbert
		schemes $X^{[1,2]}$ and $X^{[2,3]}$ also admit a self-dual multiplicative
		Chow--K\"unneth decomposition, with the property that the Chern classes
		$c_p(X^{[1,2]})$ sit in $\CH^p(X^{[1,2]})_0$, \emph{resp.}~$c_p(X^{[2,3]})$ sit
		in $\CH^p(X^{[2,3]})_0$.
	\end{thm}
	
	\begin{proof}
		The nested Hilbert scheme $X^{[1,2]}$ is by definition the same as the
		universal length-$2$ closed subscheme over $X^{[2]}$, which is isomorphic to the
		blow-up of $X\times X$ along the diagonal. The latter was already treated in
		\cite[Theorem 6]{sv}. 
		
		The nested Hilbert scheme $X^{[2,3]}$ is the blow-up of $X\times X^{[2]}$ along
		the universal length-2 subscheme.
		The latter, which we denote $Y$, is isomorphic to $X^{[1,2]}$. By \cite[Theorem 6]{sv} and the above,
		both 
		$X\times X^{[2]}$ and $Y\simeq X^{[1,2]}$ admit a self-dual multiplicative
		Chow--K\"unneth decomposition such that their Chern classes are of pure grade
		$0$. Therefore, by Proposition \ref{prop multCK blow-up},  we only need to check
		that the set $\{Y,X\times X^{[2]}\}$ is admissible in the sense of Definition
		\ref{def complete admissible}, that is, we only need to check that
		\begin{enumerate}[(i)]
			\item the Chern classes of the normal bundle $\mathscr{N}_{Y/X\times X^{[2]}}$
			sit
			in $\CH^*(Y)_0$\,;
			\item the morphism $i:Y\rightarrow X\times X^{[2]}$ is of pure grade 0.
		\end{enumerate}

		We first show (i).
		As already observed, $Y$ is the blow-up of $X\times X$ along the diagonal and is
		naturally equipped with a self-dual multiplicative Chow--K\"unneth
		decomposition. Let $\rho: Y\rightarrow X\times X$ be the blow-up morphism. Then
		$\rho$ is of pure grade-0 by Proposition \ref{prop multCK blow-up}. Consider the
		following diagram
		\[
		\xymatrix{
			& & \mathscr{T}_X \ar@{=}[r]\ar[d] &\mathscr{T}_X\ar[d] &\\
			0\ar[r] &\mathscr{T}_Y\ar[r]\ar[d] &\mathscr{T}_X \oplus \mathscr{T}_{X^{[2]}}
			\ar[r]\ar[d] &\mathscr{N}_{Y/X\times X^{[2]}}\ar[r] &0\\
			&\mathscr{T}_{X^{[2]}}\ar@{=}[r]\ar[d] &\mathscr{T}_{X^{[2]}} &&\\
			&\mathscr{Q} & & &
		}
		\]
		where all sheaves are understood to be their pull-back to $Y$ and $\mathscr{Q}$
		is the quotient of $\mathscr{T}_{X^{[2]}}|_Y$ by $\mathscr{T}_Y$. 
		We already know that all the Chern classes of $\mathscr{Q}$ are in the graded-0
		part $\CH^*(Y)$.
		Note that every local section $v$ of $\mathscr{N}_{Y/X\times X^{[2]}}$ lifts to
		$(v_1,v_2)\in \mathscr{T}_{X^{[2]}}\oplus \mathscr{T}_X$, and that the image
		$\bar{v}_2\in\mathscr{Q}$ is independent of the choice of the lifting
		$(v_1,v_2)$. This gives a homomorphism $\mathscr{N}_{Y/X\times X^{[2]}}
		\longrightarrow \mathscr{Q}$, which fits into the following short exact sequence
		\[
		\xymatrix{
			0\ar[r] &\mathscr{T}_X|_Y\ar[r] &\mathscr{N}_{Y/X\times X^{[2]}}\ar[r]
			&\mathscr{Q}\ar[r] &0.
		}
		\]
		By assumption, the Chern classes of $\mathscr{T}_X$ are in $\CH^*(X)_0$.
		Moreover, the natural morphism $Y \hookrightarrow X\times X^{[2]}
		\twoheadrightarrow X$ is the composition of $\rho$ with a projection $X\times
		X\rightarrow X$, the two of which are of pure grade-0. It follows that the Chern
		classes of $\mathscr{T}_X|_Y$ sit in $\CH^*(Y)_0$. The multiplicativity of
		$\CH^*(Y)_0$ then implies that the Chern classes of $\mathscr{N}_{Y/X\times
			X^{[2]}}$ are in the graded-0 part.
		
		Now we show (ii). To show that the graph $\Gamma_i \subset Y\times X\times
		X^{[2]}$ is of pure grade-0, we only need to show that its pull back $\Gamma'_i
		\subset Y\times X \times Y$ is of pure grade-0. It is clear that $\Gamma'_i$
		has two irreducible components
		\[
		((x,y),x,(x,y))\qquad \text{and}\qquad ((y,x),x,(x,y)).
		\]
		Then it is clear that both components are of grade-0.	
	\end{proof}

	\begin{rmk}
		There is a natural generically 3-to-1 morphism $X^{[2,3]} \rightarrow X^{[3]}$.
		With the notations of \S \ref{sec stratification}, this morphism is finite
		\'etale over $X^{[3]} \backslash B_2$, totally ramified along $B_2\backslash
		B_3$, and the fiber over a point of $B_3$ is $\PP^1$. It is likely that a
		multiplicative Chow--K\"unneth decomposition for $X^{[3]}$  could also be obtained  by proving an
		analogue of the technical descent Lemma \ref{lem technical} so that Proposition \ref{prop multCK generic finite} applies to the generically finite morphism $X^{[2,3]}\rightarrow X^{[3]}$. 
	\end{rmk}

\end{document}